\documentclass[letterpaper,11pt]{article}
\usepackage{amsmath}
\usepackage{amsthm}
\usepackage{amsfonts}
\newcommand{\BigO}[1]{\ensuremath{\operatorname{O}\bigl(#1\bigr)}}
\usepackage{algorithm}
\usepackage{algpseudocode}
\usepackage{comment}
\usepackage{amsthm}
\usepackage{algpseudocode}
\usepackage{amssymb}
\usepackage{algorithm}
\usepackage{graphicx}
\usepackage{epstopdf}
\DeclareGraphicsExtensions{.eps}
\usepackage[export]{adjustbox}
\usepackage{color}
\usepackage{float}
\usepackage{pgfplots}
\usepackage[normalem]{ulem}
\usepackage{fullpage}
\usepackage{algpseudocode}
\usepackage{authblk}

\newtheorem{theorem}{Theorem}
\newtheorem{lemma}{Lemma}
\newtheorem{definition}{Definition}

\newtheorem{proposition}{Proposition}
\newtheorem{corollary}{Corollary}
\newtheorem{conjecture}{Conjecture}

\date{}
\makeatletter
\newcommand\footnoteref[1]{\protected@xdef\@thefnmark{\ref{#1}}\@footnotemark}
\makeatother
\newlength\figureheight \newlength\figurewidth
\newcommand{\mytilde}{\raise.17ex\hbox{$\scriptstyle\mathtt{\sim}$}}
\title{Community Detection on  Euclidean Random Graphs \footnote{An extended abstract of certain results appears in the proceedings of ACM-SIAM Symposium on Discrete Algorithms (SODA) 2018 \cite{soda}.}}

\author{%

	Emmanuel Abbe \footnote{Department of Applied and Computational Mathematics and Electrical Engineering, Princeton University. Email - eabbe@princeton.edu},\hspace{4mm}
	Fran\c cois Baccelli, \footnote{Department of Mathematics and ECE, The University of Texas at Austin. Email - baccelli@math.utexas.edu.}\hspace{4mm}
		Abishek Sankararaman\footnote{Department of ECE, The University of Texas at Austin. Email - abishek@utexas.edu.}

}
\providecommand{\keywords}[1]{\textbf{{Keywords---}} #1}

\begin{document}
\maketitle

\begin{abstract}

We study the problem of community detection on Euclidean random geometric graphs where each vertex has two latent variables: a binary community label and a $\mathbb{R}^d$ valued location label which forms the support of a Poisson point process of intensity $\lambda$. A random graph is then drawn with edge probabilities dependent on both the community and location labels. In contrast to the stochastic block model (SBM) that has no location labels, the resulting random graph contains many more short loops due to the geometric embedding. We consider the recovery of the community labels, partial and exact, using the random graph and the location labels. We establish phase transitions for both sparse and logarithmic degree regimes, and provide bounds on the location of the thresholds, conjectured to be tight in the case of exact recovery. We also show that the threshold of the distinguishability problem, i.e., the testing between our model and the null model without community labels exhibits no phase-transition and in particular, does not match the weak recovery threshold (in contrast to the SBM).

\end{abstract}


\section{Introduction}
\label{sec:intro}

 Community Detection, also known as the graph clustering problem, is the task of grouping together nodes of a graph into representative clusters. This problem has several incarnations that have proven to be useful in various applications (\cite{comm_det_nature}) such as social sciences (\cite{holland},\cite{social_science_app}), image segmentation \cite{image_segment}, recommendation systems (\cite{reco_1},\cite{reco_2}), web-page sorting \cite{web_app},  and biology (\cite{bio_application}, \cite{ppi_detect}) to name a few. In the present paper, we introduce a new class of spatial random graphs with  communities and consider the Community Detection problem on it. Our motivation for a new class of random graph model comes from applications where nodes have geometric attributes, such as in social networks or more generally in graphs of similarities, where the similarity function has metric properties.
We study two \emph{regimes} of the random graph - the sparse degree regime  and the logarithmic degree regime. The sparse degree regime is one wherein the average node degree does not scale with the total number of nodes of the graph, while the logarithmic degree regime is one where the average node degree is proportional to the logarithm of the total number of nodes.
\\

 \textbf{Model Overview} - The random graph will be denoted by $G_n$, which has a random  $N_n$ number of nodes which is Poisson distributed with mean $\lambda n$. In our formulation, $\lambda > 0$ is a {fixed} constant that denotes the \emph{intensity} parameter  and $n$ is a scaling parameter, and we will consider the asymptotic as $n \rightarrow \infty$. Nodes 
 are equipped with two i.i.d. labels, a uniform $\{-1,+1\}$ valued \emph{community label} and a uniform $B_n := \left[ -\frac{n^{1/d}}{2},\frac{n^{1/d}}{2} \right]^d$, $d \in \mathbb{N}$ valued \emph{location label}. Therefore, the average number of nodes having location labels in any subset of $B_n$ of unit volume is $\lambda$, which explains why we call $\lambda$ as the intensity parameter. To draw the edges, we consider two sequences of functions $(f_{in}^{(n)}(\cdot))_{n \in \mathbb{N}}$ and $(f_{out}^{(n)}(\cdot))_{n \in \mathbb{N}}$ such that $1 \geq f_{in}^{(n)}(r) \geq f_{out}^{(n)}(r) \geq 0$ for all $r \geq 0$ and $n \in \mathbb{N}$.  Conditional on the node labels, two nodes with location labels $x,y \in B_n$ and community labels $Z_x,Z_y \in \{-1,+1\}$ 
 are connected by an edge in $G_n$ independently of other edges  with probability $f_{in}^{(n)}(||x-y||)$ if $Z_x = Z_y$ and with probability $f_{out}^{(n)}(||x-y||)$ if $Z_x \neq Z_y$. In this interpretation, $||\cdot||$ denotes the Euclidean norm on the set $B_n$ if $G_n$ is sparse, or denotes the toroidal metric on $B_n$ in the non-sparse case. We call the random graph $G_n$ \emph{sparse} if the connection functions $f_{in}^{(n)}(\cdot)$ and $f_{out}^{(n)}(\cdot)$ do not depend on $n$. More precisely, the graph $G_n$ is sparse if for all $r \geq 0$, $f_{in}^{(n)}(r) := f_{in}(r)$ and $f_{out}^{(n)}(r) := f_{out}(r)$, for functions $f_{in}(\cdot)$ and $f_{out}(\cdot)$ satisfying $0 < \int_{x \in \mathbb{R}^d}(f_{in}(||x||) + f_{out}(||x||))dx < \infty$. In this sparse regime, the   average  degree of any node in $G_n$ is bounded above by  $(\lambda/2)\int_{x \in \mathbb{R}^d} (f_{in}(||x||) + f_{out}(||x||))dx < \infty$ \emph{uniformly} in $n$. In this regime, we draw an edge between two nodes $i$ and $j$ with probability $f_{in}(||X_i - X_j||)$ if the community labels $Z_i$ and $Z_j$ are the same or with probability $f_{out}(||X_i - X_j||)$ if the community labels $Z_i$ and $Z_j$ are different, where $||\cdot||$ denotes the Euclidean norm on the set $B_n$.
Furthermore, the `boundary effects' due to the edges of the set $B_n$ will not matter asymptomatically as $n \rightarrow \infty$ as the average degree is uniformly bounded in $n$ (We make this precise in Section \ref{sec:math_model}) \footnote{We introduce the toroidal metric as it makes the space symmetric, which greatly aids in the proof}. We call this the sparse regime as the average degree is a constant independent of $n$. If the connection functions $f_{in}^{(n)}(\cdot)$ and $f_{out}^{(n)}(\cdot)$ depended on $n$ and further-more satisfy $  \int_{x \in \mathbb{R}^d} f_{in}^{(n)}(||x||) dx = C_{in}\log(n)$ and $\int_{x \in \mathbb{R}^d} f_{out}^{(n)}(||x||) dx = C_{out}\log(n)$ for some $C_{in} > C_{out} \geq 0$, for all $n$, then we call the graph $G_n$ as the \emph{logarithmic degree} regime or simply as the logarithmic regime. This is so since the average degree of any node is proportional to the logarithm of the total number of nodes in $G_n$. In this case of logarithmic regime, we avoid having to deal with boundary effects by considering the toroidal metric on the set $B_n$. Precisely, conditional on the location and community labels on nodes, we place an edge between nodes $i$ and $j$ in $G_n$ with probability $f_{in}^{(n)}(||X_i - X_j||)$ if $Z_i = Z_j$ and with probability $f_{out}^{(n)}(||X_i - X_j||)$ if $Z_i \neq Z_j$, where $||\cdot||$ is the toroidal metric on $B_n$. We provide a more formal description of the random graph process in both regimes in Section \ref{sec:math_model}.
 \\

 \textbf{Problem Statement} - In the sparse regime, we study the problem of `weak-recovery', which asks how and when one can estimate the community labels of the nodes of $G_n$ better than at random, given observational data of locations labels of all nodes and the graph $G_n$. We say weak recovery is solvable in the sparse case (made precise in Section \ref{subsec:comm_det}) if there exists an algorithm which takes as input the graph $G_n$ along with the location labels and produces a partition of the nodes such that the fraction of misclassified nodes is strictly smaller than a half as $n$ goes to infinity, i.e., we asymptotically beat a random guess of the partition. Although this requirement on estimation is very weak, we see through our results that this is indeed the best one can hope for in the sparse graph setting considered here. In the logarithmic regime, we consider the problem of `Strong Recovery' or also known as \emph{exact-recovery}, which asks how and when can one recover the partition of nodes into communities \emph{exactly} based on the observation of the random graph $G_n$ and the location labels of all nodes. As this is a stronger requirement on the the estimator, the graph needs to be \emph{sufficiently dense} in order to perform Exact-Recovery. More precisely, we see that the average node degree must scale logarithmically to the number of nodes to capture the phase transition for exact-recovery. In both of these problems, we assume that the estimator has access to the model parameters $\lambda$ and $f_{in}^{(n)}(\cdot)$ and $f_{out}^{(n)}(\cdot)$. However, we present how one could possibly implement our algorithm in practice, when the connection functions are not known explicitly. From a mathematical perspective, the estimation of the connection functions from data in our spatial setup  is an interesting research question in itself which is beyond the scope of this paper. 
\\

 \textbf{Remark on the Two Different Distance Metrics} - For technical simplicity, we choose to use the Euclidean metric in the case of sparse graphs and the torridal metric in the case of non-sparse graphs. The Euclidean metric in the sparse graph case allows us to couple all the finite graphs as a subgraph of the limit inifinite graph (made precise in Section \ref{sec:math_model}), while the torridal metric in the non-sparse graph case allows us to use the translation invariance of the torus (Appendix \ref{appendix_palm}). 
\\

 \textbf{Main Results and Technical Contributions - }
 Our main results in this paper pertain to fundamental phase-transitions, which dictate how and when we can do Community Detection in the two regimes of interest. In the sparse regime, we show that (in Theorems \ref{thm:main_lb_cd} and \ref{thm:positive_direction_1}) for every $f_{in}(\cdot)$ and $f_{out}(\cdot)$ \footnote{We assume that $f_{in}(\cdot) \neq f_{out}(\cdot)$ upto Lebesgue measure $0$ and $\int_{x \in \mathbb{R}^d}(f_{in}(||x||) - f_{out}(||x||))\mathrm{d}x < \infty$.}, there exists a critical non-trivial $\lambda_c \in (0,\infty)$, such that if $\lambda < \lambda_c$, then no algorithm can estimate the community labels of $G_n$ better than at random, and if $\lambda > \lambda_c$, then there exists an algorithm that can estimate the community labels better than at random.  From a rather straightforward `monotonicity' argument (in Proposition \ref{prop:monotone}), one can establish the existence of $\lambda_c \in [0,\infty]$ such that the community detection problem shifts from being unsolvable to solvable. Our key technical result in this paper is to establish that this phase-transition is \emph{non-trivial}, i.e., $\lambda_c$ is neither $0$ nor $\infty$. Furthermore, in certain special cases of $f_{in}(\cdot)$ and $f_{out}(\cdot)$, we are able to characterize exactly the phase-transition point $\lambda_c$. 
 \\
 

  To establish weak-recovery is solvable for sufficiently high $\lambda$,  we give a new algorithm called {\ttfamily Good-Bad-Grid} abbreviated as {\ttfamily GBG}. The key idea is to observe that for any two nodes that are `near-by', we can classify them correctly with exponentially (in $\lambda$) small probability of error by simply considering their neighborhoods. Compared to the SBM, this is `easier', since the geometric embedding provides common neighbors even in the sparse regime, whereas one needs to go down to neighbors of neighbors of large depth to obtain intersections in the sparse SBM, as done in \cite{abbe_sandon_ch} with the sphere comparison algorithm \footnote{Counting common neighbors is also exploited in the algorithm of \cite{gbm}}. However in contrast to the SBM, this is not sufficient to produce a global clustering since there will be certain pairs incorrectly classified that need to be identified and corrected. We establish this by embedding `consistency checks' into our algorithm to correct some of the misclassified pairs by partitioning the space $B_n$ into `good' and `bad' regions. Hence the name of our algorithm is {\ttfamily Good-Bad-Grid}. To analyze this algorithm, we then couple the partitioning of space with another percolation process to prove that our algorithm will misclassify a fraction strictly smaller than half of the nodes if $\lambda$ is sufficiently high, with an explicit estimate of the constant.  Furthermore, in certain special instances of connection functions $f_{in}(\cdot)$ and $f_{out}(\cdot)$, our lower bound and upper bound match to give a sharp phase-transition and we can characterize the critical $\lambda$ for these cases. Moreover,  in Section \ref{sec:unknown_conn_func} we give a way to implement our algorithm without any knowledge of the model parameters $f_{in}^{(n)}(\cdot),f_{out}^{(n)}(\cdot)$ and $\lambda$ and is purely `data-dependent'. We prove $\lambda_c > 0$ in Theorem \ref{thm:main_lb_cd}, where the technical analysis relies on identifying an \emph{easier} problem than weak-recovery, which we call Information Flow from Infinity. This reduction is similar to that done in the case of classical SBMs \cite{MNS_Prob_Theory}.
  \\
  
  Our next result concerns the \emph{distinguishability} problem, which asks how well one can solve a hypothesis testing problem between our graph and an appropriate null model (a plain random connection model with connection function $(f_{in}(\cdot)+f_{out}(\cdot))/2$) without communities but having the same average degree and distribution for spatial locations. We show that for \emph{all} parameters, we can solve the {distinguishability} problem with success probability $1 - o_n(1)$, even if we cannot learn the partition better than a random guess. We do so by identifying suitable graph `triangle profiles' and showing that they are different in the planted partition model and the null model. Similar ideas have appeared in the context of distinguishing the SBM from an Erd\H{o}s-R\'enyi random graph \cite{MNS_Prob_Theory}. However, we show in this paper, that the associated computations in analyzing this problem are much simpler thanks to the translation invariance of the Euclidean space.
  In our model, we are able to infer the existence of the partition but do not learn anything  about it in certain regimes. This is because we can always `see' the   partition `locally' in space, but there is no way to consistently piece together the small partitions in different regions of space into one coherent partition of the graph if $\lambda$ is small. This phenomenon is new and starkly different from what is observed in the classical Erd\H{o}s-R\'enyi based symmetric Stochastic Block Models (SBM) with two communities where the moment one can identify the presence of a partition, one can also recover the partition better than a random guess (\cite{MNS_Prob_Theory}). Moreover, such phenomena where one can infer the existence of a partition but not identify it better than random are conjectured not to occur in the SBM even with many communities \cite{decelle},\cite{neeman_distinction}.
\\

In the logarithmic degree regime, we give explicit conditions on the model parameters $\lambda$, $f_{in}^{(n)}(\cdot)$ and $f_{out}^{(n)}(\cdot)$ under which Exact-Recovery is impossible to solve. We establish this by reducing this problem to a Hypothesis testing problem between Poisson random vectors and them employing the Large-Deviations results of  \cite{abbe_sandon_ch} to identify an explicit condition on the model parameters.  On the positive side, we show that a direct adaptation of our {\ttfamily GBG} algorithm to this logarithmic degree regime performs Exact-Recovery if the intensity $\lambda$ is sufficiently large, thereby establishing the phase-transition for Exact-Recovery. However, we conjecture that the algorithm is sub-optimal and that the lower bound identifies a sharp phase-transition. We describe a procedure that may achieve the lower bound, but leave this as an open problem.
\\

\textbf{Central Technical Challenges - } A good estimator must utilize both the spatial data about nodes, as well as the combinatorial information provided by the random graph to perform clustering. This is different from  classical graph clustering algorithms which only look at the edges and any weights on the edges. In our model, the spatial labels provide some form of `side-information' which any estimator must exploit. As an illustrative example to see this, consider the connection functions $f_{in}(\cdot)$ and $f_{out}(\cdot)$  in the sparse case to be of bounded support. In this case, the absence of an edge between two nearby nodes makes it likely that these two nodes belong to opposite communities but  the lack of an edge between far-away nodes farther than the support of either connection functions  does not give any community membership information. Thus the lack of an edge in this example has different interpretations depending on the location labels which needs to be exploited in a principled manner by the estimator. Indeed this is best seen in our lower bound for exact-recovery case, where if $f_{out}^{(n)}(r)$ is identically $0$, i.e. there are no cross community edges, then Exact-Recovery might still be possible even before the subgraph on the nodes of each individual communities become fully connected. This takes place since we can use the spatial location information in a non-trivial fashion to estimate the labels of isolated nodes. Nonetheless, the location labels alone without the graph provide no information on community membership as the community and location labels are independent of each other. 
\\

From a technical perspective, the core challenge in studying our spatial graph model in the sparse regime lies in the fact that it is not `locally tree-like'. The spatial graph is locally dense (i.e., there are  lots of triangles) which arises as a result of the constraints imposed by Euclidean geometry, while it is globally sparse (i.e., the average degree is bounded above by a constant). 
The sparse SBM on the other hand, is locally `tree-like' and has very few short cycles \cite{MNS_Prob_Theory}. This comes from the fact that the connection probability in a sparse SBM scales as $c/n$ for some $c > 0$. In contrast, the connection function in our model in the sparse regime does not scale with $n$. From an algorithmic point of view however, most commonly used techniques (message passing, broadcast process on trees, convex relaxations, spectral methods etc)  are not straight forward to apply to our graph (if one ignores the additional information provided by spatial labels) since their analysis fundamentally relies on the locally tree-like structure of the graph (see \cite{abbe_overview} and references therein). Nevertheless, we show that the presence of spatial labels, enables one to consider a very simple algorithm based on counting common neighbors, to provide an efficient clustering policy, that works even in the sparse graph regime. 
\\

 \textbf{Motivations for a Spatial Model - }
The most widely studied model for Community Detection is the Stochastic Block Model (SBM), which is a multi-type Erd\H{o}s-R\'enyi graph. In the simplest case, the two community symmetric SBM corresponds to a random graph with $n$ nodes, with each node equipped with an i.i.d. uniform community label drawn from $\{-1,+1\}$. Conditionally on the labels, pairs of nodes are connected by an edge independently of other pairs with two different probabilities depending on whether the end points are in the same or different communities. Structurally, the sparse SBM is known to be locally tree-like (\cite{MNS_Prob_Theory},\cite{abbe_overview}) while real social networks   are observed to be \emph{transitive} and sparse. Sparsity in social networks can be understood through `Dunbar's number' \cite{dunbar_base},  which concludes that an average human being can have only about $500$ `relationships' (online and offline) at any point of time. Moreover, this is a fundamental cognitive limitation of the person and not that of access or resources, thereby justifying models where the average node degree is independent of the population size. Social networks are transitive in the sense that any two agents that share a mutual common neighbor tend to have an edge among them as well, i.e., the graph has many  triangles. Similar phenomena also takes place in graphs of similarities, where vertices are connected based on metric similarity functions. These aspects point out the limitations of the sparse SBM and a large collection of models have been proposed to better fit applications under the realm of Latent Space Models (\cite{space_1},\cite{space_2})  and inhomogeneous random graphs (\cite{bollobas}). See also \cite{abbe_overview} for more references. These are sparse \emph{spatial} graphs in which the agents of the social network are assumed to be embedded in an abstract \emph{social space} that is modeled as an Euclidean space and conditional on the embedding,  edges of the graph are drawn independently at random as a non-increasing function of the distance between two nodes. Thanks to the properties of Euclidean geometry, these models are transitive and sparse, and have a better fit to data than any SBM (\cite{space_1}). Such modeling assumptions in the context of multiple communities was also recently verified in parallel independent work \cite{gbm}, where the nodes have both a community label and a location label. The locations labels are sampled uniformly on a sphere and the edges are generated by nodes `nearby' in this sphere connecting with probabilities that depend on the community labels. Several empirical validations of this model on real data is also conducted in \cite{gbm} which suggests that such spatial random graph model provides a good fit for several real world networks. However, we note that the sparse SBM enjoys certain advantages over the geometric random graph considered here, namely that of having low diameter. in agreement with the `small world' phenomena observed in many real world networks (see \cite{milgram}). Therefore a natural next step is to superimpose an SBM with the type of geometric graphs considered here to obtain both a lot of triangles and small diameter, i.e. a type of small world SBM. 
\\

Thus, one can view our model as the simplest planted-partition version of the Latent Space model, where the nodes are distributed uniformly in a large compact set $B_n$ and conditional on the locations, edges are drawn depending on Euclidean distance through connection functions $f_{in}.(\cdot)$ and $f_{out}(\cdot)$. Although, our assumptions are not particularly tailored towards any real data-sets, our setting is the most challenging regime for the estimation problem as the location labels alone without the graph  reveal no community membership information. However,in this paper we assume the location labels on nodes are known exactly to the estimator. In practice, it is likely that the locations labels are unknown (as in the original Latent Space models where the social space is unobservable) or are at-best estimated separately. Nonetheless, our formulation with known location labels forms a crucial first step towards more general models where the location labels are noisy or missing. The problem with known spatial location labels is itself quite challenging as outlined in the sequel  and hence we decided to focus on this setting alone in the present paper. Another drawback of our formulation is that we assume the estimator has knowledge of the model parameters $f_{in}(\cdot)$ and $f_{out}(\cdot)$. In our spatial setup, the estimation of  connection functions from data  is an interesting research question in itself which is however beyond the scope of this paper.
\\


\subsection{Organization of the paper}

We give a formal description of the model and the problem statement in Section \ref{sec:math_model}. We then present our main theorem statements in Section \ref{sec:results}. The subsequent sections will develop the ideas and the proofs needed for our main results. We describe our {\ttfamily GBG} Algorithm in Section \ref{sec:algo} where we first give the idea and then the details of the algorithm. The analysis of our algorithm for the sparse case is performed in Section \ref{sec:analysis_algo}, where the key idea is to construct coupling arguments with site percolations. We establish the lower bound for Community Detection in the sparse case in  Section \ref{sec:lower_bound}, where we first introduce the Information Flow from Infinity problem and then prove that this is easier than Community Detection. Subsequently, we provide a proof of the impossibility result for Information Flow from Infinity. In Section \ref{sec:hypo}, we consider the distinguishability problem and provide a proof of it.  In Section \ref{sec:exact_recovery}, we discuss the non-sparse regime and study the Exact-Recovery problem. Thus, Sections \ref{sec:algo},\ref{sec:lower_bound} and \ref{sec:hypo} all pertain to the sparse graph problem and can be read in any relative order. Section \ref{sec:exact_recovery} exclusively only pertains to the non-sparse case and studies the Exact-Recovery problem. In Section \ref{sec:related_work}, we survey related work and place our model and results in context.

\section{Mathematical Framework and Problem Statement}
\label{sec:math_model}

We describe the mathematical framework based on stationary point processes and state the problem of Community Detection. More precisely, we assume the presence of a single infinite marked Poison Point Process (PPP)\footnote{This is defined in Appendix  \ref{appendix_PPP}}, and consider the random graphs $G_n$ as an appropriate truncation of this infinite object.  In the sparse case, we can construct a single infinite random graph $G$ and consider $G_n$ as an appropriate finite sub-graph of $G$. In the non-sparse case, there is no direct limiting infinite graph. Nevertheless, we can couple all the graphs $G_n$ on a single probability space by constructing them on a single marked PPP. We set a common shorthand notation we use throughout the paper.  For two arbitrary positive sequences $(a_n)_{n \in \mathbb{N}}$ and $(b_n)_{n \in \mathbb{N}}$, we let $b_n = o(a_n)$ to denote the fact that $\lim_{n \rightarrow \infty} b_n/a_n = 0$.


\subsection{The Planted Partition Random Connection Model}
\label{subsec:graph_model}

We suppose there exists an abstract probability space $(\Omega,\mathcal{F},\mathbb{P})$ on which we have an appropriately marked PPP $\bar{\phi}$. We will construct the sequence of random graphs $(G_n)_{n \in \mathbb{N}}$ on this space simultaneously for all $n$ as a measurable function of this marked PPP $\bar{\phi}$. More formally, we assume $\phi$ to be the support of a homogeneous PPP of intensity $\lambda$ on $\mathbb{R}^d$ with the enumeration that $\phi := \{X_1,X_2,\cdots\}$. In this representation, $X_i \in \mathbb{R}^d$ for all $i \in \mathbb{N}$. Furthermore, we assume the enumeration to be such that for all $i > j \in \mathbb{N}$, we have $||X_i||_{\infty} \geq ||X_j||_{\infty}$, i.e., the points are ordered in accordance to increasing $l_{\infty}$ distance. We further mark each atom $i \in \mathbb{N}$ of $\phi$ with random variables  $Z_i \in \{-1,+1\}$ and $\{U_{ij}\}_{j \in \mathbb{N} \setminus \{i\} } \in [0,1]^{\mathbb{N} \setminus \{i\}} $ satisfying $U_{ij} = U_{ji}$ for all $i\neq j \in \mathbb{N}$. We denote by $\bar{\phi}$ to be this marked PPP. The sequence $\{Z_i\}_{i \in \mathbb{N}}$ is i.i.d. with each element being uniformly distributed in $\{-1,+1\}$. For every $i \in \mathbb{N}$, the sequence $\{U_{ij}\}_{j \in \mathbb{N} \setminus \{i\}}$  is i.i.d. with each element of $\{U_{ij}\}_{j \in \mathbb{N} \setminus \{i\}}$ being uniformly distributed on $[0,1]$.
The interpretation of this marked point process is that for any node $i \in \mathbb{N}$, its location label is $X_i$,  community label is $Z_i$ and $\{U_{ij}\}_{j \in \mathbb{N} \setminus \{i\}}$  are used to sample the graph neighbors of node $i$. We describe this construction in both the sparse and logarithmic degree regimes below.
\\
%

 Denote by the set $B_n := \left[ -\frac{n^{1/d}}{2} , \frac{n^{1/d}}{2}\right]^d$, the cube of area $n$ in $\mathbb{R}^d$. For all $n \in \mathbb{N}$, we let $N_n := \sup\{i \geq 1: X_i \in B_n\}$. Since the nodes are enumerated in increasing order of $l_{\infty}$ distance, it follows that for all $i \in [1,N_n]$, $X_i \in B_n$. Furthermore, from basic properties of PPP, $N_n$ is a Poisson random variable of mean $\lambda n$. We construct the graph $G_n$, by assuming its vertex set to be $\{1,\cdots,N_n\}$ and the location label of any node $i \in [1,N_n]$ to be $X_i \in B_n$ and its community label to be $Z_i \in \{-1,+1\}$. However, we use the marks 
 $\{U_{ij}\}_{j \in \mathbb{N} \setminus \{i\}}$  slightly differently depending on whether the graph is sparse or not. Recall that in the sparse regime, we have the connection functions $f_{in}^{(n)}(\cdot)$ and $f_{out}^{(n)}(\cdot)$ to be independent of $n$. In this case, we first construct an infinite graph $G$ with vertex set $\mathbb{N}$ and place an edge between any two nodes $i,j \in \mathbb{N}$  if and only if  $U_{ij} = U_{ji} \leq f_{in}(||X_i - X_j||) \mathbf{1}_{Z_i = Z_j} + f_{out}(||X_i - X_j||)\mathbf{1}_{Z_i \neq Z_j}$. The graph $G_n$ is then the induced subgraph of $G$ consisting of the nodes $1$ through $N_n$, i.e., the subgraph of $G$ restricted to the node set with location labels in $B_n$. In the logarithmic degree regime, we assume that the set $B_n$ is equipped with the toroidal metric rather than the Euclidean metric for simplicity.  Formally, for any $x := (x_1,\cdots,x_d), y:= (y_1,\cdots, y_d) \in B_n$, the toroidal distance on $B_n$ is given by $||x-y||_{\mathcal{T}_n} = || ( \min(|x_1-y_1|, n^{1/d} - |x_1-y_1 |),\cdots, \min(|x_d-y_d|, {n}^{1/d} - |x_d-y_d |)  )||$, where $||.||$ is the standard Euclidean norm on $\mathbb{R}^d$. For any $i,j \in [1,N_n]$, we draw an edge between nodes $i$ and $j$ in $G_n$ if and only if  $U_{ij} = U_{ji} \leq f_{in}^{(n)}(||X_i - X_j||_{\mathcal{T}_n}) \mathbf{1}_{Z_i = Z_j} + f_{out}^{(n)}(||X_i - X_j||_{\mathcal{T}_n})\mathbf{1}_{Z_i \neq Z_j}$. 
\\

The infinite random graph $G$ in the sparse regime can be viewed as a `planted-partition' version of the classical random-connection model (\cite{Meester_Roy}). Given $\lambda \in \mathbb{R}_{+}$, $g(\cdot) : \mathbb{R}_{+} \rightarrow [0,1]$ and $d \geq 1$, the classical random-connection model $H_{\lambda,g(\cdot),d}$, is a random graph whose vertex set forms a homogeneous PPP of intensity $\lambda$ on $\mathbb{R}^d$. Conditionally on the locations,  edges in $H_{\lambda,g(\cdot),d}$ are placed independently of each other where two points at locations $x$ and $y$ of $\mathbb{R}^d$ are connected by an edge in $H_{\lambda,g(\cdot),d}$ with probability $g(||x-y||)$. This construction can be made precise by letting the edge random variables for each node be marks of the PPP similarly to the construction of $G$.

\subsection{The Community Detection Problem}
\label{subsec:comm_det}

In this paper, we study two different notions of Community Detection - weak recovery and exact recovery, depending on whether the graph $G_n$ is sparse or non-sparse. In the sparse regime, our definition of Community Detection is analogous to the notion of `weak-recovery' considered in the classicial SBM literature (\cite{decelle}) and in the logarithmic degree regime, our definition of Community Detection is analogous to the notion of `Exact-Recovery' (\cite{abbe_exact}) in the SBM literature. To state the two notions of community recovery, we set more notation. Let $\phi_n$ be the restriction of the point process $\phi$ to the set $B_n$. Notice that the cardinality of $\phi_n$ is $N_n$ which is distributed as a Poisson random variable of mean $\lambda n$, and $X_i \in B_n$ for all $i \in [1,N_n]$. Moreover, conditionally on $N_n$, the location variables $(X_i)_{i \in [1,N_n]}$ are placed uniformly and independently in $B_n$. Before describing the problem, we need the definition of `overlap' between two sequences.

\begin{definition}
	Given a $t \in \mathbb{N}$, and two sequences $\mathbf{a},\mathbf{b} \in \{-1,1\}^t$, the \emph{overlap} between $\mathbf{a}$ and $\mathbf{b}$  is defined as $\frac{|\sum_{i = 1}^{t} a_ib_i | }{t}$, i.e., the absolute value of the normalized scalar product. 
	\label{def:overlap}
\end{definition}

We define two notions of performance of community detection, weak recovery and exact recovery defined below.

\begin{definition}
	\emph{Weak Recovery} is said to be \textbf{solvable in the sparse regime} for $\lambda,d,f_{in}(\cdot)$ and $f_{out}(\cdot)$ if for every $n \in \mathbb{R}_{+}$, there exists a sequence of $\{-1,+1\}$ valued random variables $\{{\tau}_{i}^{(n)}\}_{i=1}^{N_n}$ which is a {deterministic} function of the observed data $G_n$ {and} $\phi_n$ such that there exists a constant $\gamma > 0$ satisfying
	\begin{align}
	\lim_{n \rightarrow \infty} \mathbb{P} \left[ \mathcal{O}_n \geq \gamma \right] = 1,
	\label{eqn:comm_det_finite}
	\end{align}
	where $\mathcal{O}_n$ is the overlap between $\{\tau_{i}^{(n)}\}_{i=1}^{N_n}$ and $\{Z_i\}_{i = 1}^{N_n}$. 

	\label{defn:comm_det}
\end{definition}
In the above definition, we let the overlap $\mathcal{O}_n := 1$ if $N_n = 0$. In the logarithmic degree regime, we ask for exact-recovery which is formally stated as follows.
\begin{definition}
	\emph{Exact-Recovery} is said to be \textbf{solvable in the logarithmic degree regime} for $\lambda,d,(f_{in}^{(n)}(\cdot))_{n \in \mathbb{N}}$ and $(f_{out}^{(n)}(\cdot))_{n \in \mathbb{N}}$ if for every $n \in \mathbb{R}_{+}$, there exists a sequence of $\{-1,+1\}$ valued random variables $\{{\tau}_{i}^{(n)}\}_{i=1}^{N_n}$ which is a {deterministic} function of the observed data $G_n$ {and} $\phi_n$ such that 
	\begin{align}
	\lim_{n \rightarrow \infty} \mathbb{P} \left[ \mathcal{O}_n = 1\right] = 1,
	\label{eqn:comm_det_er}
	\end{align}
	where $\mathcal{O}_n$ is the overlap between $\{\tau_{i}^{(n)}\}_{i=1}^{N_n}$ and $\{Z_i\}_{i = 1}^{N_n}$. 
	
	\label{defn:comm_det_er}
\end{definition}

A key new feature of our Definitions  \ref{defn:comm_det} and \ref{defn:comm_det_er} comes from our assumption that the algorithm has knowledge of all location labels on the nodes and it only needs to estimate the missing community labels. In the sparse regime, we ask when can any estimator assign community labels to the nodes that beats a `random guess'. Observe that if an estimator guessed every node to be in Community $+1$, then the achieved overlap $\mathcal{O}_n$ converges almost-surely to $0$, thanks to the Strong Law of Large Numbers. Thus, achieving a positive $\gamma$ asks whether an estimator can asymptotically beat the trivial estimator. In the non-sparse regime, we ask when and how can we  recover the commuity label of \emph{all} nodes, upto a global flip. 
\\


Observe that we assume the algorithm has access to the parameters $f_{in}^{(n)}(\cdot)$, $f_{out}^{(n)}(\cdot)$ and $\lambda$ although this assumption may not always hold in practice. As mentioned, the estimation of model parameters from data itself will form an interesting technical question which we leave for future work. We take an absolute value in the definition of  overlap since  the distribution of $G_n$ is symmetric in the community labels. In particular, if we flipped all community labels of $G_n$, we would observe a graph which is equal in distribution to $G_n$. Thus, any algorithm can produce the clustering only up-to a global sign flip, which we capture by considering the absolute value. We take finite restrictions $B_n$ since the overlap is not well defined if $N_n = \infty$. A natural question then is of `boundary-effects', i.e. the nodes near the boundary of $B_n$ will have different statistics for neighbors than those far away from the boundary. However since  $G_n$ is sparse, except for a $o_n(1)$ fraction of nodes, all nodes in $G_n$ will have the same degree as in the infinite graph $G$, i.e. the boundary effects are negligible. In the non-sparse case, since we consider the set $B_n$ as a torus, to precisely avoid the technicalities arising out of considering the edge effects.
\\
 
 The following elementary monotonicity property is evident from the definition of the problem and sets the stage for stating our main results.
 \begin{proposition}
 	For every $f_{in}(\cdot), f_{out}(\cdot)$ and $d$, there exists a $ \lambda_c \in [0,\infty]$ such that
 	\begin{itemize}
 		\item  $\lambda < \lambda_{c} \implies$  weak-recovery is not solvable.
 		
 		\item  $\lambda > \lambda_c \implies$ there exists an algorithm (which could possibly take exponential time) to solve weak-recovery.
 	\end{itemize}
 	\label{prop:monotone}
 \end{proposition}

The proof is deferred to Appendix \ref{appendix-monotone}. This proposition is not that strong since it does not rule out the fact that $\lambda_c$ is either $0$ or infinity. Moreover, this proposition does not tell us anything about polynomial time algorithms, just of the existence or non-existence of any (polynomial or exponential time) algorithms. The first non-trivial result would be to establish that $0 < \lambda_{c}   < \infty$, i.e. the phase transition is strictly non-trivial and then to show that for possibly a larger constant, the problem is solvable efficiently. We establish both of this in Section \ref{sec:results}.

\subsection*{Distinguishability of the Planted Partition in the Sparse Regime}

In this paper, we also consider a related problem to weak-recovery, namely the \emph{distinguishability} question in the sparse regime. Roughly speaking, this problem asks whether one can identify if the observed data $(G_n,\phi_n)$ has communities or not. More precisely, we have the following definition.
\begin{definition}
	The sparse planted partition model with parameters $\lambda$, $d$ and connection functions $f_{in}(\cdot)$ and $f_{out}(\cdot)$ is said to be \emph{distinguishable}, if for every $g(\cdot):\mathbb{R}_{+} \rightarrow [0,1]$, one can identify with probability at-least $\frac{1}{2} + \gamma - o_n(1)$ for some $\gamma > 0$, whether the observed data $(G_n,\phi_n)$ is a sample of the planted partition with the above parameters, or is a sample of the random connection model $H_{\lambda,g(\cdot),d}$, given an uniform prior over the two models.
	\label{defn:distinguishability}
\end{definition}

This problem asks if we can even identify the existence of communities, before trying to identify them. Such hypothesis testing questions are of critical importance in practice where one needs to be reasonably sure of the presence of communities in a given graph, before attempting to cluster the graph. Our main result in Theorem \ref{thm:identifiability} is that one can always solve the distinguishability problem, i.e. it exhibits no phase-transition. Thus our results on weak-recovery and distinguishability predicts regimes of the problem (i.e. $d=1$ and $\lambda < \lambda_c$ for $d \geq 2$), where we can be very sure of the presence of communities, but cannot recover it better than at random.

 \subsection*{Notation - Palm Probability}
 Before stating the results, we will need an important definition. We  define the Palm Probability measure $\mathbb{P}^0$ of a point process in this subsection for future reference. Roughly stated, the Palm measure is the distribution of a marked point process $\phi$ `seen from a typical atom'. We refer the reader to \cite{daley} for the general theory of Palm measures. However thanks to Slivnyak's theorem \cite{daley}, we have a simple interpretation of $\mathbb{P}^0$ which is what we will use. The measure $\mathbb{P}^0$ is obtained by first sampling $\phi$ and $G$ from $\mathbb{P}$ and placing an additional node indexed $0$ at the origin of $\mathbb{R}^d$ and equipping it with independent community label and edges.  The label of this node at origin will be denoted by  $Z_0 \in \{-1,+1\}$ which is uniform and independent of anything else. This extra node is also equipped as a mark the sequence $\{U_{0j}\}_{j \geq 1}$, which will be used to sample the edges as before. Informally, we place an edge between this extra point at $0$ and any other node $i \in \mathbb{N}$ in the graph $G_n$ with probability $f_{in}^{(n)}(||X_i||)$ if $Z_i = Z_0$, or with probability $f_{out}^{(n)}(||X_i||)$ if $Z_i \neq Z_0$, independently of everything else. More generally, for any $k \in \mathbb{N}$ and $x_1,\cdots x_k \in \mathbb{R}^d$, we denote by the measure $\mathbb{P}^{x_1,\cdots,x_k}$ to be the union of $\bar{\phi}$ from before with additional points placed at $x_1,\cdots x_k$, with them having independent marks as before. Informally, the community labels of the extra nodes located at $x_1,\cdots,x_k$ is uniform and independent, and will have edges among themselves and the points of $\phi$ with the same distribution as before. However, for the logarithmic degree regime discussed in Section \ref{sec:exact_recovery}, since the set $B_n$ is equipped with the toroidal metric, we need a definition of Palm measure on the torus rather than Euclidean plane. We give the necessary adaptation of Palm measure needed to handle the torus case in the Appendix \ref{appendix_palm}.

\section{Main Results}
\label{sec:results}

\subsection{Lower Bound for Weak Recovery}

 To state the main lower bound result in Theorem \ref{thm:main_lb_cd} below, we set some notation. For the random connection model graph $H_{\lambda,g(\cdot),d}$, denote by $\mathcal{C}_{H_{\lambda,g(\cdot),d}}(0)$  the set of nodes of $H_{\lambda,g(\cdot),d}$ that are in the same connected component as that of the node at the origin under the measure $\mathbb{P}^0$. Denote by $\theta(H_{\lambda,g(\cdot),d}) := \mathbb{P}^{0}[| \mathcal{C}_{H_{\lambda,g(\cdot),d}}(0)| = \infty] $  the {percolation probability} of the random graph $H_{\lambda,g(\cdot),d}$, i.e. the probability (under Palm) that the connected component of the origin has infinite cardinality.  
 
 \begin{theorem}
 	If $\theta(H_{\lambda,f_{in}(\cdot) - f_{out}(\cdot), d}) = 0$, then weak-recovery is not solvable.
 	\label{thm:main_lb_cd}
 \end{theorem}
This theorem states that if the two functions $f_{in}(\cdot)$ and $f_{out}(\cdot)$ are not `sufficiently far-apart', then no algorithm to detect the partition of nodes can beat a random guess. As a corollary, this says that Community Detection is impossible for $d=1$.
\begin{corollary}
	For all $\lambda >0$, $f_{in}(\cdot),f_{out}(\cdot)$ such that $\int_{x \in \mathbb{R}} f_{in}(||x||)dx < \infty$, weak-recovery is not solvable if $d=1$.
	\label{cor:one_dim}
\end{corollary}
\begin{proof}
	This is based on the classical fact that for all $g(\cdot) : \mathbb{R}_{+} \rightarrow \mathbb{R}_{+}$ such that $\int_{x \in \mathbb{R}} g(||x||)dx < \infty$, $\theta(H_{\lambda,g(\cdot),1}) = 0$.
\end{proof}

 The following corollary  gives a quantitative estimate of the percolation probability for higher dimensions in terms of the problem parameters.
 
 \begin{corollary}
 	For all $d \geq 2$, if $\lambda \leq \lambda_{lower}:= (\int_{x \in \mathbb{R}^d} (f_{in}(||x||) - f_{out}(||x||))dx)^{-1} $, then  weak-recovery cannot be solved. Thus, $ \lambda_c > (\int_{x \in \mathbb{R}^d} (f_{in}(||x||) - f_{out}(||x||))dx)^{-1}$.
 	\label{cor:lambda_lower}
 \end{corollary}
 \begin{proof}
 	From classical results on percolation \cite{Meester_Roy}, by comparison with a branching process, we see that $\lambda \int_{x \in \mathbb{R}^{d}}g(||x||) \leq 1 \implies \theta(H_{\lambda,g(\cdot),d}) = 0$. 
 	\label{cor:lb_cor}
 \end{proof}
 
 Recall that if the graph $G$ is sparse (finite average degree), then $\int_{x \in \mathbb{R}^d} f_{in}(||x||)dx < \infty$ which implies from Corollary \ref{cor:lambda_lower} that $\lambda_{c}$ is strictly positive in the sparse regime. The following proposition shows that this lower bound is tight for certain specific families of connection functions.

\begin{proposition}
	For all $d \geq 2$ and $R_1 \geq R_2$, if $f_{in}(r) = \mathbf{1}_{r \leq R_1}$ and $f_{out}(r) = \mathbf{1}_{r \leq R_2}$ and $\lambda$ is  such that $\theta(H_{\lambda,f_{in}(\cdot) - f_{out}(\cdot), d}) > 0$, then weak-recovery can be solved in time proportional to $n$ with the proportionality constant depending on the parameters $\lambda,R_{in}$ and $R_{out}$.
	\label{prop:lower_is_tight}
\end{proposition}

Hence, in view of Theorem \ref{thm:main_lb_cd},  for $f_{in}(r) = \mathbf{1}_{r \leq R_1}$ and $f_{out}(r) = \mathbf{1}_{r \leq R_2}$ for some $R_1 \geq R_2$, then if $\lambda$ is such that $\theta(H_{\lambda,f_{in}(\cdot) - f_{out}(\cdot), d}) = 0$, no algorithm (exponential or polynomial) time can solve weak-recovery, while if $\lambda$ is such that $\theta(H_{\lambda,f_{in}(\cdot) - f_{out}(\cdot), d}) > 0$, then a linear time algorithm exists to solve weak-recovery. This gives a sharp phase-transition for this particular set of parameters where the problem shifts from being unsolvable even with unbounded computation to being efficiently solvable.

\subsection{Algorithm and an Upper Bound for Weak Recovery}

Our main result in the positive direction is our  {\ttfamily{GBG}} algorithm described  in Section \ref{sec:algo}.  The main theorem statement on the performance of {\ttfamily GBG} is the following.

\begin{theorem}
	If $f_{in}(\cdot)$ and $f_{out}(\cdot)$ are such that $\{r \in \mathbb{R}_{+} : f_{in}(r) \neq f_{out}(r)\}$ has positive Lebesgue measure and $d \geq 2$, then there exists a $\lambda_{upper} < \infty$ depending on $f_{in}(\cdot), f_{out}(\cdot)$ and  $d$, such that for all $\lambda \geq \lambda_{upper}$, the {\text{\ttfamily{GBG}}} algorithm solves the weak-recovery problem. Moreover, {\ttfamily{GBG}} when run on data $(G_n,\phi_n)$,  has time complexity order $n^2$ and storage complexity order $n$.  
	\label{thm:positive_direction_1} 
\end{theorem}


This gives a complete non-trivial phase-transition for the sparse graph case where we have $0 < \lambda_c < \infty$ which implies the existence of  different phases. We also note that our algorithm is asymptotically optimal in a weak sense made precise in the sequel below. Denote by $\mathcal{O}_{\lambda}$ as the maximum overlap achieved by our algorithm, with the definition of overlap as given in Definition \ref{def:overlap}. More precisely, denote by $\mathcal{O}_{\lambda}$ as 
 \begin{equation}
 \mathcal{O}_{\lambda} := \sup \{\gamma \geq 0: \lim_{n \rightarrow \infty}\mathbb{P}[\mathcal{O}_n > \gamma] =1 \},
 \end{equation}
 where $\mathcal{O}_n$ is the overlap achieved by the {\ttfamily GBG} algorithm when run on the data $G_n$ and $\phi_n$. Thus, for each fixed $\lambda$, the overlap $\mathcal{O}_{\lambda} \in [0,1]$ is determinstic quantity.
\begin{proposition}
We have the limit that $\lim_{\lambda \rightarrow \infty} \mathcal{O}_{\lambda} = 1$.	
	\label{prop:weak_optimal}
\end{proposition}
In words, this proposition states that as the `signal' gets stronger, the fraction of nodes correctly classified by our algorithm tends to $1$. On a related note, we also mention in Section \ref{sec:unknown_conn_func}, a practical way of implementing the algorithm when the connection functions $f_{in}(\cdot)$ and $f_{out}(\cdot)$ are not known explicity.

\subsection{Distinguishability of the Planted Partition}

The key result we show here is that unlike in the traditional Erd\H{o}s-R\'enyi setting, the planted partition random connection model is always mutually singular with respect to \emph{any} random connection model without communities. Before precisely stating the result, we set some notation. Denote by $\mathbb{M}_{\mathcal{G}}(\mathbb{R}^d)$  the Polish space of all simple spatial graphs whose vertex set forms a locally finite set of $\mathbb{R}^d$. Thus, our random graph $G$ or the random connection model $H_{\lambda,g(\cdot),d}$ can also be viewed through the induced measure on the space $\mathbb{M}_{\mathcal{G}}(\mathbb{R}^d)$.

\begin{theorem}
	For every $\lambda >0$, $d \in \mathbb{N}$ and connection functions $f_{in}(\cdot)$ and $ f_{out}(\cdot)$  satisfying $1 \geq f_{in}(r) \geq f_{out}(r) \geq 0$ for all $r \geq 0$, and $\{r \geq 0 : f_{in}(r) \neq f_{out}(r) \}$ having positive Lebesgue measure and $g(\cdot):\mathbb{R}_{+} \rightarrow [0,1]$, the probability measures induced on the space of spatial graphs $\mathbb{M}_{\mathcal{G}}(\mathbb{R}^d)$ by  $G$ and $H_{\lambda,g(\cdot),d}$ are mutually singular. 
	\label{thm:identifiability}
\end{theorem}

 This theorem \ref{thm:identifiability} implies that this distinguishability problem as stated in Definition \ref{defn:distinguishability} can be solved with probability of success $1-o_n(1)$ for \emph{all} parameter values. Thus, the distinguishability problem as stated in our spatial case exhibits no phase-transition. A consequence of our results is that in certain regimes ($\lambda < \lambda_{c}$ for $d\geq 2$ and $\lambda >0$ for $d=1$), we can be very sure by observing the data that a partition exists, but cannot identify it better than at random. Such phenomena was proven not to be observed in a symmetric SBM with two communities (\cite{MNS_Prob_Theory}) and conjectured not to occur in any arbitrary SBM (\cite{decelle}). Technically, this theorem gives in particular that $G$ and  $H_{\lambda, \frac{f_{in}(\cdot) + f_{out}(\cdot)}{2},d}$ are mutually singular. Note that if $g(\cdot) \neq \frac{f_{in}(\cdot) + f_{out}(\cdot)}{2}$, then the average degrees of $G$ and $H_{\lambda,g(\cdot),d}$ are different and hence the  empirical average of the degrees in $G_n$ and $H_{\lambda,g(\cdot),d}$ restricted to $B_n$,  will converge almost surely as $n \rightarrow \infty$ (thanks to the ergodic property of PPP) to the mean degree, thereby making the two induced measures mutually singular Thus, the only non-trivial random connection model that can possibly be not singular with respect to $G$ is $H_{\lambda,\frac{f_{in}(\cdot) + f_{out}(\cdot)}{2},d}$, i.e. the case of equal average degrees.  We show by a slightly different albeit similar ergodic argument in Section \ref{sec:hypo} that even in the case of equal average degrees, the two induced measures are mutually singular.

\subsection{Phase Transition for Exact-Recovery}

Our main achievement in the present paper is an explicit necessary condition on the model parameters for Exact Recovery given in the following Theorem. To highlight the ideas, we primarily focus on a simplest non-trivial example of connection functions in the logarithmic regime, where we are able to conjecture a closed form expression for the phase-transition threshold. However, we present the results on the general case in Proposition \ref{prop:er_nec_cond} in the Sequel in Section \ref{sec:exact_recovery}.

\begin{definition}
	For any $0 \leq b < a \leq 1$, $n \in \mathbb{N}$, $\lambda >0$ and $d \in \mathbb{N}$, we denote by $\mathcal{G}(\lambda n,a,b,d)$ as the distribution of graph $G_n$ we defined in Section \ref{sec:math_model} with $f_{in}^{(n)}(r) = a \mathbf{1}_{r \leq \log(n)^{1/d}}$ and $f_{out}^{(n)}(r) = b \mathbf{1}_{r \leq \log(n)^{1/d}}$ for all $ r \geq 0$, where the distance on the set $B_n$ is the torridal metric.
	\label{defn:gsbm}
\end{definition}

In the results that follow, denote by $\nu_d$ to be the volume of the unit Euclidean unit ball in $d$ dimensions. 

\begin{theorem}
	For any $\lambda >0$, $d \in \mathbb{N}$ and $0 \leq b < a \leq 1$ such that  $\lambda \nu_d ( 1- \sqrt{ab} - \sqrt{(1-a)(1-b)}) < 1$,  Exact-Recovery of $G_n \sim \mathcal{G}_n(\lambda n,a,b,d)$ is not solvable.
	\label{thm:er_main_formula}
\end{theorem}

This connection functions are the simplest non-trivial instance of our model that makes the study interesting. We also believe that this necessary condition to be tight. However, at this point, we only present the following conjecture and do not pursue a proof of this.

\begin{conjecture}
		For any $\lambda >0$, $d \in \mathbb{N}$ and $0 \leq b < a \leq 1$ such that  $\lambda \nu_d ( 1- \sqrt{ab} - \sqrt{(1-a)(1-b)}) > 1$, Exact-Recovery of $G_n \sim \mathcal{G}_n(\lambda n,a,b,d)$ is solvable.
		\label{conjecture_er}
\end{conjecture}

 We believe two-round techniques developed in \cite{abbe_sandon_ch} applied to the spatial graph case can be fruitful in establishing this conjecture. To establish the existence of a phase-transition however, we analyze the {\ttfamily GBG} algorithm and adapt it to the non-sparse case to yield the following result. 

\begin{theorem}
	For every $\lambda >0$, $d \in \mathbb{N}$ and $0 \leq b < a \leq 1$, there exists $C(a,b,d) > 0$ such that if  $\lambda \nu_d ( 1- \sqrt{ab} - \sqrt{(1-a)(1-b)}) > C(a,b,d)$, Exact-Recovery of $G_n \sim \mathcal{G}_n(\lambda n,a,b,d)$ is solvable by {\ttfamily GBG} algorithm.
	\label{thm:ER_Upper_Bound}
\end{theorem}

This theorem gives the existence of different phases of the exact-recovery problem depending on $\lambda$. In particular, it states that if the intensity $\lambda$ is sufficiently high, then Exact-Recovery is solvable by our {\ttfamily GBG} algorithm. Using similar ideas as for the sparse regime, we mention in Section \ref{sec:er_ub_proof}, a way of implementing the algorithm without knowledge of the problem parameters.

\section{Algorithm for Performing Community Detection}

\label{sec:algo}

In this section, we outline an algorithm called {\ttfamily GBG} described in Algorithm  \ref{alg:main-routine} that has time complexity of order $n^2$ and storage complexity of order $n$. We make the presentation here assuming that $G_n$ is sparse, although a straightforward adaptation can be made to apply this algorithm in the logarithmic degree regime as well. We will make this adaptation more precise in the sequel in Section \ref{sec:exact_recovery}. The algorithm we present and analyze requires the knowledge of the model parameters $\lambda,f_{in}^{(n)}(\cdot)$ and $f_{out}^{(n)}(\cdot)$, although we show in Section \ref{sec:unknown_conn_func}, that by a simple modification, we can implement the algorithm  even if the model parameters are unknown to the algorithm.

\subsection{Key Idea behind the Algorithm - Dense Local Interactions}

The main and simple idea in our algorithm is that the graph $G_n$ is `locally-dense' even though it is globally sparse. In contrast to sparse Erd\H{o}s-R\'enyi based graphs in which the local neighborhood of a typical vertex `looks like a tree',  our graph will have a lot of triangles due to Euclidean geometry. This simple observation that our graph is locally-dense enables us to propose simple pairwise estimators as described in Algorithm \ref{alg:Pairwise} which exploits the fact that two nodes `nearby' in space have a lot of common neighbors (order $\lambda$). For concreteness, consider the case when $f_{in}(r) = a \mathbf{1}_{r \leq R}$ and $f_{out}(r) = b \mathbf{1}_{r \leq R}$ for some $R > 0$ and $0 \leq b < a \leq 1$. This means that points at Euclidean distance of $R$ or lesser are connected by an edge in $G$ with probability either $a$ or $b$ depending on whether the two points have the same community label or not. Moreover from elementary calculations, the number of common graph neighbors for any two nodes of $G$ at a distance $\alpha R$ away for some $\alpha < 2$  is a Poisson random variable with mean either $\lambda c(\alpha) R^d (a^2+b^2)/2$ or $\lambda c(\alpha) R^d ab$ (for some constant $c(\alpha)$ that comes from geometric arguments) depending on whether the two nodes have the same or different community labels. Thus, using a simple strategy consisting of counting the number of common neighbors and thresholding gives a probability of mis-classifying any `nearby' pair of nodes to be exponentially  small in $\lambda$. We implement this  idea in the sub-routine \ref{alg:Pairwise} below. Now, to produce the global partition one needs care  to aggregate the pairwise estimates into a global partition. Since some pair-wise estimates are bound to be in error, we must identify them and avoid using those erroneous pair-wise estimates (see also Figure \ref{fig:inconsistent}). We achieve this by classifying regions of space $B_n$ as `good' or `bad' and then by considering the pair-wise estimates only in the `good' regions.
 We prove that if $\lambda$ is sufficiently large, then the `good' regions will have sufficiently large volume and hence will succeed in detecting the communities better than at random.
\\

We summarize our main algorithm below before presenting the formal pseudo-code.

\begin{itemize}
	\item \emph{Step 1} Partition the region $B_n$ into small constant size cells and based on `local-geometry' classify each cell as good or bad. This is accomplished in the {\ttfamily Is-A-Good} routine.
	\item \emph{Step 2} Consider connected components of the Good cells and then in each of them apply the following simple classification rule. We enumerate the nodes in each connected component of Good cells in an arbitrary fashion subject to the fact that subsequent nodes are `near-by'. Then we sequentially apply the {\ttfamily Pairwise-Classify} Algorithm given in \ref{alg:Pairwise}. 
	\item \emph{Step 3} Do not classify the nodes in the bad cells and just output an estimate of $+1$ for them.
\end{itemize}

\subsection{Notation and Definitions}

In this section, we  specify the needed notations for describing our algorithm. We will  assume that the connection functions $f_{in}(\cdot)$ and $f_{out}(\cdot)$ satisfy the hypothesis of  Theorem \ref{thm:positive_direction_1}. Thus, there exists $0 \leq \tilde{r} < R <\infty$ such that $f_{in}(r) > f_{out}(r)$ for all $r \in [\tilde{r},R]$. In the rest of this section, we will use the $\tilde{r}$ and $R$ coming from the connection functions. 
\\

To describe the algorithm, we need to set some notation. We  partition the entire infinite domain $\mathbb{R}^d$ into good and bad regions. However this is just for simplicity and in practice, it suffices to do the partition for the region $B_n$.
We first tessellate the space $\mathbb{R}^d$ into cubes of side-length $\frac{R}{4d^{1/d}}$ where $R$ is as above. We identify the tessellation with the index set $\mathbb{Z}^d$, i.e. the cell indexed $z$ is a cube of side-length $\frac{R}{4d^{1/d}}$ centered at the point $\frac{zR}{4d^{1/d}} \in \mathbb{R}^d$. The subset of $\mathbb{R}^d$ that corresponds to cell $z$ is denoted by $Q_z$. Hence the cell indexed $0$ is the cube of side-length $\frac{R}{4d^{1/d}}$ centered at the origin. We now give several definitions on the terminology used for the $\mathbb{Z}^d$ tessellation and not to be confused with the terminology for describing the graph $G_n$. We collect all the different notation and terminology in this sub-section for easier access and reference.

\begin{definition}
	A set $U \subseteq \mathbb{Z}^d$ is said to be $\mathbb{Z}^d$-\textbf{connected} if for every $x,y \in U$, there exists a $k \in \mathbb{N}$ and $x_1, \cdots x_k \in U$ such that for all $i \in [0,k+1]$, $||x_i - x_{i-1}||_{\infty} = 1$, where $x_{0} := x$ and $x_{k+1} := y$. 
\end{definition}

\begin{definition}
	For any $z \in \mathbb{Z}^d$, denote by $\mathbb{Z}^d$-\textbf{neighbors} of $z$ the set of all $z^{'} \in \mathbb{Z}^d$ such that $||z - z^{'}||_{\infty} \leq 1$. 
\end{definition}

\begin{definition}
	For any subset $A \subset \mathbb{Z}^d$ and any $k \in \mathbb{N}$, the \textbf{$k$ thickening of $A$} is denoted by $\mathbf{L}_{k}(A) : = \cup_{z \in A} \cup_{z^{'} \in \mathbb{Z}^d: ||z - z^{'}||_{\infty} \leq k } z^{'}$.
\end{definition}

\begin{definition}
	For any set $B \subseteq \mathbb{Z}^d$, denote by the set $Q_B := \cup_{z \in B}Q_z$.
\end{definition}


\begin{definition}
	Let $\mathcal{Z}(\cdot) : \mathbb{R}^d \rightarrow \mathbb{Z}^d$ be the projection function, i.e. $\mathcal{Z}(x) := \inf \{z \in \mathbb{Z}^d : ||\frac{Rz}{4 d^{1/d}} - x||_{\infty} \leq 0.5 \}$. In case, of more than one $z$ achieving the minimum, we take the lexicographically smallest such $z$.
\end{definition}

\begin{definition}
	For any two points $x, y \in \mathbb{R}^d$, denote by $S_R(x,y) := B(x,R) \cap B(y,R) $, i.e.  the intersection of two balls of radius $R$ centered at points $x$ and $y$. 
\end{definition}
\begin{definition}
	 For any two points $x,y \in \mathbb{R}^d$ such that $||x-y||_{2} < R$, define by the two constants $M_{in}(x,y)$ and $M_{out}(x,y)$ as follows.
	\begin{align*}
	M_{in}(x,y) = \int_{z \in S_R(x,y)} \left( f_{in}(||x-z||)f_{in}(||y-z||) + f_{out}(||x-z||)f_{out}(||y-z||) \right) dz
	\end{align*}
	\begin{align*}
	M_{out}(x,y) = \int_{z \in S_R(x,y)} ( f_{in}(||x-z||)f_{out}(||y-z||) + f_{out}(||x-z||)f_{in}(||y-z||) )dz
	\end{align*}
\end{definition}

Observe that the definitions of $M_{in}(x,y)$ and $M_{out}(x,y)$ immediately give that
\begin{align*}
M_{in}(x,y) - M_{out}(x,y) = \int_{z \in S_R(x,y)} \left( f_{in}(||x-z||) - f_{out}(||x-z||) \right) \left( f_{in}(||y-z||) - f_{out}(||y-z||) \right) dz.
\end{align*}

\begin{definition}
	For any two points $x,y \in \phi$, denote by $E_{G}^{(R)}(x,y)$  the number of common graph neighbors of $x$ and $y$ in $G$ which are within a distance $R$ from both $x$ and $y$.
\end{definition}

\subsection{Algorithm Description in Pseudo Code}

We first present two sub-routines in Algorithms \ref{alg:Pairwise} and \ref{alg:Is_Good}  that  classify each cell of $\mathbb{R}^d$ to be either Good or Bad. The algorithm is parametrized by $\epsilon \in \left(0,\frac{1}{2}\right)$ which is arbitrary and fixed.

\begin{algorithm}[H]
	\caption{Pairwise Classifier}
	\label{alg:Pairwise}
	\begin{algorithmic}[1]
		\Procedure{Pairwise-Classify}{$i,j,\phi, G$}
		\If {$E_{G}^{(R)}(X_i,X_j) > \frac{\lambda}{2}\left( M_{in}(X_i,X_j) + M_{out}(X_i,X_j) \right)$}
		\Return $1$
		\Else \\
		\Return $-1$
		\EndIf
		\EndProcedure
	\end{algorithmic}
\end{algorithm}

In this algorithm, we classify two nodes as in the same partition if the number of common graph neighbors they have exceeds a threshold. The threshold is the average of the expected number of neighbors if the two nodes in consideration are of the same or opposite communities. Such simple tests suffices for our purpose, although one could imagine a more accurate estimator that also takes into account the number of nodes in $S_{R}(X_i,X_j)$ that do not have any edges to $X_i$ and $X_j$; or the location labels of the common neighbors. 

\begin{algorithm}[H]
	\caption{Is A-Good Testing}
	\label{alg:Is_Good}
	\begin{algorithmic}[1]
		\Procedure{Is-A-Good}{$z,G$}
		\If {$|\phi \cap Q_z| < \lambda (R/4)^d(1/d) (1- \epsilon)$} \Return FALSE
		\EndIf
		\State $\phi^{(z)} := \phi \cap ( \cup_{z^{'} : ||z-z^{'}||_{\infty} \leq 1} Q_{z^{'}})$
		\For{ all $\forall k \geq 1$, and all $X_1, \cdots X_k \in \phi^{(z)}$}
		\If{ $\prod_{i=1}^{k}  \text{PAIRWISE-CLASSIFY}(X_i, X_{i+1},G) = -1$ } \Comment{Where $X_{k+1} := X_1$}
		\State \Return FALSE
		\EndIf
		\EndFor
		\State \Return TRUE
		\EndProcedure
	\end{algorithmic}
\end{algorithm}

\begin{algorithm}
\caption{GBG}
\label{alg:main-routine}
\begin{algorithmic}[1]
	\Procedure{Main-Routine}{$G_n,\phi_n$}
	\State Classify each cell in $B_n$ to be either A-Good or A-Bad using subroutine {\ttfamily Is-A-Good}.
	\State Let $\mathcal{D}_1, \cdots \mathcal{D}_k$ be the A-Good $\mathbb{Z}^d$-connected components in $B_n$. 
	\For{$l=1,l\leq k$}
	\State Let $X_{l_1}, \cdots X_{l_{n_j}} \in \phi_n \cap Q_{\mathcal{D}_j}$ be \textbf{maximal} and arbitrary s.t $||\mathcal{Z}(X_{l_o}) - \mathcal{Z}(X_{l_{o+1}})||_{\infty} \leq 1,  \forall 1 \leq o \leq n_j-1$
	\State Set $\hat{\tau}_{l_1}^{(n)} = +1$
	\For{$w = 2, w \leq n_j$}
	\State Set $\hat{\tau}_{l_w}^{(n)} = \text{Pairwise-Classify}(l_{w-1},l_{w},\phi_n,G_n) \hat{\tau}_{l_{w-1}}^{(n)}$
	\EndFor
	\EndFor
	\For{$c = 1, c \leq N_n$}
	\If{$\hat{\tau}_{c}^{(n)} = 0$}
	\State Set $\hat{\tau}_{c}^{(n)} = +1$
	\EndIf
	\EndFor 
	\State \Return $\{\hat{\tau}_{i}^{(n)}\}_{i=1}^{N_n}$
	\EndProcedure
\end{algorithmic}
\end{algorithm}

To understand the algorithm, we need some definitions which classify cells of $\mathbb{Z}^d$  into Good or Bad depending on the `local graph geometry'.

\begin{definition}
	A cell $Q_z$ is \textbf{A-Good} if
	
	\begin{enumerate}
		\item $\bigg|\phi \cap Q_z \bigg| \geq \max \left( \lambda \left( \frac{R}{4 } \right)^d \frac{1}{d} (1- \epsilon),1 \right)$; and
		\item {\ttfamily Is-A-Good}($z,G$) returns TRUE
	\end{enumerate}
	A cell is called \textbf{A-Bad} if it is not A-Good.
	\label{defn:A-Good}
\end{definition}

The key idea of our simple algorithm lies in the definition of A-Good cells. We classify a cell to be A-Good if  there are no `inconsistencies' in the Pairwise-Estimates. See Figure \ref{fig:inconsistent} for an example of pair-wise inconsistency due to the {\ttfamily Pairwise-Classify} algorithm. In words, a cell is A-Good, if among the nodes of $G$ that either lie in the cell under consideration or in the neighboring cells, there are no inconsistencies in the output returned by the {\ttfamily Pairwise-Classify}  algorithm. Moreover, one can test whether a cell is A-Good or not based on the data $(\phi,G)$ itself as done in Algorithm \ref{alg:Is_Good}. Thus, we use the nomenclature of \emph{Algorithm}-Good as A-Good.  
\\

The main routine in Algorithm \ref{alg:main-routine} proceeds as follows. In Line $3$, we extract out all A-Good connected cells in the spatial region $B_n$. Suppose that there are $k$ A-Good connected components denoted by $\mathcal{D}_1, \cdots, \mathcal{D}_k$. Our algorithm looks at each connected component independently and produces a labeling of the nodes in them. In Line $5$, we enumerate all nodes in any A-Good connected component $\mathcal{D}_l$ as $X_{l_1} \cdots X_{l_{n_l}}$ such that for all $ 1 \leq o < n_l$, we have $|| \mathcal{Z}(X_{l_o}) - \mathcal{Z}(X_{l_{o+1}}) ||_{\infty} \leq 1$. Such an enumeration of any A-Good connected component is possible since by definition, every A-Good cell is non-empty of nodes. Now, we sequentially estimate the community labels in Line $8$ using the {\ttfamily Pairwise-Classify} sub-routine applied on `nearby' pairs of nodes. In Line $13$, we assign an estimate of $+1$, i.e. extract no meaningful clustering for nodes that fall in A-Bad cells. See also Figure \ref{fig:algo_explanation} for an illustration.

\begin{figure}
	\centering
	\includegraphics[scale=0.2]{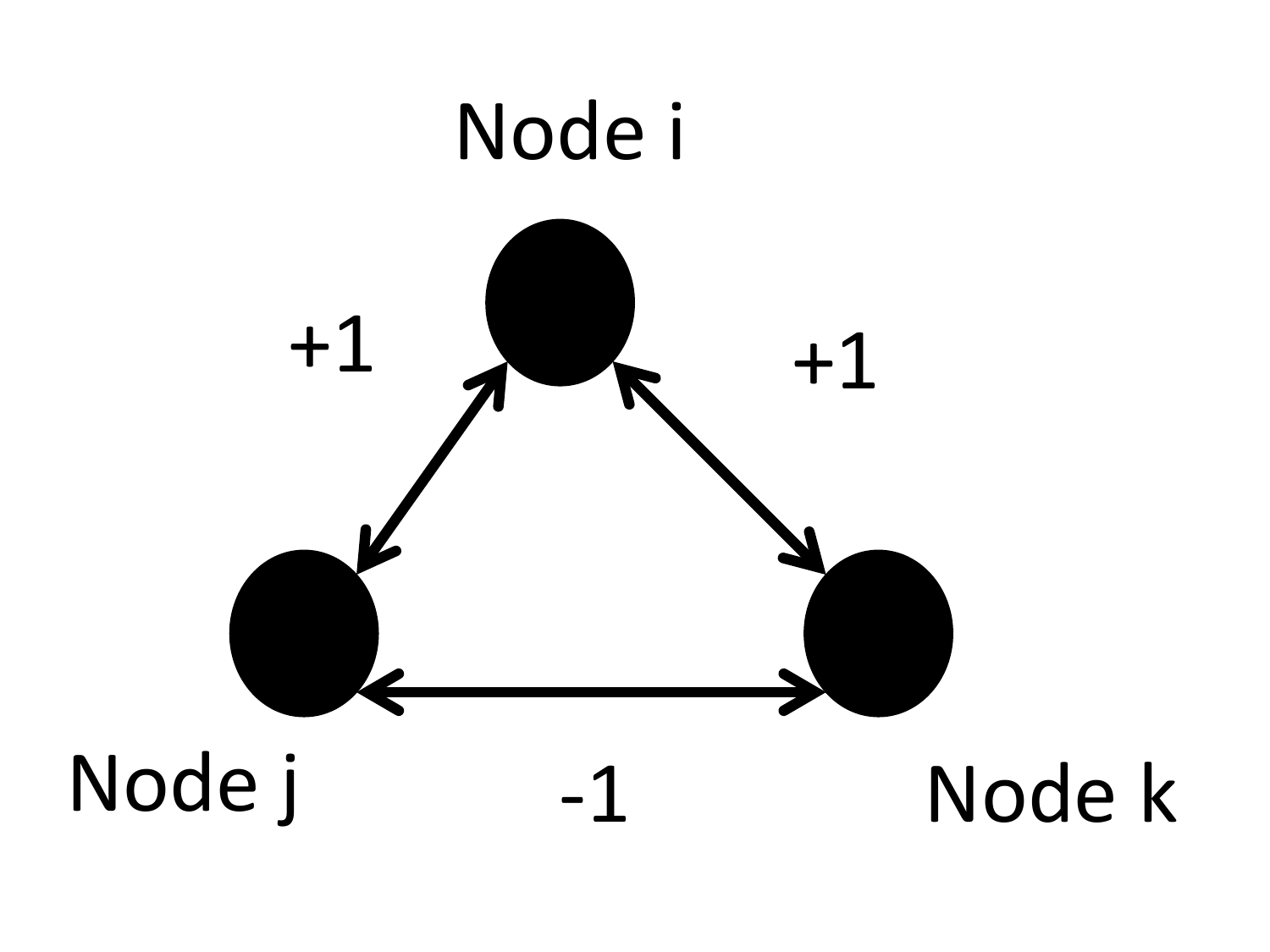}
	\caption{An illustration where Pairwise-Classify leads to inconsistency. The values on the edges represent the output of pairwise classify run on the two end points as inputs. In this example it is clear that for at-least one pair $(i,j),(j,k),(k,i)$, the output of pairwise estimate is different from the ground truth.}
	\label{fig:inconsistent}
	\end{figure}

\subsection{Complexity and Implementation}

We discuss a simple implementation of our algorithm  which takes time of order $n^2$ to run and storage space of order $n$. The multiplicative constants here depend on $\lambda$. We store the locations $\phi_n$ as a vector whose length is order $\lambda n$ and the graph $G_n$ as an adjacency list. An adjacency list representation is appropriate since $G_n$ is sparse and the average degree of any node is a constant (that depends on $\lambda$). Once we sample the locations $\phi_n$, the graph $G_n$ takes time of order $n^2$ to sample. However, if one represented the locations of nodes more cleverly, then the sampling complexity could possibly be reduced from $n^2$. Moreover, since the average degree is a constant, the storage needed is order $n$. Given the data $\phi_n$ and the graph $G_n$, we pre-process this to store another adjacency list where for every vertex, we store the list of all other vertices within a distance of $2R$ from it. This preprocessing takes order  $n^2$ time and order $n$ space. The space complexity is order $n$ since the graph is sparse. Equipped with this, we create a `grid-list' where for each coordinate of $\mathbb{Z}^d$, we store the list of vertices whose location is in the considered grid cell. This takes just order $n$ time to build.
Moreover, since only a constant number of nodes are in any grid cell and the set $B_n$  contains order $n$ cells, the storage space needed for `grid-list' is order $n$. Furthermore, since only a constant number of nodes are in a cell, it takes a constant time to test whether a particular cell is A-Good or A-Bad. Thus, to find $\mathbb{Z}^d$ connected components of Good-cells and produce the clustering takes another order $nd$ time where $d$ is the dimension. This gives our algorithm overall  a time complexity of order $n^2$ and a storage complexity of order $n$. 

\begin{figure}
	\centering
	\includegraphics[scale=0.2]{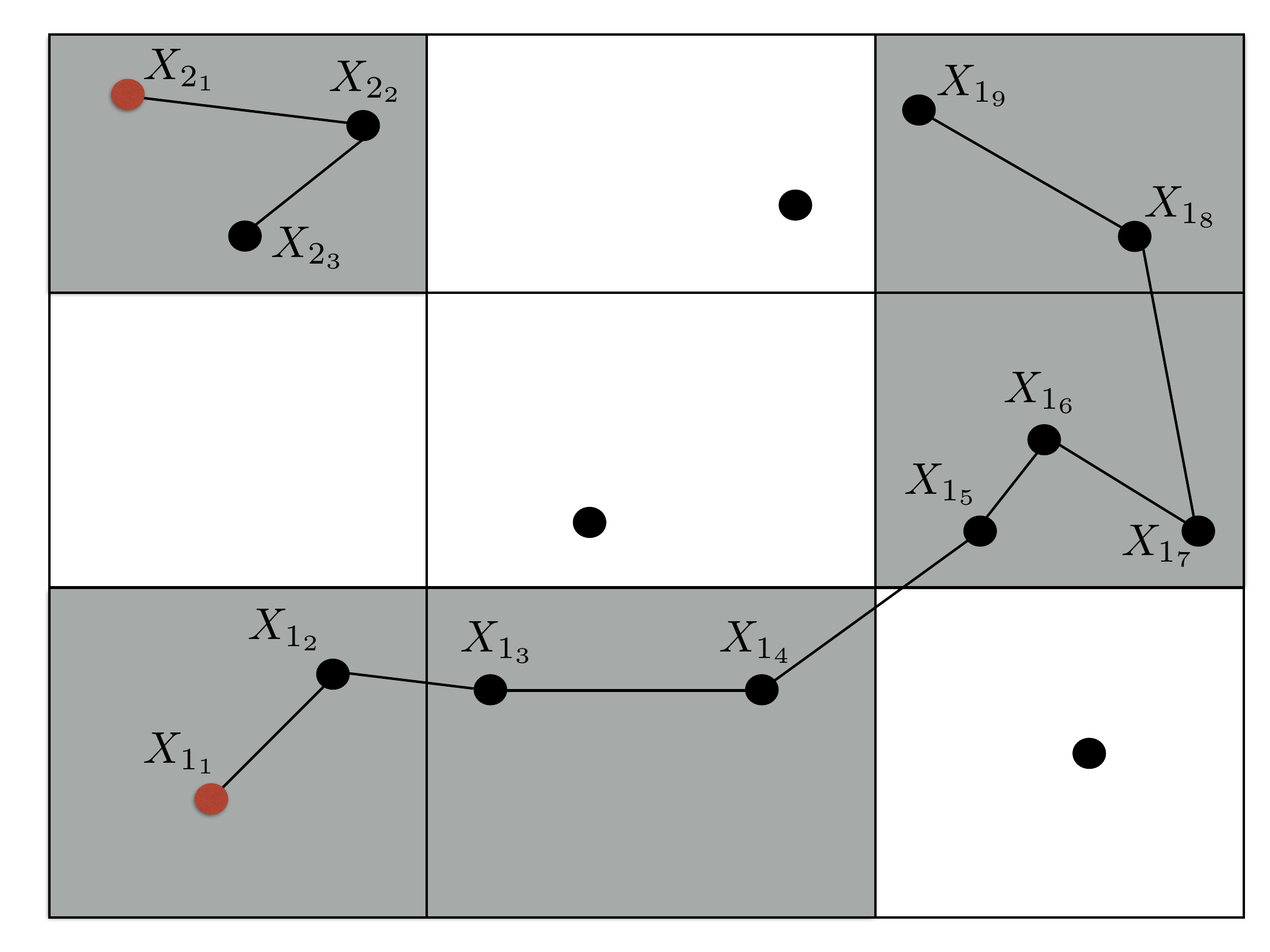}
	\caption{An illustration of algorithm \ref{alg:main-routine}. In this example, \emph{we do not draw the graph $G$}, but only show the locations of the nodes. The shaded cells corresponds to A-Good cells and in this example there are two A-Good connected components. In each component, we outline an arbitrary sequence of points $X_{1_1} \cdots X_{1_9}$ and $X_{2_1}, X_{2_2},X_{2_3}$ that will be used in line $4$ of our main Algorithm \ref{alg:main-routine}. The lines then represent how we recursively set the community label estimates of the nodes as in line $8$ of  Algorithm \ref{alg:main-routine}. The estimates for the nodes in A-Bad cell is always set to $1$.}
	\label{fig:algo_explanation}
\end{figure}

\subsection{Practical Implementation if  Model Parameters are Unknown}
\label{sec:unknown_conn_func}

In this section, we provide a simple alternative that can be used to cluster even when the model parameters $f_{in}^{(n)}, f_{out}^{(n)}(\cdot)$ and $\lambda$ are unknown to the algorithm. Assume for simplicity, that we know an estimate of $R$ such that $\int_{x \in B(0,R)} (f_{in}(||x||) - f_{out}(||x||) )dx > 0$, i.e. the set $\{r \in [0,R]: f_{out}(r) < f_{in}(r) \}$ has non-zero Lebesgue measure. Then, we can change the definition of A-Good as follows.  For any grid cell $z,z^{'} \in \mathbb{Z}^d$ such that $||z-z^{'}||_{\infty} = 1$, consider the subgraph $G_{z,z^{'}}$ of $G$ consisting of nodes whose locations lie either in cell $z$ or $z^{'}$. Consider applying some known partition to the nodes of $G_{z,z^{'}}$, for instance the standard spectral method described in \cite{abbe_overview}. We can denote the cell $z$ to be A-Good, if the number of points of $\phi$ in that cell is no smaller than $(1-\epsilon)$ of the expected value, \emph{and} for every $z^{'}$ in the $1$-thickening of $z$, the partition of the nodes of $G_{z}$, when the spectral method is applied to the induced sub-graph $G_{z,z^{'}}$ is identical, i.e. for every $z^{'},z^{''}$ in the $1$ thickening of $z$, the partition of the nodes of $G_z$ is same whether the spectral method is run on the graph $G_{z,z^{'}}$ or $G_{z,z^{''}}$. Since the spectral method of \cite{abbe_overview} does not need to know the connection functions, one can use this as an alternative definition of A-Good cell in place of Algorithm \ref{alg:Is_Good}. The only model information in this alternative implementation required is an estimate of $R$ to perform the tessellation of space. Thus in line $3$ of Algorithm \ref{alg:main-routine}, we can invoke the test described in this paragraph which does not need knowledge of the connection functions, as opposed to involing Algorithm \ref{alg:Is_Good} which does need knowledge of the parameters.


\section{Analysis and Proof of the Algorithm in the Sparse Regime}
\label{sec:analysis_algo}
The following theorem is the main theoretical guarantee on the performance of the {\ttfamily GBG} algorithm.
\begin{theorem}
	Let $\epsilon \in \left( 0, \frac{1}{2} \right)$ be arbitrarily set in Algorithm \ref{alg:Is_Good}. Let $\eta \in \left( 0, \frac{1}{2} \right) $ be such that $\left( \frac{1}{2} + \eta \right)(1 - \epsilon) > \frac{1}{2}$. Then there exists a constant $\lambda_0 < \infty$ depending on $f_{in}(\cdot),f_{out}(\cdot), d, \epsilon$ and $ \eta$ such that 
	for all $\lambda > \lambda_0$, Algorithm \ref{alg:main-routine} will solve weak-recovery.
	\label{thm:algo_analysis_global}
\end{theorem}


To prove the main result, we will need an additional classification of the cells of $B_n$ as either T-Good or T-Bad. The nomenclature stands for \emph{Truth}-Good.

\begin{definition}
	A cell $Q_z$ is \textbf{T-Good} if - 
	\begin{enumerate}
		\item $\bigg|\phi \cap Q_z \bigg| \geq \max \left( \lambda \left( \frac{R}{4 } \right)^d \frac{1}{d} (1- \epsilon),1 \right)$; and
		\item For all $i,j \in \mathbb{N}$ such that $X_i,X_j \in Q_{\mathbf{L}_1(z)} \cap \phi$, \text{Pairwise-Classify}($i,j,\phi,G$)  returns $\mathbf{1}_{Z_i = Z_j} - \mathbf{1}_{Z_i \neq Z_j}$, i.e. the ground truth.   
	\end{enumerate}
	If a cell is not T-Good, we call it \textbf{T-Bad}.
	\label{defn:T-Good}
\end{definition}

A cell is T-Good, if for any pair nodes which either lie in the cell under consideration or the neighboring cells, the output of the pairwise estimation  matches the ground-truth. Of-course since the ground truth is unknown, one cannot test whether a cell is T-Good or not. We introduce the notion of a T-Good cell to aid in the analysis.
\\

\subsection{Proof Roadmap}

The proof of Theorem \ref{thm:algo_analysis_global} can be split into three parts. The first part is composed of  combinatorial  arguments leveraging the definitions of  A-Good and T-Good cells. These combinatorial lemmas (Proposition \ref{prop:comb_global}) will conclude that it suffices to ensure that there exists a `giant' T-Good connected component in the data $(\phi_n,G_n)$. The next is a local analysis wherein we conclude that the probability a cell is T-Good can be made arbitrarily large by choosing the constant $\lambda$ sufficiently high (Corollary \ref{cor:high_val_p}). The final step is to couple the process of T-Good cells to that of dependent site percolation on $\mathbb{Z}^d$ to conclude that if  a single cell is T-Good with sufficiently high probability, then there exists a giant T-Good component comprising of many nodes (Proposition \ref{prop:concluding_comm_det}).

\subsection{Combinatorial Analysis}
\label{sec:combinatorial_analysis_algo}

 The main result in this sub-section we want to establish is the following statement.  If we establish this, then the performance of our algorithm will follow from a study of the properties of the random graph $G$.

\begin{proposition}
	If there exists a connected component of T-Good cells in the data $(\phi_n,G_n)$ which contains a fraction of nodes of $G_n$ strictly larger than a half with probability $1 - o_n(1)$, then the output returned by Algorithm \ref{alg:main-routine} solves the weak-recovery problem as given in Definition \ref{defn:comm_det}.
	\label{prop:comb_global}
\end{proposition}

The proof of the above proposition is based on the following two elementary combinatorial propositions.

\begin{proposition}
	If a cell $Q_z$ is \textbf{T-Good}, then it is also \textbf{A-Good}. In particular, every  connected \textbf{T-Good} component is  contained in some connected \textbf{A-Good} component.
	\label{lem:T-Good-A-Good}
\end{proposition}

\begin{proof}
	It suffices to prove that for any $ k \in \mathbb{N}$, $\prod_{i=1}^{k} ( \mathbf{1}_{Z_i = Z_{i-1}} - \mathbf{1}_{Z_i \neq Z_{i-1}}) = 1$, where $X_{k} := X_0$ and $Z_{k} := Z_0$, i.e. a cycle. We can see this by contradiction. Assume $\prod_{i=1}^{k} ( \mathbf{1}_{Z_i = Z_{i-1}} - \mathbf{1}_{Z_i \neq Z_{i-1}}) = -1$. This implies that an odd number of $-1's$ exists in the product. This can never be, since this would imply that $Z_0$ must be both simultaneously $+1$ and $-1$. Hence, such a product is always $+1$.
\end{proof}

The following proposition  is the basis of Line $5$ in the {\ttfamily GBG} in Algorithm \ref{alg:main-routine}. For every $z \in \mathbb{Z}^d$, denote by  $\mathcal{A}(z)$  the \emph{maximal} $\mathbb{Z}^d$ connected set containing $z$ such that for all $u \in \mathcal{A}(z)$, cell $u$ is A-Good.

\begin{proposition}
	For every $z \in \mathbb{Z}^d$ such that cell $z$ is A-Good, there exists a \emph{unique} partition of $\phi_{\mathcal{A}(z)} := \phi_{\mathcal{A}(z)}^{(+)} \coprod \phi_{\mathcal{A}(z)}^{(-)}$  such that for all $z,z^{'} \in \mathcal{A}(z)$ with $||z-z^{'}||_{\infty} \leq 1$ and all $X_i \in \phi \cap Q_z$ and $X_j \neq X_i \in \phi \cap Q_{{z}^{'}}$, we have
	\begin{itemize}
		\item If $X_i \in \phi_{z}^{(+)}$ and $X_j \in \phi_{z}^{(-)}$ or , if $X_i \in \phi_{z}^{(-)}$ and $X_j \in \phi_{z}^{(+)}$, then Pairwise-Classify($i,j,G$) will return $-1$.
		\item If $X_i,X_j \in \phi_{z}^{(+)}$ or if $X_i,X_j \in \phi_{z}^{(-)}$, then Pairwise-Classify($i,j,G$)  returns $+1$. 
	\end{itemize}
Moreover, the partition produced in Line $8$ of our Algorithm \ref{alg:main-routine} coincides with this partition.
	\label{prop:consistent_partition}
\end{proposition}

This Proposition shows that all nodes inside A-Good connected components can be partitioned into two sets uniquely, such that the T-Good sub-component inside the A-Good component will be partitioned according to the underlying ground truth. Moreover, by following any arbitrary enumeration of the nodes of $G$ as done in Line $5$ of Algorithm \ref{alg:main-routine}, we can now build this unique partition of nodes of the A-Good component. This is what allows our algorithm to be fast. The proof of this Proposition is quite standard and is defered to the Appendix in \ref{appendix_partition_proof}. We are now in a position to conclude  the proof of Proposition \ref{prop:comb_global}.

\begin{proof} Proof of Proposition \ref{prop:comb_global} \\
	
	Proposition \ref{prop:consistent_partition} justifies Line $5$ of Algorithm \ref{alg:main-routine}. First note that since every A-Good cell is non-empty of nodes of $G$, the arbitrary sequence in Line $5$ of Algorithm \ref{prop:consistent_partition} will enumerate all the nodes in each connected component. In other words, the only estimates that will be set in Line $13$ of Algorithm \ref{alg:main-routine} are those nodes that fall in the A-Bad cells. Moreover, the
	partition of the A-Good connected components in Line $8$ will coincide with the partition 
	referred to in Proposition \ref{prop:consistent_partition}.
	From the definition of T-Good components, the unique partition referred to in Proposition \ref{prop:consistent_partition} will be such that the T-Good component will be partitioned according to the ground truth. Hence, if  there exists a T-Good connected component that has a fraction $\alpha > \frac{1}{2}$ of the nodes of $G_n$, Algorithm \ref{alg:main-routine} will partition this set of nodes in accordance to the ground truth. Thus, the achieved overlap will be at-least  $2\alpha - 1 > 0$. This follows since the mis-classification of all nodes apart from this `giant' connected T-Good component cannot diminish the overlap below $2\alpha - 1$ which is still positive.
\end{proof}


\subsection{Local Analysis}
\label{sec:local_analysis_algo}

The main goal of this subsection is to show that the probability a cell is T-Good can be made arbitrarily high by taking $\lambda$ sufficiently high, which is done in Corollary \ref{cor:high_val_p}. In order to present the arguments, we recall the definition of a generalized Palm distribution. For any $k \in \mathbb{N}$ and $x_1, \cdots, x_k \in \mathbb{R}^d$, we denote by $\mathbb{P}^{x_1, \cdots x_k}$ to be the Palm distribution of $\phi$ at $x_1, \cdots x_k$.  This measure is the one induced by first sampling $\phi$ and $G$ and then placing additional points at $x_1, \cdots, x_k$ and equipping them with independent community labels and edges. More precisely, we give these nodes i.i.d. uniform community labels $Z_{-1}, \cdots Z_{-k} \in \{-1,1\}^{k}$. Conditionally on all the labels and $\phi$, we draw an edge between any $i,j \in \{-k,-(k-1),\cdots\}$ such that at-least one of $i$ or $j$ belong to $\{-k,\cdots,-1\}$ as before, i.e. with probability $f_{in}(||X_i - X_j||)$ if the two nodes have the same community labels or with $f_{out}(||X_i - X_j||)$ if the two nodes have opposite community labels independently of other edges.

\begin{proposition}
	For any two $x \neq y \in \mathbb{R}^d$ such that  $||x - y||_{2} < 2R$, then conditionally on the labels of the points at $x$ and $y$ denoted as $Z_x$ and $Z_y$ respectively, we have for all $k \in \mathbb{N}$
	\begin{itemize}
		\item If $Z_x = Z_y$, then $ \mathbb{P}^{x,y}[E_{G}^{(R)}(x,y) = k] = \frac{e^{-\lambda M_{in}(x,y) } (\lambda M_{in}(x,y))^k}{k !}$, i.e. is distributed as a Poisson random variable with mean $\lambda M_{in}(x,y)$.
		\item If $Z_x \neq Z_y$, then $ \mathbb{P}^{x,y}[E_{G}^{(R)}(x,y) = k] = \frac{e^{-\lambda M_{out}(x,y) } (\lambda M_{out}(x,y))^k}{k !}$, i.e. is distributed as a Poisson random variable with mean $\lambda M_{out}(x,y)$.
	\end{itemize}
	\label{prop:dist_edges}
\end{proposition}

\begin{proof}
	 Slivnyak's theorem for independently marked PPP gives that conditionally on $k$ points at locations $x_1, \cdots x_k \in \mathbb{R}^d$, the marked point process $\bar{\phi} \setminus \{x_1, \cdots, x_k\}$ has the same distribution as the original marked point process, i.e. is a PPP of intensity $\lambda$ with independent marks. The independent thinning property of the PPP states that if any point at $x \in \phi$ is retained with probability $p(x)$ and deleted with probability $1-p(x)$, independently of everything else, then the set of points not deleted forms a (potentially in-homogeneous) PPP.
	\\

	Notice that  the event that  any $k \in \phi \setminus \{x,y\}$ such that $k \in B(x,R) \cap B(y,R)$ has an edge to both points $x$ and $y$ in $G$ only depends on the location $k$ and the community labels of points at locations $k,x$ and $y$ and is independent of everything else. Now, since the community labels are i.i.d. and independent of $\phi$, the independent thinning property of PPP gives that the distribution of $E_{G}^{(R)}(x, y)$ is a Poisson random variable.
	\\
	
	It remains to notice that the means are precisely $\lambda M_{in}(x,y)$ and $ \lambda M_{out}(x,y)$. This follows from the Campbell - Mecke's theorem, that for any $F(\cdot) : \mathbb{R}^d \rightarrow \mathbb{R}_{+}$, we have for independently marked process is
	\begin{align}
	\mathbb{E}^{ x,y}_{\phi}\left[\sum_{z \in \phi \setminus \{x,y\}} F(z)\right] = \lambda \int_{z \in \mathbb{R}^d} \mathbb{E}^{x,y,z}_{\phi}[F(z)]dz.
	\end{align}
	
	Now, setting $F(z) := \mathbf{1}_{z \text{ has an edge to }x \text{ and } y} \mathbf{1}_{||z-x||_{2} < R} \mathbf{1}_{||z-y||_{2} < R}$ will conclude the statement on the means.
\end{proof}

\begin{proposition}
	For all connection functions $f_{in}(\cdot)$ and $f_{out}(\cdot)$ satisfying the hypothesis of  Theorem \ref{thm:positive_direction_1}, there exists a constant $c > 0$ such that
	for all $x \neq y \in \mathbb{R}^d$ satisfying $||x - y||_{2} \leq (3/4)R$, we have 
	\begin{align}
	\mathbb{P}^{x,y}[(x,y) \text{ is misclassified by Algorithm } \ref{alg:Pairwise}] \leq e^{-c \lambda},
	\end{align}
	where the constant $c$ satisfies
	\begin{align}
	c \geq \inf_{x,y \in \mathbb{R}^d : ||x-y||_{2} \leq 3R/4}    (\mathbf{1}_{M_{out}(x,y) > 0}M_{out}(x,y) + \mathbf{1}_{M_{out}(x,y) = 0}M_{in}(x,y) ) h\left( \frac{M_{in}(x,y) - M_{out}(x,y)}{2 M_{in}(x,y)} \right),
	\end{align}
	where $h(t) := (1+t) \log (1+t) - t$, for all $t \in \mathbb{R}_{+}$. In particular, $c > 0$ is strictly positive.
	\label{prop:pairwise_estimate}
\end{proposition}
\begin{proof}
	
	From Proposition \ref{prop:dist_edges}, we know that 
	$E_{G}^{(R)}(x,y)$ is either a Poisson random variable with mean $\lambda M_{in}(x,y)$ if the two nodes have the same community label or is a Poisson random variable of mean $\lambda M_{out}(x,y)$ if the two nodes have opposite community labels.  Thus, the probability of mis-classification is then
	\begin{multline}
	\mathbb{P}^{x,y}[\text{points at } x \text{ and } y \text{ are mis-classified }] = \\ \frac{1}{2} \mathbb{P}\left[X \geq  \lambda\frac{M_{in}(x,y) + M_{out}(x,y)}{2} \right] + \frac{1}{2} \mathbb{P}\left[Y \leq \lambda\frac{M_{in}(x,y) + M_{out}(x,y)}{2} \right],
	\end{multline}
	where $X$ is a Poisson random variable of mean $\lambda M_{out}(x,y)$ and $Y$ is a Poisson random variable of mean $\lambda M_{in}(x,y)$. The above interpretation is a probabilistic restatement of Algorithm \ref{alg:Pairwise}. The coefficient $1/2$ denotes the case that the points at $x$ and $y$ could be in the same community or in opposite communities. Thus, by a basic application of Chernoff's bound, we have 
	\begin{multline}
	\mathbb{P}^{x,y}[\text{points at } x \text{ and } y \text{ are mis-classified }] \leq \\ \frac{1}{2} e^{-\lambda M_{out}(x,y) h\left( \frac{M_{in}(x,y) - M_{out}(x,y)}{2 M_{in}(x,y)} \right)}  + \frac{1}{2} e^{-\lambda M_{in}(x,y) h\left( \frac{M_{in}(x,y) - M_{out}(x,y)}{2 M_{in}(x,y)} \right)},
	\end{multline}
	where $h(\cdot)$ is defined in the statement of the proposition.
	\\
	
	Now under the assumptions on the connection functions $f_{in}(\cdot)$ and $f_{out}(\cdot)$, for all $r \in [\tilde{r},R]$, $f_{in}(r) > f_{out}(r)$, we have that, $\inf_{x,y \in \mathbb{R}^d : ||x-y||_{2} \leq (3/4)R } M_{in}(x,y) - M_{out}(x,y) > 0$. Moreover, since
	$M_{in}(x,y)$ and $M_{out}(x,y)$ are non-negative, $M_{in}(x,y) - M_{out}(x,y) > 0$ implies automatically that $M_{in}(x,y) >0 $ for all $x,y \in \mathbb{R}^d$ such that $||x-y||_{2} \leq (3R/4)$. Hence, it follows that
	\begin{align}
	\sup_{x,y \in \mathbb{R}^d : ||x-y||_{2} \leq (3/4)R } \mathbb{P}^{x,y}[\text{points at } x \text{ and } y \text{ are mis-classified }] \leq e^{-c \lambda} ,
	\end{align}
	where $c$ is a strictly positive constant as given in the statement of the proposition.
	
\end{proof}

\begin{lemma}
	For all $z \in \mathbb{Z}^d$, 
	\begin{align}
 \mathbb{P}[\text{Cell } z \text{ is T-Good in graph } G ] \geq 1 - e^{-\lambda  (R/4)^d \frac{1}{d} h(\epsilon)} - \lambda^2 (3R/4)^d \frac{1}{d} e^{-c \lambda  } ,
	\end{align}
	where the constant $c$ and function $h(\cdot)$ are defined in Proposition \ref{prop:pairwise_estimate}.

	\label{lem:T-Good-Prob}
\end{lemma}

\begin{proof}
	This follows from a basic union bound. We will prove an upper bound to a cell being T-Bad. A cell is T-Bad if either the number of points is smaller than $\lambda (R/4d^{1/d})^d (1- \epsilon)$ or there exists two points $X_i$ and $X_j$ in the $1$ thickening of the cell $\{z\}$ such that when Algorithm \ref{alg:Pairwise} is run on input $(i,j,G)$, the returned answer is different from the truth.
	\\
	
	From a simple Chernoff bound, the probability that a cell has fewer than $\lambda (R/4d^{1/d})^d (1- \epsilon)$ is at-most $e^{-\lambda (R/4d^{1/d})^d h(\epsilon)}$, where $h(\epsilon)$ is strictly positive for all $\epsilon >0$. 
	\\
	
	We bound the probability that there exist two nodes that Algorithm \ref{alg:Pairwise} mis-classifies by the first moment method. We use the fact that if $X \geq 0$ is a $\mathbb{N}$ valued random variable, then $\mathbb{P}[X > 0] \leq \mathbb{E}[X]$. Hence, the probability that there exists a pair of points of $\phi$ that are mis-classified is bounded by the average number of pairs of points that are misclassified. Thus, for each cell $z$, we compute 
	\begin{align}
	\mathbb{E} [\sum_{i,j \in \mathbb{N}} \mathbf{1}_{X_i, X_j \in \mathbf{L}_1(z)} \mathbf{1}_{\text{ Algorithm } \ref{alg:Pairwise} \text{ mis-classifies } i \text{ and } j }].
	\label{eqn:pairwise_miss_1}
	\end{align}

	From the Moment-Measure expansion and the Campbell-Mecke theorem for an independently marked PPP (\cite{stoyan}), we obtain 
	\begin{multline}
	\mathbb{E} [\sum_{i,j \in \mathbb{N}} \mathbf{1}_{X_i, X_j \in \mathbf{L}_1(z)} \mathbf{1}_{\text{ Algorithm } \ref{alg:Pairwise} \text{ mis-classifies } i \text{ and } j }]  \\ = \lambda ^2 \int_{x \in Q_{\mathbf{L}_{1}(z)} } \int_{y \in Q_{\mathbf{L}_{1}(z)} } \mathbb{P}^{x,y}[\text{points at } x \text{ and } y \text{ are mis-classified }] dx dy 
	\leq \lambda^2 \left( \frac{3R}{4} \right)^d \frac{1}{d} e^{-c \lambda}. 
	\end{multline}
	The last inequality follows directly from Proposition \ref{prop:pairwise_estimate}. Therefore, by a simple union bound, we see that 
	\begin{align}
	\mathbb{P}[\text{ Cell } z \text{ is T-Bad} ]\leq e^{-\lambda (R/4d^{1/d})^d h(\epsilon)} + \lambda^2 \left( \frac{3R}{4} \right)^d \frac{1}{d} e^{-c \lambda}.
	\end{align}
	The proposition is proved by taking complements.
\end{proof}

Thus, we immediately have the following corollary which is what we will use in the sequel. The key fact to be used here is that the tessellation size $R$ does not depend on $\lambda$ and only depends on the connection functions $f_{in}(\cdot)$ and $f_{out}(\cdot)$. 
\begin{corollary}
	For every $p \in (0,1)$, and every $f_{in}(\cdot)$ and $f_{out}(\cdot)$ satisfying the hypothesis of  Theorem \ref{thm:positive_direction_1}, there exists a $\lambda^{'}$ such that for all $\lambda > \lambda^{'}$, and all $z \in \mathbb{Z}^d$, $\mathbb{P}[\text{Cell } z \text{ is T-Good}] \geq p$.
	\label{cor:high_val_p}
\end{corollary}
\begin{proof}
	It suffices to notice that for each fixed $f_{in}(\cdot)$, $f_{out}(\cdot)$ and $d$, we have 
	\begin{align}
	\lim_{\lambda \rightarrow \infty} p(\lambda) \geq \lim_{ \lambda \rightarrow \infty} 1 - e^{-\lambda  (R/4)^d \frac{1}{d} h(\epsilon)} - (\lambda^2 + \lambda) (3R/4)^d \frac{1}{d} e^{-c \lambda  }  = 1,
	\end{align}
	where $c$ is given in Proposition \ref{prop:pairwise_estimate}.
	
\end{proof}

\subsection{Global Analysis}

In this section, we present the central tool required to analyze about the `giant' connected T-Good component in the graph $G_n$. This will help us conclude the proof in the subsequent Section in Proposition \ref{prop:concluding_comm_det}. To do so, we exploit a coupling between the T-Good cells in the graph $G$ and a certain dependent site percolation process on $\mathbb{Z}^d$.
\\

\noindent{\bf Notations and Definitions} -  Denote by $(Y_z)_{z \in \mathbb{Z}^d}$ to be the random $0-1$ field on $\mathbb{Z}^d$ where $Y_z := \mathbf{1}_{\text{Cell } z \text{ is T-Good in } G}$. From the construction of the field, notice that the random field $(Y_z)_{z \in \mathbb{Z}^d}$ is only mildly dependent. Indeed, given any two $z,z^{'} \in \mathbb{Z}^d$, such that $||z - z^{'} ||_{1} \geq 12 d^{1/d}$, we have that $Y_z$ and $Y_{z^{'}}$ are independent random variables. This follows from the fact that we only look upto Euclidean distance of at-most $2R$ from any point inside a cell $z$ to determine whether a cell is T-Good or T-Bad. Since, in an independently marked PPP, events corresponding to disjoint sets of $\mathbb{R}^d$ are independent, the claim follows. For any $z \in \mathbb{Z}^d$,  cell $z$ is \textbf{open in} $\mathbb{Z}^d$ if $Y_z = 1$. Similarly, any edge connecting $z$ and $z^{'}$ is said to be open if both its end points are open. For any $z \in \mathbb{Z}^d$, we denote by ${\mathcal{C}}(z)$ to be the maximal connected random subset of $\mathbb{Z}^d$ containing $z$ such that all $z^{'} \in {\mathcal{C}}(z)$ satisfies $Y_{z^{'}} = 1$. The main proposition we want to establish in this section is the following.

\begin{proposition}
	For every $\eta \in \left( 0,\frac{1}{2}\right)$, there exists $\lambda_0(\eta,\epsilon) < \infty$ (where $\epsilon$ is set in Algorithm \ref{alg:Is_Good})  chosen sufficiently high (as a function of $f_{in}(r), f_{out}(r), r \in [0,R]$ and $d$), such that for all $\lambda > \lambda_0(\eta,\epsilon)$ and all $j \in \mathbb{Z}^d$ 
	\begin{align}
	\liminf_{n \rightarrow \infty}\frac{1}{(2n)^d} \sum_{i \in \mathbb{Z}^d: ||i-j||_{\infty} \leq n} \mathbf{1}_{i \in \mathcal{C}(j)} \geq \frac{1}{2} + \eta ,
	\end{align}
	$\mathbb{P}$ almost-surely on the event that $\{|\mathcal{C}(j)| = \infty \}$. Moreover, for $\lambda > \lambda_0(\eta,\epsilon)$, and all $j \in \mathbb{Z}^d$, $\mathbb{P}[|\mathcal{C}(j)| = \infty] \geq \frac{1}{2} + \eta$ and $\mathbb{P}[\exists j \in \mathbb{Z}^d : |\mathcal{C}(j)| = \infty] = 1$. 
	\label{prop:infinite_cluster_average}
\end{proposition}

The key insight out of the proposition we want is to ensure that by taking $\lambda$ sufficiently high, there exists an infinite open component in the process $(Y_z)_{z \in \mathbb{Z}^d}$, i.e. there exists $z \in \mathbb{Z}^d$ such that $|\mathcal{C}(z)| = \infty$. Moreover, we want to show that this infinite component  contains more than half of the sites of $\mathbb{Z}^d$. The reason this does not immediately follow from Corollary \ref{cor:high_val_p} is that we have not yet established that the infinite open component in $(Y_z)_{z \in \mathbb{Z}^d}$ if it exists is unique. However \cite{liggett} provides a clean `black-box' methodology to establish this and our proposition can be viewed as a direct corollary of Theorem $1$ in \cite{liggett}. We will  first dominate the process $(Y_z)_{z \in \mathbb{Z}^d}$ by an independent percolation process which is known to have a unique infinite component and then leverage this domination to conclude the proposition.

\begin{proof}

	 Notice that the process $(Y_z)_{z \in \mathbb{Z}^d}$ is $M := \lceil 12 d^{1/d} \rceil$ dependent. Moreover, thanks to Proposition \ref{prop:pairwise_estimate}, for every $z \in \mathbb{Z}^d$,
	\begin{align}
	\mathbb{P}[Y_z = 1 \vert \sigma (Y_u : u \in \mathbb{Z}^d, ||u - z||_{\infty} > M)] \geq p(\lambda), \text{  } \mathbb{P} \text{  a.s. },
	\end{align}
	where $p(\lambda) \rightarrow 1$ as $\lambda \rightarrow \infty$. 
	\\
	
	Thus, from Theorem $1$ in \cite{liggett}, the law of $(Y_z)_{z \in \mathbb{Z}^d}$ stochastically dominates that of i.i.d. Bernoulli    $\tilde{p}(\lambda)$ random variables    where $\tilde{p}(\lambda)$ converges to $1$ as $p(\lambda)$ converges to $1$. More precisely,  Theorem $1$ from \cite{liggett} gives the existence of a probability space $(\Omega^{'}, \mathcal{F}^{'}, \mathbb{P}^{'})$ containing two sequences of $\{0,1\}$ valued random variables $(Y^{'}_{z})_{z \in \mathbb{Z}^d}$ and $(\tilde{Y}^{'}_{z})_{z \in \mathbb{Z}^d}$ such that
	\begin{itemize}
		\item 	The distribution of $(Y^{'}_{z})_{z \in \mathbb{Z}^d}$ is the same as that of $(Y_{z})_{z \in \mathbb{Z}^d}$.
		\item For all $z \in \mathbb{Z}^d$, $Y^{'}_{z} \geq \tilde{Y}^{'}_{z}$, $\mathbb{P}^{'}$ almost-surely.
		\item $\mathbb{P}^{'}[\tilde{Y}^{'}_{z} = 1 \vert \sigma ( \tilde{Y}^{'}_{u} : u \in \mathbb{Z}^d \setminus \{z\} )] = \tilde{p}(\lambda)$, $\mathbb{P}^{'}$ almost-surely. In other words, $(\tilde{Y}^{'}_{z})_{z \in \mathbb{Z}^d}$ is an i.i.d. sequence of Bernoulli random variables with success probability $\tilde{p}(\lambda)$. 
		\item $\tilde{p}(\lambda)$ converges to $1$ as $p(\lambda)$ converges to $1$.

	\end{itemize}
	
	Denote by $\mathcal{C}^{'}(0)$ and $\tilde{C}^{'}(0)$  the cluster at the origin of the process $(Y^{'}_{z})_{z \in \mathbb{Z}^d}$ and $(\tilde{Y}^{'}_{z})_{z \in \mathbb{Z}^d}$ respectively. Denote by $\theta_d(\lambda) := \mathbb{P}^{'}[|\tilde{C}^{'}(0)| = \infty]$. From a direct application of Peirl's argument (\cite{bollobas}, Chapter $1$), it is also well know that $\theta_d(\lambda) \rightarrow 1$ as $\tilde{p}(\lambda) \rightarrow 1$. Thanks to Line $4$ above, we have $\theta_d(\lambda) \rightarrow 1$ as $p(\lambda) \rightarrow 1$. From Corollary \ref{cor:high_val_p}, this can be rephrased as $\lim_{ \lambda \rightarrow \infty} \theta_d(\lambda) = 1$.
	\\

	The stochastic domination in Line $2$ above yields 
	\begin{align}
	\frac{1}{(2n)^d} \sum_{i \in \mathbb{Z}^d: ||i-j||_{\infty} \leq n} \mathbf{1}_{i \in \mathcal{C}^{'}(j)} \geq  \frac{1}{(2n)^d} \sum_{i \in \mathbb{Z}^d: ||i-j||_{\infty} \leq n} \mathbf{1}_{i \in \tilde{\mathcal{C}}^{'}(j)} \text{ } \mathbb{P}^{'} \text{ a.s.}
	\end{align}
	On the event that $|\tilde{\mathcal{C}}^{'}(j)| = \infty$, we have 
	\begin{align}
	\frac{1}{(2n)^d} \sum_{i \in \mathbb{Z}^d: ||i-j||_{\infty} \leq n} \mathbf{1}_{i \in \mathcal{C}^{'}(j)} \geq  \frac{1}{(2n)^d} \sum_{i \in \mathbb{Z}^d: ||i-j||_{\infty} \leq n} \mathbf{1}_{ |\tilde{\mathcal{C}}^{'}(i)| = \infty} \text{ } \mathbb{P}^{'} \text{ a.s.}
	\end{align}
	This  follows from the well known fact that in an independent site percolation process that the infinite component if it exists is unique. In other-words, for all $i,j \in \mathbb{Z}^d$,  $|\tilde{\mathcal{C}}^{'}(i)| = \infty$ and $|\tilde{\mathcal{C}}^{'}(j)| = \infty$ implies  $\tilde{\mathcal{C}}^{'}(i) = \tilde{\mathcal{C}}^{'}(j)$, $\mathbb{P}^{'}$ almost-surely. Now, taking a limit on both sides, we get that
	\begin{align}
	\liminf_{n \rightarrow \infty}\frac{1}{(2n)^d} \sum_{i \in \mathbb{Z}^d: ||i-j||_{\infty} \leq n} \mathbf{1}_{i \in \mathcal{C}^{'}(j)} \geq \liminf_{n \rightarrow \infty} \frac{1}{(2n)^d} \sum_{i \in \mathbb{Z}^d: ||i-j||_{\infty} \leq n} \mathbf{1}_{ |\tilde{\mathcal{C}}^{'}(i)| = \infty} \text{ } \mathbb{P}^{'} \text{ a.s.}
	\end{align}

	From  Birkhoff's ergodic theorem, it is well known that for all $j \in \mathbb{Z}^d$ ,
	\begin{align}
	\lim_{n \rightarrow \infty} \frac{1}{(2n)^d} \sum_{i \in \mathbb{Z}^d: ||i-j||_{\infty} \leq n} \mathbf{1}_{ |\tilde{\mathcal{C}}^{'}(i)| = \infty} = \theta_d(\lambda) \text{ } \mathbb{P}^{'} \text{ a.s.}
	\end{align}

	But since $\lim_{ \lambda \rightarrow \infty}\theta_d(\lambda) = 1$, for every $\eta$ and $\epsilon$, we can take $\lambda_0(\eta,\epsilon)$ sufficiently large so that $p(\lambda)$ is sufficiently large which in turn indicates $\tilde{p}(\lambda)$ is sufficiently large so that $\theta_d(\lambda) \geq \frac{1}{2} + \eta$. The proof is concluded by 	observing that $(Y^{'}_z)_{z \in \mathbb{Z}^d} \stackrel{(d)}{=} (Y_z)_{z \in \mathbb{Z}^d}$.

\end{proof}

\subsection{Concluding that Weak-Recovery is  Solvable}
\label{sec:putting_it_together}

The following proposition along with Proposition \ref{prop:comb_global} will conclude the proof of  Theorem \ref{thm:algo_analysis_global}.

\begin{proposition}
	Let $\epsilon \in \left(0,\frac{1}{2} \right)$ be set in Algorithm \ref{alg:Is_Good}. For all  $\eta \in \left( 0,\frac{1}{2}\right)$  such that $(1 - \epsilon)\left( \frac{1}{2} + \eta\right) > \frac{1}{2}$,  for all $\lambda \geq \lambda_0(\epsilon,\eta)$ where $\lambda_0(\epsilon,\eta)$ is from Proposition \ref{prop:infinite_cluster_average},
	the fraction of nodes of $G_n$ that lie in the largest T-Good component, denoted by $\alpha_n \in [0,1]$ is such that $\liminf_{n \rightarrow \infty} \alpha_n > \frac{1}{2}$, $\mathbb{P}$ almost-surely. 
	\label{prop:concluding_comm_det}
\end{proposition}

\begin{proof}

Observe that the definition of a cell being A-Good or A-Bad is spatially `local'. More precisely, for all $z \in \mathbb{Z}^d$ such that $z + B(0,2R) \in B_n$, the event that cell $z$ being A-Good in $G_n$ is the same as cell being A-Good in $G$. We call cells $z \in \mathbb{Z}^d$ such that $z + B(0,2R) \in B_n$  \emph{internal} to $B_n$. Observe that  all $z \in \mathbb{Z}^d$ is eventually internal to $B_n$ for all $n$ large enough. Moreover, since each cell is of side $R/(4d^{1/d})$, $B_n$ has at-most $\lceil (4n^{1/d}/Rd^{1/d})^d \rceil$ cells out-of which at-least $\lfloor (4n^{1/d}/Rd^{1/d})^d \rfloor -  \lceil 8d n^{1/d} \rceil$ cells are `internal' to $B_n$. Thus, the fraction of cells in $B_n$ that are internal to $B_n$ is $1 - o_n(1)$.
\\

From Proposition \ref{prop:infinite_cluster_average}, we know that  
$\mathbb{P}[|\mathcal{C}(0) |= \infty] \geq \frac{1}{2} + \eta$ and $\mathbb{P}[ \exists z \in \mathbb{Z}^d: |\mathcal{C}(z)| = \infty] = 1$.  Moreover on the event $\{|\mathcal{C}(z)| = \infty\}$, we know from Proposition \ref{prop:infinite_cluster_average} that
\begin{align}
\liminf_{n \rightarrow \infty}\frac{1}{(2n)^d} \sum_{i \in \mathbb{Z}^d: ||i-z||_{\infty} \leq n} \mathbf{1}_{i \in \mathcal{C}(z)} \geq \frac{1}{2} + \eta \text{ } \mathbb{P} \text{ a.s.}
\end{align}
However, from an elementary counting argument, we   conclude that 
\begin{align}
\liminf_{n \rightarrow \infty}\frac{1}{(2n)^d} \sum_{i \in \mathbb{Z}^d: ||i||_{\infty} \leq n} \mathbf{1}_{i \in \mathcal{C}(z)} \geq \frac{1}{2} + \eta \text{ } \mathbb{P} \text{ a.s.}
\end{align}
In other words, the reference point does not matter when considering the limit, which can be seen easily by the following in Equation (\ref{eqn:connected_edge})
.

\begin{multline*}
\frac{1}{(2(n+z))^d} \sum_{i \in \mathbb{Z}^d: ||i-z||_{\infty} \leq n} \mathbf{1}_{i \in \mathcal{C}(z)} \leq  \frac{1}{(2n)^d} \sum_{i \in \mathbb{Z}^d: ||i-z||_{\infty} \leq n} \mathbf{1}_{i \in \mathcal{C}(z)} \\ \leq \frac{1}{(2(n+z))^d} \left(  \sum_{i \in \mathbb{Z}^d} \mathbf{1}_{||i-z||_{\infty} \geq n} \mathbf{1}_{||i||_{\infty} \leq z+n} + \mathbf{1}_{||i-z||_{\infty} \leq n}\mathbf{1}_{i \in \mathcal{C}(z)}\right).
\end{multline*}
But since for every fixed $z \in \mathbb{Z}^d$
\begin{multline*}
\frac{1}{(2(n+z))^d} \left(  \sum_{i \in \mathbb{Z}^d} \mathbf{1}_{||i-z||_{\infty} \geq n} \mathbf{1}_{||i||_{\infty} \leq z+n} + \mathbf{1}_{||i-z||_{\infty} \leq n}\mathbf{1}_{i \in \mathcal{C}(z)}\right) - \frac{1}{(2(n+z))^d} \sum_{i \in \mathbb{Z}^d: ||i-z||_{\infty} \leq n} \mathbf{1}_{i \in \mathcal{C}(z)} \\ = \BigO{n^{1-d}},
\end{multline*} 
it follows that  for all $z \in \mathbb{Z}^d$ 
\begin{align}
\liminf_{n \rightarrow \infty}\frac{1}{(2n)^d} \sum_{i \in \mathbb{Z}^d: ||i-z||_{\infty} \leq n} \mathbf{1}_{i \in \mathcal{C}(z)} = \liminf_{n \rightarrow \infty}\frac{1}{(2n)^d} \sum_{i \in \mathbb{Z}^d: ||i||_{\infty} \leq n} \mathbf{1}_{i \in \mathcal{C}(z)}.
\label{eqn:connected_edge}
\end{align}

Let $z \in \mathbb{Z}^d$  be arbitrary and condition on the event $\{|\mathcal{C}(z)| = \infty\}$. On this event, Equation (\ref{eqn:connected_edge}) and Proposition \ref{prop:infinite_cluster_average} along with the fact that the fraction of  cells in $B_n$ that are internal is  $1-o_n(1)$ give that the fraction of internal cells in $B_n$ in the connected T-Good component of cell $z$ (i.e. in $\mathcal{C}(z)$) is $\frac{1}{2} + \eta - o_n(1)$. Since there are at-least $\lambda(R/4)^d (1/d)(1-\epsilon)$  nodes of $G$ in each T-Good cell, the number of nodes of $G_n$ in this T-Good connected component is at-least $(\lfloor (4n^{1/d}/R)^d \rfloor -  \lceil 8d n^{1/d} \rceil)\left(\frac{1}{2} + \eta\right) \lambda (R/4)^d (1-\epsilon) > \frac{1}{2} (\lfloor (4n^{1/d}/R)^d \rfloor -  \lceil 8d n^{1/d} \rceil)  \lambda (R/4)^d $ since we assumed that $(1-\epsilon)\left( \frac{1}{2} + \eta \right) > \frac{1}{2}$. Moreover, from elementary Chernoff and Borell Cantelli arguments, we get that for every fixed $\epsilon^{'} > 0$, there exists a random $n_{\epsilon^{'}}$ such that for all $n \geq n_{\epsilon^{'}}$, the number of nodes in $G_n$ is less than or equal to   $ \lceil (4n^{1/d}/R)^d \rceil \lambda (R/4)^d(1 + \epsilon^{'})$ almost-surely. Now, fix an $\epsilon^{'} > 0$, such that there exists a $\gamma >0$ satisfying $\frac{(1/2 + \eta)(1-\epsilon)}{(1+\epsilon^{'})} = \frac{1}{2} + \gamma$. Thus, for $n$ larger than $n_{\epsilon^{'}}$, the fraction of nodes in $G_n$ lying the  T-Good component of cell $z$ is $\alpha_n$, where

\begin{align}
\alpha_n \geq \frac{(\lfloor (4n^{1/d}/R)^d \rfloor -  \lceil 8d n^{1/d} \rceil)\left(\frac{1}{2} + \eta\right) \lambda (R/4)^d (1-\epsilon)}{\lceil (4n^{1/d}/R)^d \rceil \lambda (R/4)^d(1 + \epsilon^{'})} \geq \frac{1}{2} + \gamma - o_n(1) 
\end{align}
almost-surely, i.e., $\lim_{n \rightarrow \infty}\mathbb{P}\left[\alpha_n > \frac{1}{2} \bigg| |\mathcal{C}(z)| = \infty\right] =1 $. But since $\mathbb{P}[\exists z \in \mathbb{Z}^d: |\mathcal{C}(z)| = \infty] = 1$, we can drop the conditioning on the event $\{|\mathcal{C}(z)| = \infty\}$ and conclude that with probability $1$, a fraction of nodes of $G_n$ strictly larger than half lie in a connected T-Good component.

\end{proof}

\subsection{Proof of Proposition \ref{prop:weak_optimal}}

From Proposition \ref{prop:concluding_comm_det}, we know that for every $\epsilon \in (0,1)$ and $\eta \in (0,\frac{1}{2}])$, there exists $\lambda_0(\epsilon,\eta) < \infty$, such that for all $\lambda > \lambda_0(\epsilon,\eta)$, the GBG algorithm achieves an overlap of $(\frac{1}{2} + \eta)(1-\epsilon)$. The proof is concluded by noticing that for any $\delta \in (\frac{1}{2},1)$, we can choose $\epsilon \in (0,1)$ and $\eta \in (\frac{1}{2} , 1)$ such that $(\frac{1}{2} + \eta)(1-\epsilon) > \delta$.

\section{Lower Bound for Community Detection}
\label{sec:lower_bound}

The goal of this section is to prove Theorem \ref{thm:main_lb_cd}. The central idea is to consider the problem of how well can one estimate whether two uniformly randomly chosen nodes of $G_n$ belong to the same or opposite communities better than at random. This problem is indeed easier than Community Detection which requires one to produce an entire partition of the nodes of $G_n$. We will show that the natural way to understand the pairwise classification problem is through another problem which we call `Information Flow through Infinity' which we define in the sequel in Section \ref{subsec:info_flow_pblm}. Informally, this problem asks 
whether one can estimate with success probability larger than a half, the community label of any node chosen uniformly at random from $G_n$, {given} the graph, the spatial locations \textbf{and} the true community labels of all nodes whose spatial locations are \emph{far} away (at infinity) from this chosen node. Subsequently, the core technical argument of this section is to establish an impossibility result for Information Flow from Infinity which we state below in Theorem \ref{thm:main_lower_bound}. To aid us in developing the technical arguments, it is instructive to first consider the proof of Proposition \ref{prop:lower_is_tight}  (which was stated in Section \ref{sec:results}), which identifies a special case of connection functions $f_{in}(\cdot)$ and $f_{out}(\cdot)$ when the phase-transition is sharp.

\subsection{Proof Roadmap}

We will first establish in Section \ref{subsec:motivating_example}, the proof of Proposition \ref{prop:lower_is_tight}, which considers a specific examples of $f_{in}(\cdot)$ and $f_{out}(\cdot)$. Subsequently, in Section \ref{subsec:info_flow_pblm}, we define the Information flow from Infinity problem (in Definition \ref{defn:info_flow_infinitiy}) and show in Lemma \ref{lem:lemma_easy} that Information flow from Infinity is easier than Community Detection. We will require two supporting Propositions \ref{prop:monotone_repeat} and \ref{lem:palm_ergodicity}, which will aid in the proof of Lemma \ref{lem:lemma_easy}. Subsequently, in Subsection \ref{subsec:lower_bound_proof}, we will state and prove Theorem \ref{thm:main_lower_bound}, which will gives the impossibility for Information flow from Infinity problem. Thanks to Lemma \ref{lem:lemma_easy}, this then proves Theorem \ref{thm:main_lb_cd}.

\subsection{Proof of Proposition \ref{prop:lower_is_tight}}
\label{subsec:motivating_example}
Let $R_{in} > R_{out} \geq 0$ be arbitrary and consider the two functions to be  $f_{in}(r) = \mathbf{1}_{r \leq R_{in}}$ and $f_{out}(r) = \mathbf{{1}}_{r \leq R_{out}}$. In words, two points of opposite communities are connected if and only if their distance is lesser than $R_{out}$ and two points of the same community are connected if and only if their distance is smaller than $R_{in}$. In this example, it is clear that for any two points $X_i,X_j \in \phi$, no matter their community labels $Z_i$ and $Z_j$, if $||X_i - X_j||_{2} \leq R_{out}$, then $i$ and $j$ are always connected in $G$. Similarly, any two points $X_i$ and $X_j$ such that $||X_i - X_j||_{2} > R_{in}$ are never connected by an edge in $G$ no matter their community labels $Z_i$ and $Z_j$. Hence,  the \emph{informative} pairs of points in this example are those $X_i,X_j$ such that $||X_i - X_j||_{2} \in (R_{out}, R_{in}]$. Moreover, it is immediate that, if $||X_i - X_j||_{2} \in (R_{out}, R_{in}]$ and $i \sim_{G}j$, then $Z_i = Z_j$. On the other hand if $||X_i - X_j||_{2} \in (R_{out}, R_{in}]$ and $i \nsim_{G} j$, then it must be the case that $Z_i \neq Z_j$. For any two points $X_i$ and $X_j$ such that $||X_i - X_j||_{2} \in [0,R_{out}] \cup (R_{in},\infty)$, the presence or absence of an edge is not informative as it is a certain event. 
\\

This example motivates the following simple algorithm for Community Detection. Partition the nodes of $G_n$ into $\mathcal{D}_1, \cdots \mathcal{D}_k$ where each component $\mathcal{D}_i$ is a maximal set of nodes $\{X_{i_1},\cdots X_{i_{l_i}}\}$ of $G_n$ such that for all $j \in [1,l_i]$, we have $||X_{i_{j-1}} - X_{i_j}|| \in (R_{out},R_{in}]$. In words, we form another graph $T_n$ from the points $\phi_n$ such that any two nodes $i$ and $j$ of $G_n$ are connected in $T_n$ if and only if $||X_i - X_j|| \in (R_{out},R_{in}]$. Then $\mathcal{D}_1, \cdots \mathcal{D}_k$ are the connected components of the graph $T_n$. The algorithm works by considering and labeling each connected component $\mathcal{D}_i$ independently of other components. For each cluster $i \in [1,k]$, estimate the node label of $X_{i_1}$ to be $+1$. Then for every $j \in [2,l_i]$, recursively estimate the node label by the following procedure- 
\begin{itemize}
	\item If $i_{j-1} \sim_{G_n} i_j$ then set $Z_{i_{j}} = Z_{i_{j-1}}$.
	\item If $i_{j-1} \not\sim_{G_n} i_j$ then set $Z_{i_{j}} = -Z_{i_{j-1}}$.
\end{itemize}

This algorithm considers each of the connected component of $T_n$ enumerated in an arbitrary manner and then labels the nodes in these components. The following very elementary proposition explains when this algorithm will perform well.

\begin{proposition}
	Let $R_{in} > R_{out} \geq 0$ be arbitrary such that  $f_{in}(r) = \mathbf{1}_{r \leq R_{in}}$ and $f_{out}(r) = \mathbf{1}_{r \leq R_{out}}$. If $\theta(H_{\lambda, f_{in}(\cdot)-f_{out}(\cdot),d}) > 0$, then the procedure described above  solves Community Detection for this set of parameters.
	\label{prop:lower_is_tight_proof} 
\end{proposition}
Note that in view of Theorem \ref{thm:main_lower_bound}, Proposition \ref{prop:lower_is_tight_proof} will imply Proposition \ref{prop:lower_is_tight}.

\begin{proof}
 Notice that if $f_{in}(r) = \mathbf{1}_{r \leq R_{in}}$ and $f_{out}(r) = \mathbf{1}_{r \leq R_{out}}$, then $f_{in}(r) - f_{out}(r) = \mathbf{1}_{R_{out} \leq r < R_{in}}$. From the properties of the construction of the graph, any two $i \neq j \in \mathbb{N}$ such that $||X_i - X_J|| \in (R_{out},R_{in}]$ satisfies - 
 \begin{itemize}
 	\item $Z_i = Z_j$ if $i \sim_Gj$
 	\item $Z_i \neq Z_j$ if $i \nsim_G j$.
 \end{itemize}
 Hence, it is clear that the algorithm described in the preceding paragraph partitions each cluster $\mathcal{D}_i$, $i \in [1,k]$ exactly in accordance to the ground truth. However, it could be that the estimated signs in each of the connected components $\mathcal{D}_{i}$ could be flipped from the underlying ground truth and hence the achieved overlap can still be small even though we partition each cluster $\mathcal{D}_i$ accurately. To argue that the overlap achieved by the algorithm is not too small, a sufficient condition is that  there exists a unique giant (of size $cn - o(n)$ for some $c>0$) component of $T_n$ and all other connected components are $o(n)$. Then, we will have by the strong-law of large numbers that the overlap achieved will be $c$, i.e. the mislabeling in all small components will `cancel' each other out and in particular cannot drive the overlap of $c$ achieved in the giant component to $0$. From the definition of percolation, a unique giant component in $T_n$ exists if and only if $\theta(H_{\lambda, f_{in}(\cdot)-f_{out}(\cdot),d}) > 0$ since $T_n \stackrel{(d)}{=} H_{\lambda,f_{in}(\cdot)-f_{out}(\cdot),d}$.

\end{proof}

In the sequel, we will generalize the above example to come up with the general lower bound for Community Detection problem. 

\subsection{The Information Flow from Infinity Problem}
\label{subsec:info_flow_pblm}

This problem refers to how well can one estimate the community label of a tagged node of a graph better than at random, given some extra `information at infinity'. We make this problem precise by posing this question under the Palm Probability measure $\mathbb{P}^0$. Recall that the Palm measure is the distribution of the graph $G$ obtained by placing an additional node at the origin and equipping it with an independent community label and edges to other existing nodes. For every $r \in \mathbb{R}_{+}$, denote by $\phi^{(r)}$ and $G^{(r)}$  the  point-process and graph, in which every vertex $i \in \mathbb{N}$ (which is at location $X_i \in \mathbb{R}^d$)  is equipped with the random variable $Z_i \mathbf{1}_{||X_i||_2 \geq r}$. Note that this is not a mark since it is not translation invariant, but is a random variable associated with vertex $i$. In words, we retain the community label marks on nodes of $G$ at a Euclidean distance of $r$ or more from the origin and delete (i.e. set to $0$) the community label of those nodes which are located at distances less than $r$ from the origin.


\begin{definition}
	We say \emph{Information Flows from Infinity} if for every $r \in \mathbb{R}_{+}$  there exists a random variable $\tau_r \in \{-1,+1\}$,  measurable (deterministic function) with respect to the observed data $(\phi^{(r)},G^{(r)})$ and a constant $\gamma > 0$ such that 
	\begin{align}
	\liminf_{r \rightarrow \infty} \mathbb{P}^{0}[\tau_r = Z_0] \geq \frac{1}{2} + \gamma.
	\label{eqn:info_flow_defn}
	\end{align}
	\label{defn:info_flow_infinitiy}
\end{definition} 
We say information `flows' from infinity if we are able to non-trivially estimate the community label at origin, given  `information at infinity'. Note that for each $r$, there exists algorithms (i.e. $\tau_r$) such that $\mathbb{P}^{0}[\tau_r = Z_0] > \frac{1}{2}$. However, the non-trivial question is to understand if the limit as $r \rightarrow \infty$ is still strictly larger than a half. This definition is similar in spirit to those considered in Ising models to detect phase-transition for multiplicity of Gibbs states (as in \cite{Tree_95} and \cite{MNS_Prob_Theory}).  We first establish a monotonicity property of this problem and connect it with the Community Detection Problem.

\begin{proposition}
	For every $d \in \mathbb{N}$ and $f_{in}(\cdot),f_{out}(\cdot) : \mathbb{R}_{+} \rightarrow [0,1]$, the  limit \\ $\lim_{r \rightarrow \infty}\sup_{\tau_{r} \in \sigma((\bar{G}^{(r)} ,\bar{\phi}^{(r)}))} \mathbb{P}^{0} [\tau_{r} = Z_0]$ exists. Moreover,  $\lambda \rightarrow  \lim_{r \rightarrow \infty}\sup_{\tau_{r} \in \sigma((\bar{G}^{(r)} ,\bar{\phi}^{(r)}))} \mathbb{P}^{0} [\tau_{r} = Z_0]$ is non-decreasing. 
	\label{prop:monotone_repeat}
\end{proposition}

Note the supremum is over all possible estimators of the community label at origin.
\begin{proof}
	
	Denote by $\tilde{\xi}(\lambda, r) := \sup_{\tau_{r} \in \sigma((\bar{G}^{(r)} ,\bar{\phi}^{(r)}))} \mathbb{P}^{0}_{\phi} [\tau_{r} = Z_0]$. 
	Notice that, for each fixed $\lambda$ and $r^{'} \geq r$, we have $\sigma((\bar{G}^{(r^{'})} ,\bar{\phi}^{(r^{'})})) \subseteq \sigma((\bar{G}^{(r)} ,\bar{\phi}^{(r)}))$. This follows from the fact that sample path-wise,$(\bar{G}^{(r^{'})} ,\bar{\phi}^{(r^{'})})$ is a measurable function of $(\bar{G}^{(r)} ,\bar{\phi}^{(r)})$ which is obtained by zeroing all revealed labels in the set $B_{r}^{\complement}   \cap B_{r^{'}}$.  Hence, the limit in proposition \ref{prop:monotone_repeat} exists.
	\\
	
	It remains to prove that $\xi(\lambda) := \lim_{r \rightarrow \infty} \tilde{\xi}(\lambda,r)$ is non-decreasing in $\lambda$. It suffices to prove that $\tilde{\xi}(\lambda,r)$ is non-decreasing in $\lambda$ for every $r$. We show this by using  a standard coupling argument used to prove monotonicity of percolation probabilities (for example in Chapter $2$, \cite{Meester_Roy}). The basis of the coupling argument is the independent thinning property and Slivnyak's theorem of the PPP and the associated random connection model. These two theorems gives the following two facts. Let  $(\phi,G)$ be a Poisson Point Process of intensity $\lambda$ and $G$ is the block model graph for some  connection functions $f_{in}(\cdot)$ and $f_{out}(\cdot)$ under measure $\mathbb{P}$. Then if each node of $G$ along with its incident edges are removed independently with probability $p$, the resulting point process $\phi^{'}$ is an instance of a PPP with intensity $\lambda p$ and the resulting graph $G^{'}$ is the associated block model graph with the same connection functions $f_{in}(\cdot)$ and $f_{out}(\cdot)$. Slivnyak's theorem for $(\phi,G)$ gives that if we place an extra node at origin and equip it with independent community label and edges, the resulting point-process and graph is equal in distribution to $(\phi,G)$ under the Palm measure $\mathbb{P}^{0}$. 
	\\

	Thus given a problem instance at intensity $\lambda$ under measure $\mathbb{P}^0$, we can independently remove nodes of $G$ other than the one at origin with probability $p$. The resulting graph and the Information Flow from Infinity problem will be that at intensity $\lambda p$.  Thus, the best performance at intensity $\lambda$ cannot be smaller than that at intensity $\lambda p$. Since $p$ was arbitrary, we have that   the best performance at intensity $\lambda$ cannot be smaller than that at any intensity $\lambda^{'} \leq \lambda$. In other words, for all $r \geq 0$, $\tilde{\xi}(\lambda^{'}, r) \leq \tilde{\xi}(\lambda,r)$.	
\end{proof}

We will need the following classical result on the ergodic property of marks of a stationary point process.
\begin{proposition} (\cite{daley})
	Let $\phi := \{X_1,X_2,\cdots\}$ be a homogeneous PPP with its atoms enumerated in an arbitrary measurable way. Let each atom $i \in \mathbb{N}$ be assigned a translation invariant mark random variable $J_i \in \Xi$ taking values in an arbitrary Borel measurable space $(\Xi, \mathfrak{\Xi})$. Let $B_n := \left[ -\frac{n^{1/d}}{2} , \frac{n^{1/d}}{2}\right]^d$ be the box of volume $n$ and let $X_{j^{(n)}} \in \phi$ be chosen uniformly at random among the atoms of $\phi$ that lie in $B_n$, if any. Then for all $A \in \mathfrak{\Xi}$, the  limit $\lim_{n \rightarrow \infty} \mathbb{P}[J_{j^{(n)}} \in A]$  exists and satisfies $\lim_{n \rightarrow \infty} \mathbb{P}[J_{j^{(n)}} \in A] = \mathbb{P}^0[J_{0} \in A]$, where $J_0$ is the mark of the atom of $\phi$ at origin under $\mathbb{P}^0$.
	\label{lem:palm_ergodicity}
\end{proposition}

The following proposition establishes that Community Detection is harder than Information Flow from Infinity.

\begin{lemma}
	If there exists a Community Detection algorithm (polynomial or exponential time) that achieved an overlap of $\gamma > 0$, then $\lim_{r \rightarrow \infty}\sup_{\tau_{r} \in \sigma((\bar{G}^{(r)} ,\bar{\phi}^{(r)}))} \mathbb{P}^{0} [\tau_{r} = Z_0] \geq \frac{1 + \gamma}{2}$.
	\label{lem:lemma_easy}
\end{lemma}
\begin{proof}
	We will assume that we cannot solve Information Flow from Infinity problem and then conclude that Community Detection  is not solvable. More precisely, we will assume that \\ $\lim_{r \rightarrow \infty}\sup_{\tau_{r} \in \sigma((\bar{G}^{(r)} ,\bar{\phi}^{(r)}))} \mathbb{P}^{0} [\tau_{r} = Z_0] \leq \frac{1 }{2}$ and then argue that no Community Detection algorithm can achieve a positive overlap. A Community Detection algorithm achieves an overlap $\gamma > 0$ if when  run on the data $(G_n,\phi_n)$, it produces an output $\{\tau_{i}^{(n)}\}_{i=1}^{N_n}$ satisfying 
	\begin{align}
	\frac{|\sum_{i = 1}^{N_n} \tau_{i}^{(n)} Z_i |}{N_n} \geq \gamma,
	\end{align}
	with probability $1-o_n(1)$. Now, an \emph{easier} question corresponds to asking if any two uniformly randomly chosen nodes (with replacement) of $G_n$ belong to the same or opposite community. This question is easier than Community Detection since one way to answer this pairwise question is to first produce a partition of all nodes of $G_n$ and then answer the question for the two randomly chosen nodes. Note that an overlap of $\gamma$ can be achieved if and only if a fraction $(1+\gamma)/2$ of the nodes have been correctly classified. Hence the chance that any two uniformly chosen nodes are classified correctly is at-least $(1+\gamma)/2$. Since we can achieve an overlap of $\gamma$ with probability $1 - o_n(1)$, the chance that two uniformly randomly chosen nodes of $G_n$ to be correctly classified is at-least $(1+\gamma)/2 - o_n(1)$. Hence, if we show that the best estimator for answering whether any two randomly chosen nodes from $G_n$ belong to the same or opposite community has a success probability of at-most $\frac{1}{2} + o_n(1)$, then no algorithm exists for solving Community Detection. In the rest of the proof, we will show that if the Information Flow from Infinity cannot be solved, then for every $\epsilon > 0$, the best estimator to estimate whether any two randomly chosen nodes of $G_n$ belong to the same or opposite communities will succeed with probability at-most $\frac{1}{2} + \epsilon + o_n(1)$. This will conclude the proof that no algorithm exists for solving Community Detection in view of the preceding discussion and hence the proof of Lemma \ref{lem:lemma_easy}.
	\\
	
	Let $\epsilon > 0$ be arbitrary. Under the assumption that Information Flow from Infinity cannot be solved, there exists a $r > 0$ such that $\sup_{\tau_r \in \sigma(({G}^{(r)} ,\phi^{(r)}))} \mathbb{P}^0[\tau_r = Z_0] \leq \frac{1}{2} + \frac{\epsilon}{2}$. In words, choose a $r$ such that the Information Flow from Infinity cannot succeed with probability larger than $\frac{1}{2} + \frac{\epsilon}{2}$. Now, let $n$ be large enough such that for two uniformly randomly chosen nodes of $G_n$ denoted by $i$ and $j$ to be $||X_i - X_j|| > r$ with probability at-least $1 - \frac{\epsilon}{2}$. Now, assume that we are on the event that $||X_i - X_j|| > r$. Conditionally on this event, the probability that any pairwise estimator correctly tells whether the two nodes $i$ and $j$ are in the same or opposite community will succeed with probability at-most $\frac{1}{2} + \frac{\epsilon}{2}$. This follows since conditionally on $X_i$ and $X_j$, we can make the pairwise problem  \emph{easier} by revealing all community labels of nodes at a distance of larger than $r$ from $X_i$ and asking whether we can now guess the community label at $X_i$. This will enable us to answer the pairwise question of whether $X_i$ and $X_j$ lie in the same community or not since we will know the true label of $X_j$ when the labels of nodes at distances $r$ or more from $X_i$ are revealed. This now is a problem of finding a mark $\tau_i$ of the atom $i$ of $\phi$ which denotes the best community label estimate of $X_i$ given $\phi,G$ and all community labels of nodes at a distance of $r$ or more from $X_i$. Since $X_i$ was an uniformly randomly chosen point from $\phi \cap B_n$, the chance that $\tau_i = Z_i$ is equal to the Palm probability that the best community label estimate of the node at origin is correct given $\phi,G$ and the true community labels of all nodes at a distance $r$ or more from the origin. This follows from a direct application of Proposition \ref{lem:palm_ergodicity}. Thus the probability $\tau_i = Z_i$ is bounded from above by $\frac{1}{2} + \frac{\epsilon}{2} + o_n(1)$. On the complementary event that $||X_i - X_j|| < r$, we use the trivial bound that the pairwise estimation is always successful. Hence by the law of total probability, the success probability of the pairwise estimator cannot be larger than $\frac{1}{2} + \epsilon + o_n(1)$. In other words, for every $\epsilon > 0$, there exists a $n_{\epsilon} < \infty$, such that for all $n \geq n_{\epsilon}$, the probability that we correctly identify the community membership of two uniformly randomly chosen nodes of $G_n$ is at-most $\frac{1}{2} + \epsilon$.

\end{proof}

The following is the main technical result on the Information Flow from Infinity problem.

\subsection{Main Result on Information Flow from Infinity}
\label{subsec:lower_bound_proof}
\begin{theorem}
	For every $\lambda,f_{in}(\cdot),f_{out}(\cdot)$ and $d$, the following limit exists and satisfies
	\begin{align}
	\lim_{ r \rightarrow \infty} \sup_{\tau_{r} \in \sigma(G^{(r)} ,\phi^{(r)})} \mathbb{P}^{0}[\tau_r = Z_0] \leq \frac{1}{2}\left( 1 + \theta(H_{\lambda,f_{in}(\cdot) - f_{out}(\cdot), d}) \right).
	\label{eqn:lower_bound}
	\end{align}
	\label{thm:main_lower_bound}
\end{theorem}
Recall that $\theta(H_{\lambda,f_{in}(\cdot) - f_{out}(\cdot), d})$ is the percolation probability of the classical random connection model where any two nodes of $\phi$ located at $x,y \in \mathbb{R}^d$ are connected by an edge with probability $f_{in}(||x-y||) - f_{out}(||x-y||)$. The supremeum is over all valid estimators of the community label at origin and hence if $\theta(H_{\lambda,f_{in}(\cdot) - f_{out}(\cdot), d}) = 0$, then there is no estimator that will solve the Information Flow from Infinity problem.  In view of  Lemma \ref{lem:lemma_easy}, we also get that if $\theta(H_{\lambda,f_{in}(\cdot) - f_{out}(\cdot), d})=0$, then there is no algorithm (polynomial or exponential time) to solve Community Detection. Thus, if we prove Theorem \ref{thm:main_lower_bound}, then we will conclude the proof of Theorem \ref{thm:main_lb_cd}.
\\

Before presenting the proof of Theorem \ref{thm:main_lower_bound}, we illustrate a few example setting where the bound in Equation \ref{eqn:lower_bound} is tight and loose respectively. In view of Lemma \ref{lem:lemma_easy} and Proposition \ref{prop:lower_is_tight}, the following corollary where Equation (\ref{eqn:lower_bound}) is tight holds.
\begin{corollary}
		For all $\lambda >0, R_1 \geq R_2$, if $f_{in}(r) = \mathbf{1}_{r \leq R_1}$ and $f_{out}(r) = \mathbf{1}_{r \leq R_2}$ 
		\begin{align*}
		\lim_{ r \rightarrow \infty} \sup_{\tau_{r} \in \sigma(G^{(r)} ,\phi^{(r)})} \mathbb{P}^{0}[\tau_r = Z_0] = \frac{1}{2}\left( 1 + \theta(H_{\lambda,f_{in}(\cdot) - f_{out}(\cdot), d}) \right).
		\end{align*}
\end{corollary}
In other words, we see that the inequality in Theorem \ref{thm:main_lower_bound} is achieved in certain examples. However, Theorem \ref{thm:main_lower_bound} is not an accurate characterization of the Information Flow from Infinity problem as evidenced in the following example. 

\begin{proposition}
	For all $d \geq 2$, if $f_{in}(r) = \min \left( 1, \frac{1}{\sqrt{r}} + \frac{1}{r^{d-1/4}}\right)$ and $f_{out}(r) = \min \left(1,\frac{1}{\sqrt{r}} \right)$, the inequality in Equation (\ref{eqn:lower_bound}) is strict for all values of $\lambda > 0$.
	\label{prop:lower_is_loose}
\end{proposition} 

The example in  Proposition \ref{prop:lower_is_loose} corresponds to the case when the degree of each node is almost-surely infinite. Thus, $\theta(H_{\lambda,f_{in}(\cdot) - f_{out}(\cdot), d}) = 1$ in this case. However, using results from \cite{acs_singular}, one can argue that perfect recovery is impossible in this example, i.e. $\lim_{ r \rightarrow \infty} \sup_{\tau_{r} \in \sigma(G^{(r)} ,\phi^{(r)})} \mathbb{P}^{0}[\tau_r = Z_0] < 1$. The key tool, is to see that if perfect recovery were to be possible, then it would be the case that either of the following two pairs of point-process will be mutually singular. 
\begin{enumerate}
	\item The point process
	formed by the location of those nodes of  $G$ that have an edge to the origin and have a community label $Z_0$ and the point process formed by the location of those nodes of $G$ having an edge to the origin and having a community label of $-Z_0$ are mutually singular.

	\item Or, the point process corresponding to the locations of those nodes of $G$ that have a community label $Z_0$ and do not have an edge to the origin and the point process corresponding to the locations of those nodes of $G$ that have a community label $-Z_0$ and do not have an edge to the origin are mutually singular. 
\end{enumerate}

 We will argue that in our example, neither is possible by alluding to a theorem from \cite{acs_singular}, and hence perfect recovery is not possible. We present the complete proof in the Appendix \ref{appendix_proof}.

\subsection{The Information Graph and Proof of Theorem \ref{thm:main_lower_bound} }
\label{sec:I_Graph}
In this section, we  generalize the example of the previous section and give a proof of Theorem \ref{thm:main_lower_bound}. To do so, we define a general information graph and  conclude that if this constructed information graph does not percolate, then one cannot solve the Information Flow from Infinity problem. 
\\

\noindent {\bf Notation and Definition} - We denote by $I$  the information graph whose vertex set is $\phi$. The random graph $I$ is constructed just based on the positions of the points and the random elements $\{\{U_{ij}\}_{j > i} \}_{i \in \mathbb{N}}$. Recall that the graph $G$ was built by connecting any two points $i < j \in \mathbb{N}$ if $U_{ij} \leq \mathbf{1}_{Z_i = Z_j} f_{in}(||X_i - X_j||) + \mathbf{1}_{Z_i \neq Z_j} f_{out}(||X_i - X_j||)$. Using the same random elements, we connect any $i < j \in \mathbb{N}$ by an edge in graph $I$ if $U_{ij} \in [f_{out}(||X_i - X_j||), f_{in}(||X_i - X_j||)]$. We denote by $i \sim_{I} j$ the event that points $i$ and $j$ are connected by an edge in $I$. Hence the graphs $I$ and $G$ are coupled and built on the same probability space using the same set of random elements. For each $i \in \mathbb{N}$, we denote by $V_{I}(i) \subseteq \mathbb{N}$  the random subset of the nodes contained in the connected component of node $i$ in graph $I$. Note that the information graph $I \stackrel{(d)}{=} H_{\lambda,f_{in}(\cdot) - f_{out}(\cdot),d}$, i.e. the $I$ graph we constructed is equal in distribution to the graph of the Poisson Random Connection model with vertex set forming a PPP of intensity $\lambda$ and connecting any two vertices at distance $r$ away with probability $f_{in}(r) - f_{out}(r)$ independently of everything else. This equality in distribution follows from the fact that $\{\{U_{kl}\}_{l > k}\}_{k \in \mathbb{N}}$ is an i.i.d. uniform $[0,1]$ sequence. The following structural lemma justifies the term information graph.

\begin{lemma}
	From the way we have coupled the construction of $G$ and $I$, we have 
	\begin{itemize}
		\item If $i \sim_{I} j$ and $i \sim_{G} j$, then $Z_i = Z_j$.
		\item If $i \sim_{I} j$ and $i \nsim_{G} j$, then $Z_i \neq Z_j$.
	\end{itemize}
	\label{lem:I_and_G}
\end{lemma}
\begin{proof}
	This follows from the following construction of $G$ and $I$ as follows. 
	\begin{itemize}
		\item $i \sim_{G} j$ if and only if $U_{ij} \leq f_{in}(||X_i - X_j||)\mathbf{1}_{Z_i = Z_j} + f_{out}(||X_i - X_j||)\mathbf{1}_{Z_i \neq Z_j}$. 
		\item $i \sim_{I} j$ if and only if $U_{ij} \in [f_{out}(||X_i - X_j||), f_{in}(||X_i - X_j||)]$.
		\item $\forall r \geq 0$, $f_{in}(r) \geq f_{out}(r)$. 
	\end{itemize}
	The lemma follows since the $\{U_{ij}\}_{0 \leq i <j}$ are the same with which we build both the random graphs $G$ and $I$. 
\end{proof}

We can iterate the above lemma from edges to connected components of $I$ which forms a crucial structural lemma. 

\begin{lemma}
	For all $j \in \mathbb{N}$, conditional on $G,\phi,I$, there are exactly two possible sequences $(Z_k)_{k \in V_I(j)}$ which are complements of each other that are consistent in the sense of Lemma \ref{lem:I_and_G} with the observed data $G,\phi$ and $I$.
	\label{lem:two_possible}
\end{lemma}

The proof of this follows from Lemma \ref{lem:I_and_G} and an induction argument. The proof can be found in Appendix \ref{appendix_proof_2possible}. We now present the main probabilistic observation in the sequel in Lemma \ref{lem:cond_indep_2} which essentially states that the community labels on disconnected components are independent. To do so, recall the definition that for any $i \in \mathbb{N}$, $V_I(i)$ denotes the set of nodes in the same connected component as node $i$ in graph $I$.


\begin{lemma}
	For all $\lambda > 0$, on the event $\{ |V_I(0)| < \infty\}$, 
	\begin{align*}
	\mathbb{P}^{0}\left[Z_0 = +1 \bigg| G, \{ \{U_{kl}\}_{l > k}\}_{k \in \mathbb{N}\cup \{0\}} , \phi, \{Z_i \}_{ i \in V_I^{\complement}{(0)} }\right] = \frac{1}{2}  \text{   a.s. }
	\end{align*}
	\label{lem:cond_indep_2}

\end{lemma}
\begin{proof}

	From Lemma \ref{lem:two_possible}, we know that conditionally on $\phi,G,I$, there are exactly two possible sequences $\{Z_k\}_{k \in V_I(0)}$ that are consistent with the observed data in the sense of Lemma \ref{lem:I_and_G}.  Denote  these two sequences by $\mathbf{s}$ and $\mathbf{s}^{\complement}$. It suffices to show that conditionally on $\phi,\{\{U_{kl}\}_{l > k}\}_{k \in \mathbb{N}\cup \{0\}},G$ and $\{Z_k\}_{k \in V_{I}^{c}(0)}$, the two sequences $\mathbf{s}$ and $\mathbf{s}^{\complement}$ are equally likely. We will denote by $g$  the realization of the random graph $G$. To conclude the lemma, we use  Bayes' conditional rule as follows.


	\begin{align}
	&\mathbb{P}^{0}_{\phi}[(Z_k)_{k \in V_I(0)}  =  \mathbf{s} \vert \phi, \{\{U_{kl}\}_{l > k}\}_{k \in \mathbb{N}\cup \{0\}}, (Z_k)_{k \in V_{I}^{c}(0)}, G = g]  \nonumber \\ & = \frac{ \mathbb{P}^{0}_{\phi}[G=g \vert \phi, \{\{U_{kl}\}_{l > k}\}_{k \in \mathbb{N}\cup \{0\}}, (Z_k)_{k \in V_{I}^{c}(0)}, (Z_k)_{k \in V_I(0)} = \mathbf{s}  ] }{  \mathbb{P}^{0}_{\phi}[G=g \vert \phi, \{\{U_{kl}\}_{l > k}\}_{k \in \mathbb{N}\cup \{0\}}, (Z_k)_{k \in V_{I}^{\complement}(0)}  ] }   \nonumber \\&\mathbb{P}^{0}_{\phi}[(Z_k)_{k \in V_I(0)} = \mathbf{s}  \vert \phi, \{\{U_{kl}\}_{l > k}\}_{k \in \mathbb{N}\cup \{0\}} ,(Z_k)_{k \in V_{I}^{c}(0)} ], \nonumber \\
	& \stackrel{(a)}{=} \frac{1 }{  \mathbb{P}^{0}_{\phi}[G=g \vert \phi, \{\{U_{kl}\}_{l > k}\}_{k \in \mathbb{N}\cup \{0\}}, (Z_k)_{k \in V_{I}^{\complement}(0)}  ] } \left( \frac{1}{2} \right)^{|V_I(0)|}, \text{   a.s. on the event } \{ |V_I(0)| < \infty\}\\
		& \stackrel{(b)}{=} \frac{1 }{ \sum_{\mathbf{l}}  \mathbb{P}^{0}_{\phi}[G=g, (Z_k)_{k \in V_I(0)} = \mathbf{l}\vert \phi, \{\{U_{kl}\}_{l > k}\}_{k \in \mathbb{N}\cup \{0\}}, (Z_k)_{k \in V_{I}^{\complement}(0)}  ] } \left( \frac{1}{2} \right)^{|V_I(0)|}, \text{   a.s. on } \{ |V_I(0)| < \infty\}\\
	& \stackrel{(c)}{=} \frac{1}{2} \text{   a.s. on the event } \{ |V_I(0)| < \infty\}.
	\label{eqn:Bayes_rule}
	\end{align}
	
	The first equality follows from rewriting the events using  Baye's conditional rule. In the rest of the proof, we  justify steps $(a)$ to $(c)$. We  prove the equalities and also justify that one can apply  conditional Baye's  rule without worrying about the $0$ by $0$ situation almost-surely.
	\\


	Note that conditionally on $\phi, \{\{U_{kl}\}_{l > k}\}_{k \in \mathbb{N}\cup \{0\}} , \{Z_k\}_{k \in \mathbb{N}}$, the graph $G$ is fixed and deterministic. Thus, the numerator in step $(a)$ is $1$ almost-surely.
	This follows from Lemma \ref{lem:two_possible} which states that $g$ is consistent with the data $(\phi,\{\{U_{kl}\}_{l > k}\}_{k \in \mathbb{N}\cup \{0\}}, (Z_k)_{k \in V_{I}^{\complement}(0)} )$ if $(Z_k)_{k \in V_I(0)} = \mathbf{s}$ or $\mathbf{s}^{\complement}$. Furthermore, the process $(Z_k)_{k \in \mathbb{N}}$ is an i.i.d. sequence independent of everything else. Hence, given any random finite subset $A \in \mathbb{N} $ independent of $(Z_k)_{k \in \mathbb{N}}$, the labels $(Z_k)_{k \in A}$ are uniform over $\{-1,1\}^{|A|}$. Now, since $|V_I(0)| < \infty$, and $V_I(0)$ is a function of $(\phi,\{\{U_{kl}\}_{l > k}\}_{k \in \mathbb{N}\cup \{0\}})$ which is independent of $(Z_k)_{k \in \mathbb{N}}$, it follows that,  on the event $\{ |V_I(0)| < \infty\}$,
	\begin{align}
	\mathbb{P}^{0}_{\phi}[(Z_k)_{k \in V_I(0)} = \mathbf{s}  \vert \phi, \{\{U_{kl}\}_{l > k}\}_{k \in \mathbb{N}\cup \{0\}} ,(Z_k)_{k \in V_{I}^{\complement}(0)} ] = \left( \frac{1}{2} \right)^{|V_I(0)|} \text{   a.s. }
	\label{eqn:lem_bayes_pre}
	\end{align}
	Moreover, the above expression is non-zero almost surely since $|V_I(0)| < \infty$. This justifies step $(a)$. To conclude the proof, it suffices to show that  on the event  $\{ |V_I(0)| < \infty\}$,
	\begin{align}
	\sum_{\mathbf{l}} \mathbb{P}^{0}_{\phi}[G=g,  (Z_k)_{k \in V_I(0)} = \mathbf{l}  \vert \phi, \{\{U_{kl}\}_{l > k}\}_{k \in \mathbb{N}\cup \{0\}}, (Z_k)_{k \in V_{I}^{\complement}(0)}  ] = 2 \left( \frac{1}{2} \right)^{|V_I(0)|} \text{   a.s } .
	\label{eqn:lem_bayes_rule}
	\end{align}
	This will conclude the proof  by noticing that the above expression is non-zero almost-surely. 
	\\

	Observe that the summation in Equation (\ref{eqn:lem_bayes_rule}) is over the various community labels $(Z_k)_{k \in V_I(0)}$. Thus, the summation is over the $2^{|V_I(0)|}$ different choices for $(Z_k)_{k \in V_{I}(0)}$. However, given $\phi$ and $\{\{U_{kl}\}_{l > k}\}_{k \in \mathbb{N}\cup \{0\}}$, one can construct the $I$ graph. Then, Lemma \ref{lem:two_possible}  states that the total number of possible choices for the labels $(Z_k)_{k \in V_{I}(0)}$ is now only two, which  we denoted by $\mathbf{s}$  and $\mathbf{s}^{\complement}$ in this proof. However, again from Lemma \ref{lem:two_possible}, conditionally on those two sequences, the graph constructed from the data  $(\phi, \{\{U_{kl}\}_{l > k}\}_{k \in \mathbb{N}\cup \{0\}}, (Z_k)_{k \in V_{I}^{c}(0)}, (Z_k)_{k \in V_I(0)} = \mathbf{s})$ and from the data  $( \phi, \{\{U_{kl}\}_{l > k}\}_{k \in \mathbb{N}\cup \{0\}}, (Z_k)_{k \in V_{I}^{c}(0)}, (Z_k)_{k \in V_I(0)} = \mathbf{s}^{\complement})$ is $g$, the observed graph. Hence, the proof of the claim follows from Equation (\ref{eqn:lem_bayes_pre}).

\end{proof}

The following is an immediate corollary of the definition of conditional expectation.
\begin{corollary}
	For all events $A \in \sigma(G, \phi, \{\{U_{kl}\}_{l > k}\}_{k \in \mathbb{N}\cup \{0\}}, (Z_k)_{k \in V_{I}^{\complement}(0)} )$,  we have $\mathbb{E}^{0}[\mathbf{1}_{E}\mathbf{1}_A \mathbf{1}_{Z_0 = +1}] = \mathbb{E}^{0}[\mathbf{1}_{E}\mathbf{1}_A \mathbf{1}_{Z_0 = -1}] = \frac{1}{2} \mathbb{E}^{0}[\mathbf{1}_{E}\mathbf{1}_A ]$, where $E$ is the event that $V_I(0)$ is finite.
	\label{cor:bayes_indep}
\end{corollary}

\subsection{Proof of Theorem \ref{thm:main_lower_bound}}

We are now ready to conclude the proof of  Theorem \ref{thm:main_lower_bound}. Notice that since $\tau_{r} \in \{-1,+1\}$, we can represent it as $\tau_{r} = \mathbf{1}_A - \mathbf{1}_{A^{\complement}}$, for some $A \in \sigma(({G}^{(r)} ,\phi^{(r)}))$. Hence, we have
\begin{align}
\sup_{\tau_{r} \in \sigma(({G}^{(r)} ,\phi^{(r)}))} \mathbb{P}^{0}_{\phi} [\tau_{n}^{(\delta)} = Z_0] = \sup_{A \in \sigma(({G}^{(r)} ,\phi^{(r)}))} \mathbb{E}^{0}_{\phi} [\mathbf{1}_{A}\mathbf{1}_{Z_0 = +1} + \mathbf{1}_{A^{\complement}} \mathbf{1}_{Z_0 = -1}].
\end{align}
For every $m \in \mathbb{N}$, denote by  $E_m$  the event that $C_I(0) \subseteq B_m$, i.e. the event that the connected component of the point at the origin in $I$ is contained in the set $B_m$.  The sets $E_m$ are non-decreasing. Moreover, from the definition of  percolation, $\mathbb{P}^{0}[\cup_{m \in \mathbb{N}}E_m] = \lim_{m \rightarrow \infty} \mathbb{P}^{0}[E_m] = 1- \theta(H_{\lambda,f_{in}(\cdot)-f_{out}(\cdot),d})$. Let $r \in \mathbb{R}$ be arbitrary, and condition on the event $E_r$. We have,

\begin{align*}
\sup_{A \in \sigma(({G}^{(r)} ,\phi^{(r)}))} \mathbb{E}^{0}& [\mathbf{1}_{A}\mathbf{1}_{Z_0 = +1} + \mathbf{1}_{A^{\complement}} \mathbf{1}_{Z_0 = -1}] = \\ \sup_{A \in \sigma((({G}^{(r)} ,\phi^{(r)}))} \mathbb{E}^{0}& [ \mathbf{1}_{E_r}\left(\mathbf{1}_{A}\mathbf{1}_{Z_0 = +1} + \mathbf{1}_{A^{\complement}} \mathbf{1}_{Z_0  =-1}\right)] + \mathbb{E}^{0} [ \mathbf{1}_{E^{\complement}_{r}}\left(\mathbf{1}_{A}\mathbf{1}_{Z_0 = +1} + \mathbf{1}_{A^{\complement}} \mathbf{1}_{Z_0 = -1}\right)] \\&
\leq \sup_{A \in \sigma((({G}^{(r)} ,\phi^{(r)})))} \mathbb{E}^{0} [ \mathbf{1}_{E_r}\left(\mathbf{1}_{A}\mathbf{1}_{Z_0 = +1} + \mathbf{1}_{A^{\complement}} \mathbf{1}_{Z_0  =-1}\right)] + \mathbb{P}^{0}[E^{\complement}_{r}] \\
&\stackrel{(a)}{\leq} \sup_{A \in \sigma(({G}^{(r)} ,\phi^{(r)}, \{\{U_{kl}\}_{l > k}\}_{k \in \mathbb{N}\cup \{0\}}  )}\mathbb{E}^{0} [ \mathbf{1}_{E_r}\left(\mathbf{1}_{A}\mathbf{1}_{Z_0 = +1} + \mathbf{1}_{A^{\complement}} \mathbf{1}_{Z_0  =-1}\right)] + \mathbb{P}^{0}[E^{\complement}_{r}]  \\
&\stackrel{(b)}{\leq} \sup_{A \in \sigma(G, \phi, \{\{U_{kl}\}_{l > k}\}_{k \in \mathbb{N}\cup \{0\}}, (Z_k)_{k \in V_{I}^{\complement}(0)} )}\mathbb{E}^{0} [ \mathbf{1}_{E_r}\left(\mathbf{1}_{A}\mathbf{1}_{Z_0 = +1} + \mathbf{1}_{A^{\complement}} \mathbf{1}_{Z_0  =-1}\right)] + \mathbb{P}^{0}[E^{\complement}_{r}]  \\
& \stackrel{(c)}{=} \sup_{A \in \sigma(G, \phi, \{\{U_{kl}\}_{l > k}\}_{k \in \mathbb{N}\cup \{0\}}, (Z_k)_{k \in V_{I}^{\complement}(0)} )}\frac{1}{2} \mathbb{E}^{0}[\mathbf{1}_A \mathbf{1}_{E_r}] + \frac{1}{2} \mathbb{E}^{0}[\mathbf{1}_{A^{\complement}} \mathbf{1}_{E_r}] + \mathbb{P}^{0}[E^{\complement}_{r}] \\
&= \frac{1}{2} \mathbb{P}^{0}[E_r] + \mathbb{P}^{0}[E^{\complement}_{r}].
\end{align*}
Step $(a)$ follows from enlarging the sigma algebra over which we are searching for a solution. Step $(b)$ follows from the fact that on the event $E_r$, $V_I(0) \subseteq B(0,r)$. Thus, revealing more labels will only preserve the inequality. Step $(c)$ follows from Corollary \ref{cor:bayes_indep}. Now, since the sets $E_r$ are non-decreasing , we get the theorem by taking a limit as $r$ goes to infinity on both sides, i.e.
\begin{align}
\lim_{r \rightarrow \infty}\sup_{A \in \sigma(({G}^{(r)} ,\phi^{(r)}))} \mathbb{E}^{0} [\mathbf{1}_{A}\mathbf{1}_{Z_0 = +1} + \mathbf{1}_{A^{\complement}} \mathbf{1}_{Z_0 = -1}] &\leq \lim_{r \rightarrow \infty} \frac{1}{2} \mathbb{P}^{0}[E_r] + \mathbb{P}^{0}[E^{\complement}_{r}] \nonumber \\
&=\frac{1}{2}\mathbb{P}^{0}[E] + \mathbb{P}^{0}[E^{\complement}] \nonumber \\
&= \frac{1}{2} (1 + \theta(H_{\lambda,f_{in}(\cdot) - f_{out}(\cdot),d})).
\end{align}
The limit on the LHS exists from Proposition \ref{prop:monotone} and the limit on the RHS exists since $E_r$ are non-decreasing events.

\section{ Identifiability of the Partition and Proof of Theorem \ref{thm:identifiability}}
\label{sec:hypo}

The key technical tool is the ergodicity of the PPP which is summarized in the following lemma. We need to set some notation that are needed to state the lemma. Denote by $\mathbb{M}_{\Xi}(\mathbb{R}^d)$  the set of all 
`marked' point processes on $\mathbb{R}^d$ where each point is assigned a `mark' from the measure space $\Xi$, with its associated sigma-algebra. The set $\mathbb{M}_{\Xi}(\mathbb{R}^d)$ is a Polish space, has a natural topology and hence an associated sigma-algebra (see \cite{daley}). Denote by $\theta : \mathbb{R}^d \times \mathbb{M}_{\Xi}(\mathbb{R}^d) \rightarrow \mathbb{M}_{\Xi}(\mathbb{R}^d)$ the `shift' operator  which is a measurable function where $\theta(x,\psi)$
retains the same marks but translates all points of $\psi$ by a vector $x$.

\begin{lemma}
	Let $C_n \subset \mathbb{R}^d$ be a sequence of $L_p$, $p \in [1,\infty]$ balls centered at the origin with radius going to infinity as $n \rightarrow \infty$. Let $f: \mathbb{M}_{\Xi}(\mathbb{R}^{d}) \rightarrow \mathbb{R}_{+}$ be a measurable function  such that $\mathbb{E}^{0}_{\phi}[f] < \infty$. Then, the following limit exists :
	\begin{align}
	\lim_{n \rightarrow \infty}\frac{\sum_{i \in \mathbb{N}} \mathbf{1}_{X_i \in C_n} f \circ \theta(X_i,G)}{ \sum_{i \in \mathbb{N}} \mathbf{1}_{X_i \in C_n}}  = \mathbb{E}^{0}[f] \text {   } \mathbb{P} \text{    a.s. }.
	\end{align}
	\label{lem:ergodic_thm}
\end{lemma}
This is a standard fact about ergodic point processes (for ex. \cite{daley}).
We now prove Theorem \ref{thm:identifiability}.

\begin{proof}

First if $g(\cdot) \neq \frac{f_{in}(\cdot)+f_{out}(\cdot)}{2}$ Lebesgue almost everywhere, then $G$ and $H_{\lambda,g(\cdot),d}$ are mutually singular and this can be seen through the following elementary argument. Fix some $L < \infty $ such that $\int_{x \in \mathbb{R}^d: ||x|| \leq L} g(||x||)dx \neq \int_{x \in \mathbb{R}^d: ||x|| \leq L} ((f_{in}(x) + f_{out}(||x||))/2) dx$ are both finite. Such a $L$ exists since $g(\cdot) \neq \frac{f_{in}(\cdot)+f_{out}(\cdot)}{2}$. Now, we apply the ergodic theorem, where every node $i \in \mathbb{N}$ of $\phi$ is equipped with a mark $\Xi_i \in \mathbb{N}$, which denotes the number of graph neighbors of node $i$  at a distance of at-most $L$ from $X_i$, i.e. $\Xi_i = | \{j \in \mathbb{N} \setminus\{i\} i \sim_{G}j, ||X_i - X_j|| \leq L \}|$. Thus, the Ergodic theorem implies that  the measure induced by $G$ will be concentrated on the set  
\begin{align*}
\left\{ g \in \mathbb{M}_{\mathcal{G}}(\mathbb{R}^d) : \lim_{n \rightarrow \infty}\frac{\sum_{i,j \in \mathbb{N}} \mathbf{1}_{||X_i|| \leq n, ||X_i - X_j|| \leq L }  \mathbf{1}_{i \sim_g j } }{\sum_{i,j \in \mathbb{N}} \mathbf{1}_{||X_i|| \leq n, ||X_i - X_j|| \leq L } } = \int_{x \in \mathbb{R}^d: ||x|| \leq L} (f_{in}(x) + f_{out}(||x||))/2 dx  \right\},
\end{align*}

 while the measure induced by  $H_{\lambda,g(\cdot),d}$ will be concentrated on the set
 \begin{align*}
 \left\{ g \in \mathbb{M}_{\mathcal{G}}(\mathbb{R}^d) : \lim_{n \rightarrow \infty}\frac{\sum_{i,j \in \mathbb{N}} \mathbf{1}_{||X_i|| \leq n, ||X_i - X_j|| \leq L }  \mathbf{1}_{i \sim_g j } }{\sum_{i,j \in \mathbb{N}} \mathbf{1}_{||X_i|| \leq n, ||X_i - X_j|| \leq L } } = \int_{x \in \mathbb{R}^d: ||x|| \leq L} g(||x||) dx  \right\}.
 \end{align*}

Thus, the only case to consider is the one where $g(\cdot) = (f_{in}(\cdot) + f_{out}(\cdot)/2$ Lebesgue almost-everywhere. From linearity of expectation, the average degree of any node $i \in \mathbb{N}$ in both  graphs $G$ and $H_{(\lambda,\frac{f_{in}(\cdot)+f_{out}(\cdot)}{2},d)}$ is the same and equal to $(\lambda/2)\int_{x \in \mathbb{R}^d} (f_{in}(||x||) + f_{out}(||x||))dx$ and thus empirical average of the degree does not help. However, we see that the triangle profiles differ in the two models which we leverage to prove the Theorem. 
\\

 For ease of notation, we denote by $H := H_{\lambda, (f_{in}(\cdot) + f_{out}(\cdot))/2,d}$. Define
	
	\begin{align*}
	\Delta = \mathbb{E}^{0} \left[ \sum_{x\neq y \neq 0} h(x,y) \mathbf{1}( (0,x) \in E, (0,y) \in E, (x,y) \in E ) \right].
	\end{align*}
	Denote by $\Delta_G$ and $\Delta_H$  the value of the above expression if the underlying graphs were $G$ and $H$ respectively.
	From the moment measure expansion of PPPs \cite{stoyan}, we get that

	\begin{multline*}
	\Delta_{G} = \int_{x \in \mathbb{R}^d} \int_{y \in \mathbb{R}^d} h(x,y)  (  f_{in}(||x-y||) \left( \frac{f_{in}(||x||) f_{in}(||y||) + f_{out}(||x||) f_{out}(||y||)   }{4}      \right ) +   \\ f_{out}(||x-y||)\left( \frac{f_{in}(||x||) f_{out}(||y||) + f_{out}(||x||) f_{in}(||y||)   }{4} \right) ) \lambda^2 dx dy,
	\end{multline*}
	and
		\begin{multline*}
	\Delta_{H} = \int_{x \in \mathbb{R}^d} \int_{y \in \mathbb{R}^d} h(x,y) \left( \frac{f_{in}(||x||) + f_{out}(||x||)}{2}  \right) \left( \frac{f_{in}(||y||) + f_{out}(||y||)}{2}  \right) \\ \left( \frac{f_{in}(||x-y||) + f_{out}(||x-y||)}{2} \right)  \lambda^2 dx dy.
	\end{multline*}
	
	Now observe that 
	\begin{multline*}
	\Delta_{G} - \Delta_{H} = \int_{x \in \mathbb{R}^d} \int_{y \in \mathbb{R}^d} h(x,y) \frac{1}{2} \left( f_{in}(||x||) - f_{out}(||x||)    \right) \left(  f_{in}(||y||) - f_{out}(||y||) \right) \\ \left( f_{in}(||x-y||) - f_{out}(||x-y||) \right) \lambda^2 dx dy.
	\end{multline*}
	From the fact that  $f_{in}(r) \geq f_{out}(r)$ and $f_{in}(r)$ is not equal to $f_{out}(r)$ Lebesgue almost everywhere, there exists a positive bounded function $h(x,y)$ such that $0 \leq \Delta_H < \Delta_G < \infty$. Choose one such test function $h(\cdot,\cdot)$,  for ex.  $h(x,y) = \mathbf{1}_{||x|| \leq R}\mathbf{1}_{||y|| \leq R} \mathbf{1}_{f_{in}(||x||) > f_{out}(||x||)}\mathbf{1}_{f_{in}(||y||) > f_{out}(||y||)} \mathbf{1}_{f_{in}(||x-y||) > f_{out}(||x-y||)}$ and consider the following estimator:

	\begin{algorithm}[H]
	\caption{Detect-Partitions}
	
	 Given the data, i.e.  locations of nodes and the graph, pick a $L$ large enough and compute 
	\begin{align*}
	\Delta^{(L)} := \frac{ \sum_{i \in \mathbb{N}} \mathbf{1}(|X_i| \leq L) \tilde{h}(X_i)    }{  \sum_{i \in \mathbb{N}} \mathbf{1}(|X_i| \leq L)  },
	\end{align*}
	where $\tilde{h}(X_i) = \sum_{j,k \in \mathbb{N}, j \neq k \neq i}  h( X_i - X_j, X_i - X_k) \mathbf{1}( i \sim_{G} j, i \sim_{G} k, j \sim_{G} k)$. 
	
	\end{algorithm}
From  ergodicity, we know that $\lim_{L \rightarrow \infty} \Delta^{(L)} = \Delta_{G}$, $\mathbb{P}$ almost-surely if the data is the block model graph or $\lim_{L \rightarrow \infty} \Delta^{(L)} = \Delta_{H}$ $\mathbb{P}$ almost surely if the graph is drawn according to the null model. We can apply the ergodic theorem since a spatial random graph is a marked point process (as described in Section \ref{subsec:graph_model}). Thus, the measures induced by $(\phi,G)$ and $(\phi, H_{\lambda, \frac{f_{in}(\cdot) + f_{out}(\cdot)}{2} , d})$ are mutually singular. Moreover, the above algorithm when tested on the finite data $(G_n,\phi_n)$ runs in  time proportional to $\lambda n$ with the multiplicative constants depending on the connection functions $f_{in}(\cdot)$ and $f_{out}(\cdot)$ with success probability $1 - o_n(1)$ with the $o_n(1)$ term depending on the function $h(\cdot,\cdot)$ chosen.

\end{proof}

We investigated the singularity of measures in order to understand the question of distinguishability of the planted partition model. This is a hypothesis testing problem of whether  the data (graph and spatial locations) is drawn from the distribution of $G$ or from the distribution of $H_{\lambda,g(\cdot),d}$ with a probability of success exceeding a half given a uniform prior over the models. This problem is in some sense easier than Community Detection, since this asks for asserting whether a partition exists or not, which intuitively should be simpler than finding the partition. Indeed, we show this {in our model} by proving that the distinguishability problem is trivially solvable while community detection undergoes a phase transition and is solvable only under certain regimes. In the general sparse SBM however, the equivalence between distinguishability and community detection is only conjectured and not yet proven (\cite{decelle},\cite{neeman_distinction}).  

\section{The Exact-Recovery Problem}
\label{sec:exact_recovery}

In this section, we provide a lower bound for the exact-recovery problem as stated in Definition \ref{defn:comm_det_er}. Recall that for the exact-recovery case, we equipped the set $B_n$ with the torridal metric rather than the usual Euclidean metric. This is done mainly to simplify the presentation of the results. Nevertheless, one could establish identical results to the case when the set $B_n$ equipped with the Euclidean metric albeit with significantly more heavier notation to handle the `edge effects'. We note that for any $x := (x_1,\cdots,x_d), y:= (y_1,\cdots, y_d) \in B_n$, the torroidal distance is given by $||x-y||_{\mathcal{T}_n} = || ( \min(|x_1-y_1|, n^{1/d} - |x_1-y_1 |),\cdots, \min(|x_d-y_d|, {n}^{1/d} - |x_d-y_d |)  )||$, where $||.||$ is the Euclidean norm on $\mathbb{R}^d$. The key result we establish about this model is a lower bound or a necessary condition to perform Exact-Recovery. We then observe that the {\ttfamily GBG} algorithm presented earlier achieves Exact-Recovery, if the intensity $\lambda$ is sufficiently high, thereby establishing a phase-transition. We establish the lower bound by exploiting recent advances in the understanding of error exponents for the probability of error in distinguishing between two Poisson random vectors developed in \cite{abbe_sandon_ch}. Furthermore, we also conjecture this necessary condition to also be sufficient. The exact proof of this conjecture is left open in this paper. Another important note is that since we view the set $B_n$ as a torus in this section, the notion of Palm probability needs more care in stating. In the most general sense, we must employ the moment measure expansions of a Poisson Point Process on compact topological groups to discuss the Palm measure on the torus. We provide a justification of the moment measure equations we use in this section in the Appendix \ref{appendix_palm}.

\subsection{Lower Bound for Exact-Recovery}

In this section, we prove Theorem \ref{thm:er_main_formula}. To do so, we first give a general result in Proposition \ref{prop:er_nec_cond} using the Genie aided argument introduced in \cite{abbe_exact}, and then subsequently prove an explicit formula for the lower bound stated in Theorem \ref{thm:er_main_formula} using the large deviation results of \cite{abbe_sandon_ch} for  hypothesis testing between Poisson random vectors.
To state the result, we define a notion of \emph{Flip-Bad} for nodes. Roughly speaking, we say a node $i \in \{1,\cdots,N_n\}$ is Flip-Bad in $G_n$, if on flipping the community label of node $i$ from $Z_i$ to $-Z_i$, the likelihood of the observed data $(G_n,\phi_n)$ increases. More formally, given data $(\phi_n,G_n,(Z_i)_{i = 1}^{N_n})$, the likelihood is defined as 
\begin{multline*}
\mathcal{L}(\phi_n,G_n,(Z_i)_{i = 1}^{N_n}) := \frac{e^{-\lambda n} (\lambda n)^{N_n}}{N_n !} \left(\frac{1}{2n}\right)^{N_n} \prod_{1 \leq i < j \leq N_n: Z_i = Z_j, i\sim j} f_{in}^{(n)}(||X_i - X_j) \\\prod_{1 \leq i < j \leq N_n: Z_i \neq Z_j i \sim j} f_{out}^{(n)}(||X_i - X_j||)  \prod_{1 \leq i < j \leq N_n, Z_i = Z_j, i\nsim j} (1 - f_{in}^{(n)}(||X_i - X_j||) ) \\\prod_{1 \leq i < j \leq N_n, Z_i \neq Z_j, i \nsim j} (1 - f_{out}^{(n)}(||X_i - X_j||)),
\end{multline*}
with the empty product being equal to $1$. Thus, it is easy to see that for every $n$, $0 < \mathcal{L}((\phi_n,G_n,(Z_i)_{i = 1}^{N_n})) \\ \leq 1$ almost-surely, since for every $n$, $N_n$ is finite almost surely. We say a node $j$ is \emph{Flip-Bad in $G_n$} if 
\begin{align*}
\mathcal{L}((\phi_n,G_n,(Z_i)_{i = 1}^{N_n})) \leq \mathcal{L}((\phi_n,G_n,(\tilde{Z}_i^{(j)})_{i = 1}^{N_n})),
\end{align*}
where $\tilde{Z}_i^{(j)} = Z_i$ if $i \neq j$ and $\tilde{Z}_j^{(j)} = -Z_j$. Thus, we use the term Flip-Bad, where a node is `bad' if on flipping its community label, the observed data becomes more likely. We use this definition to reason about the maximum-likelihood estimator.  More formally, for each node $i \in [1,N_n]$, the Maximum-Likelihood (ML) estimate of the community label is denoted as $\hat{Z}_i$  and satisfies

\begin{align}
(\hat{Z}_i)_{i \in \{1,\cdots,N_n\}} &= \arg\max_{z \in \{-1,+1\}^{N_n}} \mathbb{P}[(Z_i)_{i \in \{1,\cdots,N_n\} }= z \vert \phi_n,G_n] \nonumber \\
&= \arg\max_{z \in \{-1,+1\}^{N_n}}\mathcal{L}((\phi_n,G_n,z) 
\label{eqn:ML_estimator}
\end{align}
In words, it is the optimal estimate for the community labels given the observed data $\phi_n$ and $G_n$, where optimality refers to the fact that it minimizes the probability that the vector is $(\hat{Z}_i)_{i \in [1,N_n]}$ is different from the ground truth $(Z_i)_{i \in [1,N_n]}$ among all possible estimators. More formally, the ML estimator $(\hat{Z}_i)_{i \in [1,N_n]} $ in Equation (\ref{eqn:ML_estimator}) satisfies
\begin{align*}
\mathbb{P}[(\hat{Z}_i)_{i \in \{1,\cdots,N_n\}} \neq (Z_i)_{i \in \{1,\cdots,N_n\}}] = \inf_{(\tilde{Z}_i)_{i \in \{1,\cdots,N_n\}} = \mathcal{A}(\phi_n,G_n)}\mathbb{P}[(\tilde{Z}_i)_{i \in \{1,\cdots,N_n\}} \neq (Z_i)_{i \in \{1,\cdots,N_n\}}],
\end{align*}
where the infimum is over all measurable functions $\mathcal{A}$ of the data $\phi_n$ and $G_n$. This asserts that the ML estimator for the community labels is the optimal estimator that minimizes probability of error. The following proposition gives a structural condition on the model parameters of $G_n$ to identify a sufficient condition when the ML estimator will fail.


\begin{proposition}
	If the model parameters of $G_n$ satisfies
\begin{align}
 &	\limsup_{n \rightarrow \infty} \frac{ \int_{y \in B_n} \mathbb{E}^{0,y}[ \mathbf{1}_{0 \text{ is Flip-Bad in } G_n \cup \{0,y\}} \mathbf{1}_{y \text{ is Flip-Bad in }G_n \in \cup\{0,y\}}]  m_{n,d}(dy)}{    n \mathbb{E}^{0}[\mathbf{1}_{0 \text{ is flip-bad in }G_n \cup \{0\} }]^2} \leq 1,
 \label{eqn:er_corr_decay_def} \\
& \lim_{n \rightarrow \infty}n \mathbb{E}^0[\mathbf{1}_{0 \text{ is flip-bad in } G_n \cup \{0\}}] = \infty \label{eqn:er_nec_cond_inf}
\end{align}
	where $m_{n,d}$ is the Haar measure on the torus $B_n$,
	then Exact-Recovery is not solvable.
	\label{prop:er_nec_cond}
\end{proposition}
 The condition in Equation (\ref{eqn:er_corr_decay_def}) states that the event of two `far-away' nodes being Flip-Bad are asymptotically uncorrelated. Such statements are true for instance if the functions $f_{in}^{(n)}(\cdot)$ and $f_{out}^{(n)}(\cdot)$ have support that is $o(n)$. For example, if the connection functions satisfy $f_{in}^{(n)}(r) = a_n \mathbf{1}_{r \leq R_n}$ and $f_{out}^{(n)}(r)
 = b_n \mathbf{1}_{r \leq R_n}$, for some $0 \leq b_n \leq a_n \leq 1$ and $R_n = o(n^{1/d})$, then they satisfy the condition. See for example also Lemma \ref{lem:er_con_func_good}. On the other hand, Equation (\ref{eqn:er_corr_decay_def}) is satisfied even if the model is reduced to the classical SBM without geometry, i.e., any two nodes in $G_n$ are  connected only based on their community labels and not the location labels if the average degree is proportional to the logarithm of the population size.
\begin{proof}
	If there exists a node in $G_n$ that is Flip-Bad, then the Maximum Likelihood estimator for the community labels will not match the ground truth. We will show that if Equations (\ref{eqn:er_corr_decay_def}) and (\ref{eqn:er_nec_cond_inf}) hold, then $ \mathbb{P}[\exists i \in \{1,\cdots,N_n\}: i \text{ is Flip-Bad in }G_n] = 1 - o_n(1)$, which will conclude the proof of the result. 
	\\
	
	To do this, define by $Y_n := \sum_{i \in \{1,\cdots,N_n\}} \mathbf{1}_{i \text{ is Flip-Bad in }G_n}$. From the classical method of moments, it suffices to establish that $\limsup_{n \rightarrow \infty} \frac{\mathbb{E}[Y_n^2] }{\mathbb{E}[Y_n]^2} \leq 1$ and $\mathbb{E}[Y_n] \rightarrow \infty$ as $n \rightarrow \infty$. Indeed, if this were the case, then by Chebychev's inequality, we would have 
	\begin{align*}
	\mathbb{P}[Y_n = 0] \leq \frac{\mathbb{E}[Y_n^2] }{\mathbb{E}[Y_n]^2} - 1,
	\end{align*}
	which will converge to $0$. It remains to compute the first and second moment of $Y_n$. 
	\\
	
	From Campbell's theorem (see Appendix \ref{appendix_palm} for Palm probability on the torus), we have 
	\begin{align}
	\mathbb{E}[Y_n] &= \lambda \int_{y \in B_n} \mathbb{E}^{y}[\mathbf{1}_{y \text{ is Flip-Bad in }G_n \cup \{y\}}] m_{n,d}(dy) \nonumber\\
	&=  \lambda n \mathbb{E}^0[\mathbf{1}_{0 \text{ is Flip-bad in }G_n\cup \{0\} }],
	\label{eqn:er_first_moment}
	\end{align}
	where the second inequality follows from the symmetry in the torus.
	Thus from the hypothesis of the theorem in Equation (\ref{eqn:er_nec_cond_inf}), $\mathbb{E}[Y_n]$ converges to $\infty$. To compute the second moment, we use the factorial moment expansion of the Poisson Process as follows.
	\begin{align}
	\mathbb{E}[Y_n^2] &= \mathbb{E}[(\sum_{i \in [1,N_n]} \mathbf{1}_{i \text{ is Flip-Bad in }G_n})^2]\nonumber\\
	&= \mathbb{E}[Y_n] + \mathbb{E}[\sum_{i \neq j} \mathbf{1}_{i \text{ is Flip-Bad in } G_n} \mathbf{1}_{j \text{ is Flip-Bad in }G_n}]\nonumber\\
	& \stackrel{(a)}{=} \mathbb{E}[Y_n] + \lambda^2 \int_{x \in B_n} \int_{y \in B_n} \mathbb{E}^{x,y}[ \mathbf{1}_{x \text{ is Flip-Bad in } G_n \cup \{x,y\}} \mathbf{1}_{y \text{ is Flip-Bad in }G_n \in \cup\{x,y\}}] m_{n,d}(dx) m_{n,d}(dy)\nonumber \\
	& \stackrel{(b)}{=} \mathbb{E}[Y_n] + \lambda^2 n \int_{y \in B_n}\mathbb{E}^{0,y}[ \mathbf{1}_{0 \text{ is Flip-Bad in } G_n \cup \{0,y\}} \mathbf{1}_{y \text{ is Flip-Bad in }G_n \in \cup\{0,y\}}] m_{n,d}(dy),
	\label{eqn:er_prop_second_moment}
	\end{align} 
	where equality $(a)$ follows from the $2$nd order Moment Measure expansion of a Poisson process and $(b)$ follows from the symmetry in the torus. Thus, from Equations (\ref{eqn:er_first_moment}) and (\ref{eqn:er_prop_second_moment}) , we have,
	\begin{align}
	\frac{\mathbb{E}[Y_n^2]}{\mathbb{E}[Y_n]^2} = \frac{1}{\mathbb{E}[Y_n]} + \frac{\lambda^2\int_{y \in B_n}\mathbb{E}^{0,y}[ \mathbf{1}_{0 \text{ is Flip-Bad in } G_n \cup \{0,y\}} \mathbf{1}_{y \text{ is Flip-Bad in }G_n \in \cup\{0,y\}}] m_{n,d}(dy)}{ (\lambda n \mathbb{E}^0[\mathbf{1}_{0 \text{ is Flip-bad in }G_n\cup \{0\} }])^2}.
	\label{eqn:er_nec_cond_prop}
	\end{align}
	
	Thanks to the assumption on the connection functions in Equation (\ref{eqn:er_corr_decay_def}) and the first moment in Equation (\ref{eqn:er_nec_cond_prop}), we have that $\limsup_{n \rightarrow \infty} \frac{\mathbb{E}[Y_n^2]}{\mathbb{E}[Y_n]} \leq 1$.
\end{proof}

In the rest of the section, we consider the model given in Definition \ref{defn:gsbm} where the connection functions $f_{in}^{(n)}(\cdot)$ and $f_{out}^{(n)}(\cdot)$ take the form $f_{in}(r) = a \mathbf{1}_{r \leq \log(n)^{1/d}}$ and $f_{out}^{(n)}(r) = b \mathbf{1}_{r \leq \log(n)^{1/d}}$ for some $0 \leq b < a \leq 1$, to illustrate how one can use the previous proposition to obtain a closed form expression for the phase-transition threshold. We first prove Theorem \ref{thm:er_main_formula} to provide an exact necessary condition in terms of the model parameters $\lambda,a$ and $b$ for achieving exact-recovery. The proof of Theorem \ref{thm:er_main_formula} follows from the next two lemmas. Recall that, we denote by $\nu_d$ for all $d \in \mathbb{N}$ as the volume of the unit Euclidean ball in $d$ dimensions. 

\begin{lemma}
	For all $\lambda >0$, $d \in \mathbb{N}$ and $0 \leq b < a \leq 1$, if $G_n \sim \mathcal{G}(\lambda n,a,b,d)$, then \\ $\mathbb{E}^0[\mathbf{1}_{0 \text{ is Flip-Bad in }G_n \cup \{0\}}] =  e^{-\lambda \nu_d  \log(n)(1 - \sqrt{ab} - \sqrt{(1-a)(1-b)} -o(1)   )}$.
		\label{lem:er_ch_divergence}
\end{lemma}

\begin{lemma}
	For all $\lambda >0$, $d \in \mathbb{N}$ and $0 \leq b < a \leq 1$ such that $\lambda \nu_d  \log(n)(1 - \sqrt{ab} - \sqrt{(1-a)(1-b)} < 1$, the graph $G_n \sim \mathcal{G}(\lambda n,a,b,d)$ satisfies Equation (\ref{eqn:er_corr_decay_def}).
	\label{lem:er_con_func_good}
\end{lemma}

\begin{proof}  \emph{of Theorem \ref{thm:er_main_formula}}.\\
From Lemma \ref{lem:er_con_func_good}, we know that the connection functions satisfy the model assumption in Equation  (\ref{eqn:er_corr_decay_def}). From Lemma \ref{lem:er_ch_divergence}, we know that if $\lambda \nu_d (1 - \sqrt{ab} - \sqrt{(1-a)(1-b)}) < 1$, then  $\mathbb{E}^0[\mathbf{1}_{0 \text{ is flip-bad in } G_n \cup \{0\}}] = n^{-1+\delta - o(1)}$ for some $\delta >0$. Thus, $n \mathbb{E}^0[\mathbf{1}_{0 \text{ is flip-bad in } G_n \cup \{0\}}] = n^{\delta - o(1)}$ which converges to $\infty$ as $n$ goes to $\infty$. The proof is now complete thanks to Proposition \ref{prop:er_nec_cond}.
\end{proof}

We now provide the proofs of the associated lemmas.

\begin{proof} \emph{of Lemma \ref{lem:er_ch_divergence}}.\\
	This lemma is a corollary of Lemma $11$ proven in \cite{abbe_sandon_ch}, where the  error exponent for hypothesis testing between Poisson random vectors was established. 
To assess whether the node at $0$ is Flip-Bad, we need to decide given the locations and true community labels of the neighbors and non-neighbors of node at $0$, whether it belongs to community $+1$ or $-1$. However, since the connection functions have support of $\log(n)^{1/d}$ and events in disjoint regions of space are independent, it suffices to consider the neighbors and non-neighbors within the ball of radius $\log(n)^{1/d}$ around $0$. (Note from the model that there are no neighbors of the node at $0$ at a distance larger than $\log(n)^{1/d}$). The number of neighbors of nodes in the same commuity as $0$ is a Poisson random variable of mean $\lambda/2 \nu_d a\log(n)$ and the number of neighbors in the opposite community is another independent Poisson random variable of mean  $\lambda/2 \nu_d b\log(n)$. The independence follows from elementary independent thinning property of the PPP. Similarly, the number of non-neighbors in the same and opposite community as $0$ and within a distance of $\log(n)^{1/d}$ of $0$ are independent Poisson random variables of mean $\lambda/2 \nu_d (1-a)\log(n)$ and $\lambda/2 \nu_d (1-b)\log(n)$ respectively. This argument again follows from the independent thinning property of an independently marked Poisson Process. 
	\\
	
	Thus the probability of a node at $0$ being flip-bad is equal to the error made by an optimal hypothesis tester between two random vectors $(\lambda/2) \nu_d \log(n) (a,b,1-a,1-b)$ and $(\lambda/2) \nu_d \log(n) (b,a,1-b,1-a)$, given an uniform prior over the two models. Furthermore, the components of the observed random vector are independent. Thus, we are in a setting to  apply the CH-divergence theorem of \cite{abbe_sandon_ch} to characterize the error probability of the optimal hypothesis testing error. More precisely, applying Lemma $11$ from \cite{abbe_sandon_ch} the error probability in identifying between the two Poisson random vectors with uniform prior  $(\lambda/2) \nu_d \log(n) \mu $ and $(\lambda/2) \nu_d \log(n) \nu $ where the vectors $\mu:= (a,b,1-a,1-b)$ and $\nu :=  (b,a,1-b,1-a)$  is equal to $\mathbb{E}^0[\mathbf{1}_{0 \text{ is flip-bad in } G_n \cup \{0\}}]$ is given by
	\begin{align}
	\mathbb{E}^0[\mathbf{1}_{0 \text{ is flip-bad in } G_n \cup \{0\}}]:= n^{- (\lambda/2) \nu_d D_{+}(\mu,\nu) + o(1)},
	\label{eqn:er_error_prob_CH}
	\end{align}
	where the CH-Divergence $D_{+}(\mu,\nu)$ (\cite{abbe_sandon_ch}) is given by
	\begin{align}
	D_{+}(\mu,\nu) := \max_{t \in [0,1]} \sum_{x \in \mathcal{X}} (t \mu(x) + (1-t) \nu(x) - \mu(x)^t \nu(x)^{1-t}).
	\label{eqn:er_CH_Divergence}
	\end{align}
	Here $\mathcal{X} := [1,2,3,4]$, and $\mu(x)$ ($\nu(x)$) for $x \in \mathcal{X}$ refers to the $x$th component of the vector $\mu$ ($\nu$). Evaluating Equation (\ref{eqn:er_CH_Divergence}) yields that $D_{+}(\mu,\nu) = 2(1 - \sqrt{ab} - \sqrt{(1-a)(1-b)})$ with the maximum being achieved at $t=1/2$ due to symmetry in the vectors $\mu$ and $\nu$. Substituting Equation (\ref{eqn:er_CH_Divergence}) into (\ref{eqn:er_error_prob_CH}) yields the result.
\end{proof}

\begin{proof} \emph{of Lemma \ref{lem:er_con_func_good}}. \\

This lemma follows from some straightforward calculations exploiting the spatial independence across the Poisson process. To verify Equation (\ref{eqn:er_corr_decay_def}), consider the following chain of equations.
	\begin{align}
	&\int_{y \in B_n} \mathbb{E}^{0,y}[ \mathbf{1}_{0 \text{ is Flip-Bad in } G_n \cup \{0,y\}} \mathbf{1}_{y \text{ is Flip-Bad in }G_n \in \cup\{0,y\}}] m_{n,d}(dy) = \nonumber\\ & \int_{y \in B(0,2\log(n)^{1/d})} \mathbb{E}^{0,y}[ \mathbf{1}_{0 \text{ is Flip-Bad in } G_n \cup \{0,y\}} \mathbf{1}_{y \text{ is Flip-Bad in }G_n \in \cup\{0,y\}}] m_{n,d}(dy)  \nonumber\\& +\int_{y \in B_n \cap B(0,2\log(n)^{1/d})^{\complement}}  \mathbb{E}^{0,y}[ \mathbf{1}_{0 \text{ is Flip-Bad in } G_n \cup \{0,y\}} \mathbf{1}_{y \text{ is Flip-Bad in }G_n \in \cup\{0,y\}}] m_{n,d}(dy) \nonumber\\
	& \leq \int_{y \in B(0,2\log(n)^{1/d})} \mathbb{E}^{0,y}[ \mathbf{1}_{0 \text{ is Flip-Bad in } G_n \cup \{0,y\}} ]  m_{n,d}(dy)  \nonumber \\& + \int_{y \in B_n \cap B(0,2\log(n)^{1/d})^{\complement}}  \mathbb{E}^{0,y}[ \mathbf{1}_{0 \text{ is Flip-Bad in } G_n \cup \{0,y\}} \mathbf{1}_{y \text{ is Flip-Bad in }G_n \in \cup\{0,y\}}] m_{n,d}(dy)
	\label{eqn:er_conn_func_good_1}
	\end{align} 
The key observation to make is that for $x,y$ such that $||x-y|| > 2 \log(n)^{1/d}$, we have
\begin{align}
\mathbb{E}^{x,y}[ \mathbf{1}_{x \text{ is Flip-Bad in } G_n \cup \{x,y\}} \mathbf{1}_{y \text{ is Flip-Bad in }G_n \in \cup\{x,y\}}]  = \mathbb{E}^{x}[ \mathbf{1}_{x \text{ is Flip-Bad in } G_n \cup \{x\}}] \mathbb{E}^y[ \mathbf{1}_{y \text{ is Flip-Bad in }G_n  \cup\{y\}}] 
\label{eqn:er_conn_func_good2}
\end{align}

This follows since in an independently marked Poisson Process, events on disjoint sets are independent. Further more, from the symmetry in the torus, for all $x \in B_n$, we also have
\begin{align}
\mathbb{E}^{x}[ \mathbf{1}_{x \text{ is Flip-Bad in } G_n \cup \{x\}}] = \mathbb{E}^{0}[ \mathbf{1}_{0 \text{ is Flip-Bad in } G_n \cup \{0\}}].
\label{eqn:er_conn_func_good3}
\end{align}
Thus, we get from Equations (\ref{eqn:er_conn_func_good_1}),(\ref{eqn:er_conn_func_good2}) and (\ref{eqn:er_conn_func_good3}) that 
\begin{align}
\int_{y \in B_n}& \mathbb{E}^{0,y}[ \mathbf{1}_{0 \text{ is Flip-Bad in } G_n \cup \{0,y\}} \mathbf{1}_{y \text{ is Flip-Bad in }G_n \in \cup\{0,y\}}] m_{n,d}(dy)  \leq \nonumber\\&(n - 2^d \nu_d \log(n) ) \mathbb{E}^{0}[ \mathbf{1}_{0\text{ is Flip-Bad in } G_n \cup \{0\}}]^2 + 2^d \nu_d \log(n) \mathbb{E}^{0}[ \mathbf{1}_{0 \text{ is Flip-Bad in } G_n \cup \{0\}}].
\label{eqn:er_conn_func_good4}
\end{align}
Thus, since $\lambda \nu_d  \log(n)(1 - \sqrt{ab} - \sqrt{(1-a)(1-b)} > 1$, we have from Lemma \ref{lem:er_ch_divergence} that,\\ $n \mathbb{E}^{0}[ \mathbf{1}_{0 \text{ is Flip-Bad in } G_n \cup \{0\}}] = n^{\gamma}$ for some $\gamma > 0$.  Hence Equation (\ref{eqn:er_conn_func_good4}) implies Equation (\ref{eqn:er_corr_decay_def}) if $\lambda \nu_d  \log(n)(1 - \sqrt{ab} - \sqrt{(1-a)(1-b)} > 1$.

\end{proof}

\subsection{Upper Bound for Exact-Recovery - Proof of Theorem \ref{thm:ER_Upper_Bound}}
\label{sec:er_ub_proof}
In the present paper, we are only able to establish the presence of the phase transition by proving Theorem \ref{thm:ER_Upper_Bound}. We believe a `two-round' information theoretic argument can be employed to prove this result. A possible strategy is to first show that for any $0 \leq b < a \leq 1$, a `large' (i.e. all but $o(n)$) nodes will be correctly classified by the ML estimator with high probability. Further, if the parameters satisfied $\lambda \nu_d (1 - \sqrt{ab} - \sqrt{(1-a)(1-b)}) > 1$, then all nodes will be correctly classified with high probability. One can possibly make this efficient by means of `sample splitting' arguments of \cite{abbe_sandon_ch}. One can sub-sample the edges so that the graph is almost sparse so that a large fraction of the nodes can be correctly labelled by the ML estimator. Then, we can `clean-up', i.e. estimate a corrected community label estimate using the edges not used in the first round. This conjecture is also reminiscent of the `local to global' phenomena that occurs in many random graph models (\cite{abbe_exact},\cite{MNS_er},\cite{bollobas},\cite{penrose}), where an obvious local necessary condition also turns out to be sufficient. However, as a corollary to the GBG algorithm introduced above, we Theorem \ref{thm:ER_Upper_Bound} which establsihes that Exact-Recovery can be solved if the intensity $\lambda$ is sufficiently high.

\begin{proof} \emph{of Theorem \ref{thm:ER_Upper_Bound}} \\
	Notice that with $R = \log(n)^{1/d}/2d$, there are at-most $\lceil 4^d dn/\log(n) \rceil$ grid-cells in $B_n$. If we show that the chance that a grid cell is T-BAD is $n^{-1-\delta}$ for some $\delta > 0$, then by an union bound argument, we can assert that with probability at-least $1-n^{-\delta}$, all grid-cells will be T-GOOD. Furthermore, if all grid cells are T-GOOD, then by Proposition \ref{prop:consistent_partition}, all nodes in $G_n$ will be correctly partitioned by the algorithm.
	\\

	Set  $R = \log(n)^{1/d}/2d$ and $\epsilon \in (0,1)$ arbitrary as parameters of the {\ttfamily GBG} algorithm. From Proposition \ref{prop:pairwise_estimate}, we know that for any $z,z^{'} \in \mathbb{Z}^d$ such that $||z-z^{'}||_{\infty} =1$, the probability that the pairwise classifier makes an error is at-most  $e^{-c(d,a,b) \lambda \log(n) }$, where $c(d,a,b) > 0$ is a positive constant. From Lemma \ref{lem:T-Good-Prob}, it is clear, that the probability that a cell is T-BAD is at-most $n^{-c(d,a,b) \lambda } \log(n)^d C + n^{-c^{'}(\epsilon) \lambda n}$, where $c$ and $c^{'}$ are two strictly positive constants. Thus, by choosing $\lambda$ sufficiently high, we can ensure that the probability of a grid-cell being T-BAD is at-most $n^{-1-\delta}$ for some $\delta >0$.
\end{proof}

Here we also want to remark that using ideas from Section \ref{sec:unknown_conn_func}, the {\ttfamily GBG} algorithm can be implemented without knowledge of the parameters. In particular, one can use standard spectral or SDP algorithms to decide whether a cell is T-GOOD or not, as explained in Section \ref{sec:unknown_conn_func}. This can be done so since the sub-graph of $G_n$ restricted to nodes within the $1$ thickening of any grid cell $z$ is dsitributed as the SBM with random $\log(n)$ number of nodes and having a constant connection probability of either $a$ or $b$ among nodes of the same or opposite communities. Thus, one can employ standard spectral or SDP algorithms `off the shelf' to implement the algorithm without knowledge of $a$,$b$ or $\lambda$.

\section{Related Work}
\label{sec:related_work}

Community Detection on sparse graphs has mostly been studied on the SBM random graph model. The study of SBM has a long history in the different literatures of statistical physics (see \cite{decelle} and references therein), mathematics (for ex. \cite{inhomo_rg},\cite{MNS_Prob_Theory}) and computer science (for ex. \cite{planted_partition}, \cite{coja_oghalan},\cite{MNS_er}).  The reader should refer to the  survey  \cite{abbe_overview}  for further background and references on the SBM. The survey  \cite{moore_survey} gives a complete treatment of the SBM from a statistical physics view point. There has been renewed interest in the sparse regime of the SBM following the paper  \cite{decelle}, which made a number of striking conjectures on phase-transitions.  Subsequently, some of them have been established with the most notable and relevant achievements to ours being that of \cite{MNS_Prob_Theory},\cite{mns_possible}, \cite{massoulie_possibile} and \cite{bordenave}. These papers  prove that both Community Detection and the distinguishability problem for the two community sparse SBM undergo a phase-transition at the same point which they characterize explicitly. These results for the SBM motivates the investigation of the phase-transitions for Community Detection and distinguishability in our model.
However, the tools needed for our model are very different from those used to study the SBM. The key ideas for all of our  results come from different problems, mostly those studied in the theory of percolation \cite{bollobas} and stochastic geometry \cite{penrose_percolation}. Our algorithm is motivated by certain ideas that appeared in the Interacting Particle Systems literature (for ex. \cite{penrose_percolation} \cite{liggett},\cite{penrose_Particles},\cite{durrett_lecture}). The papers there developed re-normalization and percolation based ideas to study different particle systems arising in statistical mechanics, and our analysis bears certain similarity to that line of work. Our lower bound comes from identifying an \emph{easier} problem than Community Detection called Information Flow from Infinity, which is a new problem. The key idea to show this reduction comes from an ergodic argument applied on the spatial random graphs. To study the  impossibility of Information Flow from Infinity, we employ certain  `random-cluster method' and coupling arguments. Such methods are quite popular and have  proven to be extremely fruitful in other contexts for example to study mixing time of Ising Models (\cite{info_perc_1},\cite{infor_perc_2},\cite{info_perc_survey}). Similar coupling ideas have also appeared in other estimation contexts, notably the reconstruction on trees problem \cite{second_eigen}, where a lower bound was established using this method. 
\\

In the non-sparse regime, the work of \cite{abbe_exact} and \cite{abbe_sandon_ch} is most related to our methods for the non-sparse regime. Both papers studies the problem of Exact-Recovery in much greater detail in the generalized SBM, proving both lower bounds and efficient algorithms for exact recovery. In this paper, we use the lower bound framework developed in \cite{abbe_sandon_ch}, specifically, the large deviations characterization of hypothesis testing between Poisson random vectors, to provide a lower bound for our spatial random graph model.
\\

From a modeling perspective, the works of \cite{xu_gsbm} and \cite{gbm} which considers the Latent Space Models as a part of its model is the closest to our model. The paper of \cite{xu_gsbm} gives a spectral algorithm that is proven to work in the logarithmic degree regime. In particular, their methods do not work for the sparse regime which is one of the central topics discussed in the present paper. The algorithm provided in \cite{gbm} bears certain similarity to ours where nearby nodes' community membership are tested based on the common neighbors. However, it is only guaranteed to work in the logarithmic degree regime. Furthermore, no lower bounds are presented in \cite{gbm}, while the lower bound in \cite{xu_gsbm} is based on the idea that the graph is locally tree-like, which is not the case in our model.  Another related problem of group synchronization was introduced in \cite{abbe_group}, which among other things, also considered the question of how well can one identify the communities of two `far-away' nodes. This paper establishes certain phase-transitions for weak recovery that look similar to our lower bound, using different ideas from percolation and random walks on grids. Nonetheless, our algorithm is completely different from theirs and in-fact a straight forward adaptation of our algorithm to their setting can give an alternative proof of Theorem $3$ in \cite{abbe_group}. We also remark that certain results on the weak-recovery in our model appeared in a conference paper by the second and third authors \cite{soda}. However, the conference version does not carry certain proofs and does not discuss the Exact-Recovery problem in the logarithmic degree regime, which is the new technical contribution of the present paper compared to \cite{soda}.

\section{ Conclusions and Open Problems}

 In this paper, we introduced the problem of community detection in a spatial random graph where there are two equal sized communities. We studied this problem in both the sparse and non-sparse regime. Our main technical contributions in the sparse graph case are in identifying the problem of Information Flow from Infinity and connecting that with the Community Detection problem and giving a simple lower bound criterion. For developing the algorithm, we noticed that a spatial graph is sparse due to the fact that all interactions are dense, but localized which is starkly different from the reason why an Erd\H{o}s-R\'enyi graph is sparse. We leveraged this difference to propose an algorithm for community detection by borrowing further ideas from dependent site percolation processes. In the Exact-Recovery setting, we give a lower bound that we conjecture to be tight. However, this is just a first step and there are a plenty of open questions just concerning the model we introduced. 
 \\
 
 \emph{1) Is Weak-Recovery and Information Flow from Infinity equivalent ?} - In this paper, we proved that weak-recovery was harder than Information Flow from Infinity. However, our algorithm and its analysis showed that it can solve Community Detection whenever it can solve Information Flow from Infinity. Thus a natural question is whether these two problems undergo a phase-transition at the same point ? Moreover is there a relation between the optimal overlap achievable in Community Detection and the optimal success probability of estimating the community label of the origin in the Information Flow from Infinity problem ?
 \\
 
 \emph{2) Is the optimal overlap in weak-recovery monotone non-decreasing in $\lambda$ ?} -
 We saw in the proof of Proposition \ref{prop:monotone_repeat} that the optimal success probability of correctly labeling the origin in the Information Flow from Infinity problem is monotone in $\lambda$. However, for Community Detection, we only established that solvability is monotone and not the optimal overlap achievable.
 \\
 
 \emph{3) Can one resolve Conjecture \ref{conjecture_er} to identify the critical phase-transition point for Exact-Recovery} - This conjecture is reminiscent of the local to global phenomenon consistently observed in various random graph models (\cite{abbe_overview}, \cite{penrose_book},\cite{MNS_er} ). In these settings, an obvious local condition, i.e., there being no flip-bad node, also turns out to be sufficient for exact-recovery. Establishing such a result in our model will help us obtain a better understanding of our random graph model and also may aid in improved algorithms for practical situations.
 \\
 
 \emph{3) More than $2$ Communities} -
 In this paper, we focused exclusively on the case of two communities in the network, and an immediate question is that of $3$ or more communities. In the symmetric case where the connection function is $f_{in}^{(n)}(\cdot)$ within communities and $f_{out}^{(n)}(\cdot)$ across communities, a simple adaptation of our algorithm can give a sufficient condition in both the sparse and logarithmic degree regime, although will be sub-optimal. Our lower bound technique in the sparse regime can also be applied in the setting of many communities (see Theorem $2$ in \cite{abbe_group} for example ) thereby establishing the \emph{existence} of a non-trivial phase transition for any number of communities in the symmetric setting. But the open question is to identify examples similar to Proposition \ref{prop:lower_is_tight} where the phase transition for weak-recovery can be tight. A quest for such examples can possibly lead to better understanding of even the $2$ community case considered in this paper. Moreover, unified algorithmic techniques capable of handling non-symmetric case also is of interest since our algorithm does not generalize  in a straight forward way to the non-symmetric setting. As far as the logarithmic degree regime, our lower bound framework can easily extend using the same large-deviations result for Poisson hypothesis testing developed in \cite{abbe_sandon_ch}. However, an algorithm achieving this threshold in the logarithmic regime with multiple communities is not yet known.
 \\

 \emph{4) Characterization of the Phase-Transition for Weak-Recovery} -
 An obvious but harder question is whether one can characterize if not compute the exact phase-transition for either Community Detection or Information Flow from Infinity. We  show that  our lower bound is  capturing the phase-transition only in very specific cases and may not be tight in general due to corner cases similar to Proposition \ref{prop:lower_is_loose}. We also have no reason to believe that our algorithm is optimal in any sense. Thus, a structural characterization of the phase-transition is still far from being understood. 
 \\
 
 \emph{5) Computational Phase-Transition} - 
 Another aspect concerns the possible gaps between information versus computation thresholds. Is there a regime where Community Detection is solvable, but no polynomial (in $n$) time and space algorithms that operate on $(G_n,\phi_n)$ are known to exist ? 
 \\

 \emph{6) Estimating the Model Parameters} - How does one efficiently estimate  the connection functions $f_{in}(\cdot)$ and $f_{out}(\cdot)$ from the data of just the graph $G$ and the spatial locations $\phi$.

%


\section*{Acknowledgements}

EA was  supported by the grant of NSF CAREER Award CCF-1552131. FB and AS were supported by a grant of the Simons Foundations ($\#19
7982$ to
The University of Texas at Austin). The  authors thank  Gustavo de Veciana, Sanjay Shakkottai, Joe Neeman and Alex Dimakis for providing comments on an earlier  version of this work.

\bibliographystyle{plain}
\bibliography{../../../../../Notes_Email/SBM}

\begin{appendix}
	\section{Proof of Proposition \ref{prop:monotone}}
	\label{appendix-monotone}
	\begin{proof}
		Assume, for a given value of $\lambda$, there exists a community detection algorithm that achieves an overlap of $\gamma > 0$. Now, given any $\lambda^{'} > \lambda$, we will argue that we can achieve positive overlap. The proof of this follows from the basic thinning properties of the PPP. Given an instance of the problem with intensity $\lambda^{'}$, we will remove every node along with its incident edges independently with probability $1- \frac{\lambda}{\lambda^{'}}$. We assign a community label estimate of $+1$ to all the removed nodes. For the nodes that remain (which is then an instance of the problem of community detection with intensity $\lambda$), we achieve an overlap of $\gamma$ with  probability at-least $1-o_n(1)$, from the hypothesis that we can achieve positive overlap at intensity $\lambda$. Thus, from the independence of the  thinning procedure and the community labels and strong law of large numbers, the overlap achieved by this process on an instance of intensity $\lambda^{'}$ will be at-least $\frac{\lambda \gamma}{\lambda^{'}}$ with probability at-least $1-o_n(1)$. Thus, the problem of community detection solvability is monotone in $\lambda$ in the sparse case.
	\end{proof}
\section{Proof of Proposition \ref{prop:lower_is_loose}}
\label{appendix_proof}
The goal of this section is to establish Proposition \ref{prop:lower_is_loose} in which the equality in Theorem \ref{thm:main_lower_bound} cannot be achieved for any $\lambda > 0$. The functions $f_{in}(\cdot)$ and $f_{out}(\cdot)$ of Proposition \ref{prop:lower_is_loose} satisfy $\int_{r \geq 0}(f_{in}(r) - f_{out}(r))rdr = \infty$. Thus, the bound from Theorem \ref{thm:main_lower_bound} predicts that $\lim_{r \rightarrow \infty} \sup_{\tau_{r} \in \sigma({G}^{(r)} ,{\phi}^{(r)})} \mathbb{P}^{0}_{\phi}[\tau_{r} = Z_0] \leq 1$. However, we shall show that for every $\lambda > 0$, in our particular example, the probability of correctly estimating $Z_0$ is strictly smaller than $1$. In-fact, we will show something slightly stronger. We will show that given the labels of \emph{every node} other than the node at  origin, the probability of correctly estimating $Z_0$ is strictly less than $1$. The key tool to conclude about this example is the following classical result from \cite{acs_singular} which states that the measures induced by two Poisson Point Processes on $\mathbb{R}^d$ are either absolutely continuous with respect to each other or are mutually singular. More precisely, the result from \cite{acs_singular} after being adapted to our setting, states the following.
\begin{lemma}
	Let $\mu_1$ and $\mu_2$ be two measures on $\mathbb{R}^d$ such that $\mu_1 \sim \mu_2$ and $\mu_1(\mathbb{R}^d) = \mu_2(\mathbb{R}^d) = \infty$. Let $P_{\mu_1}$ and $P_{\mu_2}$ be the probability measures on the space of locally finite counting measures on $\mathbb{R}^d$ induced by two Poisson Point Processes having intensity measures $\mu_1$ and $\mu_2$ respectively. Then the following dichotomy holds.
	\begin{itemize}
		\item If $\int_{x \in \mathbb{R}^d} \left( 1- \sqrt{\frac{d \mu_1}{d \mu_2}} \right)^2 d \mu_2 < \infty$, then $P_{\mu_1} \sim P_{\mu_2}$
		\item If $\int_{x \in \mathbb{R}^d} \left( 1- \sqrt{\frac{d \mu_1}{d \mu_2}} \right)^2 d \mu_2 = \infty$, then $P_{\mu_1} \perp P_{\mu_2}$
	\end{itemize} 
	\label{lem:acs_singular}
\end{lemma}

The following lemma will conclude that the lower bound is strictly sub-optimal in this example.

\begin{lemma}
	For every $\lambda >0$ and $d \geq 2$, if $f_{in}(r) = \min \left( 1, \frac{1}{r} + \frac{1}{r^{d-1/4}}\right)$ and $f_{out}(r) = \min \left(1,\frac{1}{r} \right)$,
	$\sup_{\tau \in \sigma(G,\phi,\{Z_i\}_{i \geq 1})} \mathbb{P}^{0}[\tau = Z_0] < 1$.	
\end{lemma}
\begin{proof}
	From the independent thinning property of the Poisson Point Process, the origin partitions the process $\phi \setminus \{0\}$ into $4$ independent Poisson Processes, 
	\begin{enumerate}
		\item The point process $\phi_{+,e}$ which are the locations of those nodes of $G$ that have an edge to the origin and have for community label $Z_0$. From properties of $G$, the intensity measure of $\phi_{+,e}$ which we denote by $\mu_{in}(\cdot)$, has a density with respect to Lebesgue measure given by $f_{in}(||\cdot||)$.
		\item The point process $\phi_{-,e}$ which are the locations of those nodes of $G$ that have an edge to the origin and have for community label $-Z_0$.  The intensity measure of this point process which we denote as $\mu_{out}(\cdot)$ admits $f_{out}(||\cdot||)$ as its density with respect to Lebesgue measure.
		
		\item The point process $\phi_{+,n}$ which are the locations of those nodes of $G$ that do not have an edge to the origin and have for community label $Z_0$.  The intensity measure of this point process which we denote as $\tilde{\mu}_{in}(\cdot)$ admits $1-f_{in}(||\cdot||)$ as its density with respect to Lebesgue measure.
		
		\item The point process $\phi_{-,n}$ which are the locations of those nodes of $G$ that do not have an edge to the origin and have for community label $-Z_0$.  The intensity measure of this point process which we denote as $\tilde{\mu}_{out}(\cdot)$ admits $1-f_{out}(||\cdot||)$ as its density with respect to Lebesgue measure.

	\end{enumerate}
	
	Since, the process of graph neighbors of the origin and graph non-neighbors of the origin are independent, it suffices to conclude that both the optimal estimators for $Z_0$ based on the data of just the neighbors and  based on the data of just non-neighbors have a strictly positive chance of being wrong.
	In other words, it suffices to conclude that the measures induced on the set of locally finite counting measures on $\mathbb{R}^d$ by the process $\phi_{+,e}$ and $\phi_{-,e}$ are not mutually singular and the measures induced by $\phi_{+,n}$ and $\phi_{-,n}$ are not mutually singular either. 
	To do so, we will directly use Lemma \ref{lem:acs_singular} for our example.
	\\
	
	Notice that the chosen example satisfies the following.

	\begin{equation}
	\begin{aligned}
	\int_{x \in \mathbb{R}^2} \left( 1 - \sqrt{\frac{f_{out}(||x||)}{f_{in}(||x||)}}\right)^2dx < \infty \\
	\int_{x \in \mathbb{R}^2} \left( 1 - \sqrt{\frac{1-f_{in}(||x||)}{1-f_{out}(||x||)}}\right)^2dx < \infty
	\end{aligned}
	\label{eqn:example_conditions}
	\end{equation}
	
	
	Furthermore, notice that $\frac{d \mu_{out}}{d \mu_{in}}(\cdot) = \frac{f_{out}(\cdot)}{f_{in}(\cdot)}$ and $\frac{d \tilde{\mu}_{in}}{d \tilde{\mu}_{out}}(\cdot) = \frac{1-f_{in}(\cdot)}{1-f_{out}(\cdot)}$. Hence, from Equations (\ref{eqn:example_conditions}) and Lemma \ref{lem:acs_singular}, we see that for every $\lambda >0$, we have $P_{\mu_{in}} \sim P_{\mu_{out}}$ and $P_{\tilde{\mu}_{in}} \sim P_{\tilde{\mu}_{out}}$. Thus, $Z_0$ cannot be estimated perfectly without errors from the data and thus $\sup_{\tau \in \sigma(G,\phi,\{Z_i\}_{i \geq 1})} \mathbb{P}^{0}[\tau = Z_0] < 1$.
	
\end{proof}

\section{Proof of Proposition \ref{prop:consistent_partition}}
\label{appendix_partition_proof}

\begin{proof}
	We prove the proposition by induction. We will show that, given an unique and consistent partition of all nodes in cells $\mathcal{A}(z)$ upto $l_{\infty}$ distance of $n$ from $z$ and a certain number of cells at distance $n+1$, we can uniquely extend the consistent partition to one more cell in $\mathcal{A}(z)$ at distance $n+1$ from $z$. In other words, we construct the unique partition in a `breadth first manner' from $z$, by representing the distance on $\mathbb{Z}^d$ using the $l_{\infty}$ distance. As a corollary, this proof technique  establishes that the arbitrary sequence in Line $5$ of Algorithm \ref{alg:main-routine} will contain all the nodes of the component and will return the unique consistent partition.
	\\
	
	For the base-case, since cell $z$ is A-Good, from Line $6$ of Algorithm \ref{alg:Is_Good}, we can uniquely partition all points in cell $z$ and its $1$ step neighbors. The existence of a consistent partition can be argued as follows. Pick an arbitrary $X_i \in \phi \cap Q_z$ which we know by definition to be non-empty and label it $+1$. Now, for all points $X_j \in \cup_{z^{'} : ||z-z^{'}||_{\infty}} (\phi \cap Q_{z^{'}})$, label $X_j$ with the value returned by  Algorithm \ref{alg:Pairwise} run on the input $(i,j,G)$. Thus, we have produced a partition of all the points in the one-step neighbor of $z$ in $\mathcal{A}(z)$  We will first argue that the partition we produced above is consistent, i.e. satisfies the conditions of the present proposition. Indeed, assume to the contrary, that this partition violates the statement of the current proposition. Assume there exists two points $X_k$ and $X_l$ such that they are in opposite sets of the partition with the partitioning procedure stated above whereas Algorithm \ref{alg:Pairwise} run on $(k,l,G)$, returns a $+1$. This  implies that the product of the outputs of Algorithm \ref{alg:Pairwise} run on input $(i,k,G)$ with $(k,l,G)$ and $(i,l,G)$ is $-1$, thereby violating the fact that cell $z$ is A-Good. We can similarly, argue the absence of two points $X_k$ and $X_l$ such that they are in the same partition, but Algorithm \ref{alg:Pairwise} returns a $-1$. Now, to argue uniqueness, assume that on the contrary,  there exists another consistent partition. This implies that there must exist two points $X_i$ and $X_j$ which are in the same set in one partition and in different sets in the other partition. Thus, clearly one of the partitions must violate the two requirements of this proposition, since Algorithm \ref{alg:Pairwise} run on $(i,j,G)$ will produce just one output, thereby invalidating the consistency of at-least one the two partitions. 
	\\
	
	Now, we show the induction step. Assume there is an unique consistent partition of all nodes in $\phi \cap Q_{\mathcal{A}(z)}$ that are in cells of $l_{\infty}$ distance of at-most $n$ from $z$ and some cells at a distance of $n+1$ from $z$. Let $z^{'} \in \mathcal{A}(z)$ be such that $||z-z^{'}||_{\infty} = n+1$ and assume that $z^{'}$ has not yet been partitioned. Pick any $z^{''} \in \mathcal{A}(z)$ such that $||z - z^{''}||_{\infty} \leq n+1$ , $||z^{'} - z^{''}||_{\infty} = 1$, such that cell $z^{''}$ has already been partitioned. Note from the fact that $\mathcal{A}(z)$ is a connected subset of $\mathbb{Z}^d$, one can always find such a $z^{''}$. Pick $X_i \in \phi \cap Q_{z^{''}}$ arbitrarily. This can be done since we know that $\phi \cap Q_{z^{''}}$ is non-empty by definition of it being in $\mathcal{A}(z)$. To extend the partition, for every point $X_k \in \phi \cap Q_{z^{'}}$, run Algorithm \ref{alg:Pairwise} on the input $(i,k,G)$. Place $X_k$ in the same partition as $X_i$ if the algorithm returned $+1$; else place $X_k$ in the opposite partition of $X_i$. We will now conclude by showing that this extension still respects consistency and is unique. To argue uniqueness and consistency, it suffices to show that for all $X_k \in \phi \cap Q_{z^{'}}$ arbitrary no matter what $X_i \in Q_{z^{''}}$ we pick, the class to which $X_k$  belongs is the same. Let $z^{'''}$ be arbitrary and such that $||z^{'}-z^{'''}||_{\infty} = 1$ cell and $z^{'''}$ is already partitioned. Let $i^{'} \in \phi \cap Q_{z^{'''}}$ be arbitrary. Now, the points $X_{i^{'}},X_k,X_i$ are all within a $1$ thickening of cell $z^{'}$ which we know is A-Good. Thus, from Line $6$ of Algorithm \ref{alg:Is_Good}, the product of Pairwise-Classify on inputs $(i^{'},k),(k,i),(i,i^{'})$ is $1$. By the induction hypothesis, there is a unique relation between $i$ and $i^{'}$ (they are either in the same or different classes of the unique partition). Therefore, no matter the reference point, the class of a point $X_k \in \phi \cap Q_{z^{'}}$ is unique and consistent.
	
\end{proof}

\section{Proof of Lemma \ref{lem:two_possible}}
\label{appendix_proof_2possible}

\begin{proof}
	Denote by $T_I(j)$  be the breadth first spanning tree of $I$ constructed with $j$ as the root. Thus in the tree $T_I(j)$, and for all $k \in V_I(j)$, there is exactly one path from $j$ to $k$ in the tree $T_I(j)$. 
	\\
	
	The existence of two labelings is not so difficult since we know there exists one underlying true labeling which generated the data $G,I$. But since, the model is symmetric, the complement of the true labels will also be consistent in the sense of Lemma \ref{lem:I_and_G}. Thus, there are at least two labeling consistent with the observed data $G,I$. These two labels of $V_I(j)$ can be constructed explicitly which we do in the next paragraph. We  then show, that there are no other that can be consistent in the sense of Lemma \ref{lem:I_and_G}, which will conclude the proof.
	\\

	To construct the two possible labelings, first assume that $Z_j = +1$. Now conditionally on this and $G$, each neighbor of $j$ in $T_I(j)$ will have exactly one possible community label estimate that is consistent in the sense of Lemma \ref{lem:I_and_G}. Now, by induction, we can construct the labels of $V_I(j)$. Assume, that conditionally on $Z_j = +1$ and $G$, we have a unique set of labels for all vertices in $T_I(j)$ at graph distance of less than or equal to $k$. Let $u \in T_I(j)$ be an arbitrary vertex such that it is at graph distance $k+1$ from $j$ in $T_I(j)$. Since $T_I(j)$ is a tree, there is a unique vertex $v$ in $V_I(j)$ such that $v \sim_{T_I(j)} u$ and $v$ is at a distance of $k$ from $j$. Thus, conditionally on $Z_j = +1$, $Z_{v}$ is a fixed community label due to the induction hypothesis. Since $Z_{v}$ is fixed, then there is a unique label for $Z_u$ that will be consistent in the sense of Lemma \ref{lem:I_and_G}. Since $u$ was arbitrary, we can uniquely assign a community label to all vertices at graph distance of $k+1$ from $j$ in $T_I(j)$. Hence, by induction, conditionally on $Z_j = +1$, there is a unique community estimate for all vertices in $V_I(j)$. Similarly, if we assumed $Z_j = -1$, we will find another unique labeling for the vertices in $V_i(j)$ which will be the complement of the unique labeling obtained by assuming $Z_j = +1$. This, gives us that there exist at-least two labelings of $V_i(j)$ that are complements of each other and consistent with the observed data $G$ and $I$ in the sense of Lemma \ref{lem:I_and_G}. 
	\\

	To see that there can be no other possibilities, we argue by contradiction. Assume there are two labelings and a vertex $k$ such that in one of the labelings $Z_j = +1, Z_k = +1$ and in the other $Z_j = +1, Z_k = -1$. It is  clear that at-most one of the above labelings will be consistent in the tree $T_I(j)$ in the sense of Lemma \ref{lem:I_and_G}. This establishes that the  two sequences we constructed in the previous paragraph which are complements of each other are the only two possible sequences that are consistent in the sense of Lemma \ref{lem:I_and_G}. 
	
\end{proof}

\section{Definition of PPP}
\label{appendix_PPP}

A homogeneous PPP of intensity $\lambda$ on $\mathbb{R}^d$ is a random process $\phi := \{X_1,X_2, \cdots\}$ with each $X_i \in \mathbb{R}^d$ such that the following two holds
\begin{itemize}
	\item For every bounded Borel set $B$, the cardinality of the set $\phi(B):= |\{i \in \mathbb{N}: X_i \in B\}|$ is a Poisson random variable with mean $\lambda |B|$ where $|B|$ is the volume (Lebesgue measure) of the set $B$.
	\item For any $k \in \mathbb{N}$ and any \emph{disjoint} bounded Borel measurable sets $B_1, \cdots , B_k$, the random variables $\phi(B_1), \cdots \phi(B_k)$ are mutually independent. 
\end{itemize} 

\section{Palm Measure on the Torus}
\label{appendix_palm}
In this section, we recap the basic properties of the Poisson Point Process on the torus. As the torus is a locally compact unimodular topological group, the definition of Palm measures and the moment measure equations of PPP on a torus follows directly from the theory described in \cite{last_palm}. We reproduce the key results needed for our paper here. Fix a $n \in \mathbb{N}$ and consider the $d$ dimensional torus on the set $B_n := \left[ -\frac{n^{1/d}}{2}, \frac{n^{1/d}}{2}\right]^d$ equipped with the Haar measure denoted as $m_{n,d}(\cdot)$ which is invariant under translations on the torus. Let $(\Omega, \mathcal{F},\mathbb{P})$ be a probability space on which we have a stationary independently marked  PPP on the torus $B_n$ of intensity $\lambda > 0$ with marks in an arbitrary Polish space $\mathcal{K}$. Denote by the atoms of this point process by $\phi_n := \{X_1,\cdots,X_{N_n}\}$ enumerated in an arbitrary manner and the corresponding marks as $\{K_1,\cdots,K_{N_n}\}$. From the definition of the PPP,  $N_n$ is a Poisson random variable of mean $\lambda n$ independent of everything else and conditional on $N_n$, $(X_i)_{i=1}^{N_n}$ are  i.i.d. random variables that are uniformly distributed in the set $B_n$.  Conditionally on $N_n$ and $(X_i)_{i=1}^{N_n}$, the sequence $(K_i)_{i = 1}^{N_n}$ are i.i.d. In this framework, for any $k \in \mathbb{N}$ and $x_1,\cdots, x_k \in B_n$, the Palm measure $\mathbb{P}^{x_1,\cdots,x_k}$ corresponds to adding fictitious atoms at locations $\{x_1,\cdots,x_k\}$ and equipping them with independent marks having the same law as $K_1$. In other words, thanks to Slivnyak's theorem (\cite{last_palm}), the atoms of the PPP under the Palm measure $\mathbb{P}^{x_1,\cdots,x_k}$ is $\phi_n \cup \{x_1,\cdots x_k\}$ where $\phi_n$ is the law of the point process under $\mathbb{P}$, with the marks of these additional points having the same distribution as that of $K_1$ and independent of everything else. For this framework,  for any function $f(\cdot) : B_n \rightarrow \mathbb{R}_{+}$  such that $\int_{x \in B_n} f(x) m_{n,d}(dx) < \infty$ we have the following version of Campbell's theorem (\cite{last_palm}) -
\begin{equation}
\mathbb{E}[\sum_{x \in \phi_n} f(x)] = \lambda \int_{x \in B_n} \mathbb{E}^x[ f(x) ]m_{d,n}(dx).
\label{eqn:appendix_palm_campbell}
\end{equation}
More generally, we have the following $k$-th order moment measure expansions of the Poisson process. Let $f(\cdot): B_n^k \rightarrow \mathbb{R}_{+}$ be such that $\int_{x_1, \cdots,x_k} f(x_1,\cdots, x_k) \prod_{i=1}^{k}m_{n,d}(dx_i) < \infty$. Then, we have
\begin{equation}
\mathbb{E}[\sum_{\substack{x_1,\cdots, x_k \in \phi_n \\ \neq}} f(x_1,\cdots,x_k)] = \lambda^k \int_{x_1,\cdots, x_k \in B_n} \mathbb{E}^{x_1,\cdots,x_k}[f(x_1,\cdots,x_k) ]\prod_{i=1}^{k} m_{n,d}(dx_i).
\label{eqn:appendix_palm_moment}
\end{equation}
We use Equations (\ref{eqn:appendix_palm_campbell}) and (\ref{eqn:appendix_palm_moment}) in Section \ref{sec:exact_recovery} to compute the first and second moments of the number of Flip-Bad nodes in $G_n$.
\end{appendix}

\end{document}